\documentclass[journal ]{new-aiaa}
\usepackage[utf8]{inputenc}
\usepackage{textcomp}
\usepackage{amsfonts}
\usepackage{tikz}
\usepackage{subcaption}  
\usepackage{cleveref}
\usepackage{amsmath}
\usepackage{diagbox}   

\usepackage{lettrine}

\usepackage{amsthm}
\newtheorem{theorem}{Theorem}
\newtheorem{proposition}{Proposition}
\newtheorem{lemma}{Lemma}
\newtheorem{corollary}{Corollary}
\newtheorem{problem}{Problem}
\newtheorem{definition}{Definition}
\newtheorem{remark}{Remark}
\newtheorem{assumption}{Assumption}

\usepackage{tikz}
\usetikzlibrary{arrows.meta, shapes.geometric, calc, decorations.pathreplacing}
\definecolor{nomblue}{RGB}{31,90,166}
\definecolor{bkgreen}{RGB}{27,133,84}
\definecolor{failred}{RGB}{190,36,51}
\definecolor{orbgray}{RGB}{110,110,110}
\usepackage{tikz-3dplot}
\usetikzlibrary{3d,arrows.meta,calc}
\definecolor{cNom}{RGB}{28,63,120}     
\definecolor{cPoint}{RGB}{13,124,102}  
\definecolor{cMag}{RGB}{214,118,15}    
\definecolor{cExec}{RGB}{166,28,44}    
\definecolor{cErr}{RGB}{110,70,140}    

\crefname{lemma}{Lemma}{Lemmas}
\crefname{corollary}{Corollary}{Corollaries}
\crefname{problem}{Problem}{Problems}
\crefname{definition}{Definition}{Definitions}
\crefname{remark}{Remark}{Remarks}
\crefname{assumption}{Assumption}{Assumptions}
\crefname{section}{section}{sections}
\Crefname{section}{Section}{Sections}
\Crefname{assumption}{Assumption}{Assumptions}
\usepackage{dsfont}
\usepackage{graphicx}
\usepackage{siunitx}
\usepackage{longtable,tabularx}
\usepackage{mathtools}
\title{Robust Halo Orbit Stationkeeping Under Missed Thrust Events and Maneuver Execution Errors}

\author{Shaurya Shrivastava\footnote{Ph.D. Student, School of Aeronautics and Astronautics, Purdue University, West Lafayette, IN, 47907; shriva15@purdue.edu (Corresponding author)} and Kenshiro Oguri\footnote{Assistant Professor, School of Aeronautics and Astronautics, Purdue University, West Lafayette, IN, 47907. Senior Member AIAA}}
\affil{Purdue University, West Lafayette, IN, 47907}

\begin{document}

\maketitle

\begin{abstract}
Spacecraft stationkeeping is subject to various operational uncertainties such as missed thrust events, maneuver execution errors, and dynamics modeling errors. Missed thrust events are the most disruptive of these, and current practices mitigate that risk through Monte Carlo-driven heuristics and conservative margins that offer no formal safety guarantees and do not scale to multiple missed thrust events. This paper models these unplanned thruster outages through the Markov Jump System framework, and represents maneuver execution errors and dynamics linearization errors as bounded multiplicative noise. The resulting closed-loop system is a linear time-varying periodic Markov jump system with multiplicative noise, for which Lyapunov-like mean square stability and periodic set invariance conditions are derived. A linear matrix inequality-based optimal controller synthesis framework then designs a periodic, mode-dependent mean-square stable feedback policy that respects path constraints with probability one. The framework is demonstrated on the Earth-Moon 9:2 near-rectilinear halo orbit (NRHO), under both impulsive and low-thrust propulsion, and on a Saturn-Enceladus $L_1$ NRHO using impulsive propulsion. Nonlinear Monte Carlo simulations show that the spacecraft recovers from an arbitrary number of missed thrust events while respecting path constraints, and that the multiplicative noise model introduces additional robustness to maneuver execution error and linearization error.
\end{abstract}

\section{Introduction}

Many long-duration space missions comprise an extended orbiting phase, during which the
spacecraft is required to remain in a prescribed orbit. Such orbits are chosen to
provide favorable science observation geometry and operational benefits such as continuous
communications coverage and long-term stability. However, unmodeled perturbations, orbit
insertion errors, and other uncertainties cause the spacecraft to deviate from the baseline
orbit. The spacecraft must therefore execute trajectory correction maneuvers (TCMs) to
actively remain in the baseline orbit, which constitutes the \textit{stationkeeping} problem
\cite{gomezStationKeepingStrategiesTranslunar1998, foltaEarthMoonLibration2014,oguriEQUULEUSTrajectoryDesign2020}. Periodic
orbits near the collinear Lagrange points $L_1$ and $L_2$ have emerged as popular candidates
for several future space missions~\cite{davisORBITMAINTENANCEBURN2022,
mackenzieEnceladusOrbilanderMission2021}, but the chaotic dynamics in the vicinity of the
libration points demand a robust stationkeeping strategy. Since the orbit maintenance phase
comprises the bulk of such missions, it is critical to design a stationkeeping strategy that
is both fuel-efficient and robust to the various uncertainties common in mission operations.

Numerous stationkeeping strategies have been developed for periodic orbits. Classical
approaches include Floquet mode control, which cancels the unstable component of the
deviation using linearized dynamics~\cite{gomezStationKeepingStrategiesTranslunar1998}, and
target-point or $x$-axis crossing control, which targets desired conditions, such as a
perpendicular crossing, at a downstream reference crossing
\cite{foltaEarthMoonLibration2014,howellStationkeepingMethodLibration1993,oguriEQUULEUSTrajectoryDesign2020}. Among the libration point orbit families, the
near-rectilinear halo orbits (NRHOs), a subset of the halo orbit family
\cite{connorhowellThreedimensionalPeriodicHalo1984}, have attracted sustained attention for
their favorable stability properties, their synodic resonance, which allows eclipses to be
avoided, and their coverage of both primaries. The Earth-Moon 9:2 southern $L_2$ NRHO is the
baseline orbit for the Gateway mission~\cite{davisORBITMAINTENANCEBURN2022}, while NRHOs in
the Saturn-Enceladus system are of interest for the icy-moon science missions highlighted in
the most recent decadal survey~\cite{mackenzieEnceladusOrbilanderMission2021,
sidhoumLowthrustTrajectoryDesign2024a, booneHaloOrbitConfigurations2026}. This has motivated a
body of NRHO-specific stationkeeping work~\cite{davisORBITMAINTENANCEBURN2022,
muralidharanControlStrategyLongTerm}, alongside more recent efforts applying predictive and
optimization-based control~\cite{shimaneRevolutionSpacedOutputFeedbackModel2026},
reinforcement learning~\cite{lafargeAdaptiveClosedloopManeuver2023}, and uncertainty-aware
chance-constrained formulations~\cite{aletiChanceConstrainedOutputFeedbackControl2023,
kumagaiRobustNRHOStationkeeping2023, oguriChanceConstrainedControlSafe2024} to
stationkeeping. Contingency-aware and fail-safe schemes have also been studied, covering safe
departures, abort-aware control, and recovery following an inadvertent departure
\cite{vegaContingencyAwareStationKeepingControl2024,
marsillachFailSafeSpacecraftRendezvous2021,
zimovan-spreenRecoveryTraejctoriesInadvertent2021}. Current flight practice, however, still
relies on heuristics and extensive ground-in-the-loop simulation to achieve robustness to
operational uncertainties.

As the number of missions grows and operations move toward greater onboard autonomy, it is
crucial to design stationkeeping strategies that ensure safety under the various
uncertainties that can jeopardize mission success~\cite{starekSpacecraftAutonomyChallenges2016}.
Among the most critical of these are missed thrust events (MTEs), which are unplanned thruster outages
during which the spacecraft is unable to operate nominally and execute planned burns. MTEs
are common in deep space missions and may be caused by onboard power or
thermal constraints, or radiation-induced shutdowns. The Dawn mission, for instance, experienced
more than fifteen MTEs~\cite{grebowDawnSafingApproach2016, mccartyMissedThrustAnalysis2020},
which forced significant revisions to its original mission plan. Current mission practices
for managing MTE risk center on Monte Carlo-driven heuristics and conservative margins, which
provide no safety guarantees and cannot account for multiple outages over a trajectory.
MTE-robust control design has received significant attention in interplanetary trajectory
design~\cite{laipertAutomatedMissedThrustPropellant2015,
rubinsztejnNeuralNetworkOptimal2020, rubinsztejnDesigningTrajectoriesResilient2021,
venigallaMultiObjectiveLowThrustTrajectory2022}, multi-body trajectory design~\cite{shrivastavaMASTOMissedthrustAware2026,sinhaInitialGuessGeneration2025a,sekineRobustTrajectoryDesign2026,sekineRobustTrajectoryOptimization2026} and rendezvous
\cite{shrivastavaFaultTolerantFeedbackControl2023, nathersonReachabililtybasedMissedThrust2026,nemirovskyTrajectoryDesignControl2026},
but has received limited attention~\cite{shrivastavaMarkovJumpApproachRobust2025} in
stationkeeping control design.

MTEs represent an extreme form of a broad class of maneuver execution errors. Even without
MTEs, most executed maneuvers are imperfect and comprise a combination of pointing and
magnitude errors~\cite{gatesSimplifiedModelMidcourse,oguriRobustSpacecraftGuidance2021,oguriChanceConstrainedControlSafe2024, kumagaiRobustCislunarLowThrust}. Such errors can
accumulate to become significant, leading to large deviations from the planned trajectory.
There is therefore a need to model these errors within the control design process and to
ensure robustness against them. Furthermore, most stationkeeping strategies are designed
under simplified force model assumptions, which can differ from the ``true'' dynamics. Such
unmodeled perturbations can result in deviations from the planned nominal trajectory that can be detrimental to
long-term stationkeeping. To counter this, the control design framework must also account for
the uncertainty introduced by simplified dynamical models.

Various uncertainty models have been proposed to incorporate these errors into the control
design process. For instance, the Gates model has been used to represent maneuver execution
error, and has previously been incorporated into sequential covariance steering frameworks
\cite{kumagaiRobustCislunarLowThrust,oguriConvexBlockCholeskyApproach2026}. Dynamics error has likewise been considered in the
control design process, with previous works focusing on controllers that minimize
linearization error~\cite{qiConvexBoundNonlinear} or on modeling it as additive noise
\cite{kumagaiRobustCislunarLowThrust, oguriConvexBlockCholeskyApproach2026}. For
MTE-induced uncertainty, previous flight data has found MTE occurrence to be well modeled by
an exponential distribution~\cite{imkenModelingSpacecraftSafe2018}, which motivates our
approach of leveraging the Markov jump system (MJS) framework to model missed thrust events. The MJS framework allows us to model systems where the dynamics also depend on an independent parameter called ``mode,'' which is governed by a Markov chain. The memoryless property of the exponential distribution makes the MJS a natural modeling
choice, and has previously been applied in several astrodynamics contexts such as rendezvous~\cite{shrivastavaFaultTolerantFeedbackControl2023,fuModelReferenceTracking2017}, stationkeeping~\cite{shrivastavaMarkovJumpApproachRobust2025}, and trajectory optimization~\cite{shrivastavaMASTOMissedthrustAware2026}. Note that unlike other uncertainties, which can be simplified to consider Gaussian dynamics, MTE-induced uncertainty introduces an exponential
number of Gaussians, rendering traditional Gaussian assumption-based stochastic optimal
approaches inapplicable for MTE-robust control design.

The MJS framework has been extensively studied thanks to its practical applicability to several domains such as network control~\cite{suEventtriggeredSlidingMode2021}, fault-tolerant control~\cite{liRobustSimultaneousFault2019} and economics~\cite{dovalRecedingHorizonControl1999}. In particular, its linear subclass, the Markov jump linear system (MJLS), is the most extensively studied case, with a mature body of results on stability~\cite{jiControllabilityObservabilityDiscretetime1988,costaStabilityResultsDiscreteTime1993}, optimal control~\cite{costaConstrainedQuadraticState1999}, LMI-based controller synthesis~\cite{
elghaouiRobustStatefeedbackStabilization1996}, and robust formulations~\cite{desouzaRobustStabilityStabilization2006}. An excellent reference on the classical results can be found in~\cite{costaDiscreteTimeMarkovJump2005}. Multiplicative noise has also been studied for MJLS for $H_2$-control~\cite{maH2ControlDiscretetime2013}, variance-constrained control~\cite{barbieriOptimalControlConstrained2018} and, more recently, for covariance steering
\cite{wangCovarianceSteeringDiscreteTime2026}. 

For stationkeeping applications, we leverage the periodic system framework to synthesize controllers in a tractable manner. While periodic systems have received significant attention in the single-mode case~\cite{bittantiExtendedPeriodicLyapunov1985, bittantiPeriodicSystemsFiltering2008}, with a rich body of work on their stability~\cite{desouzaLMIApproachStabilization2000} and LMI-based controller synthesis frameworks~\cite{bohmStabilityPeriodicallyTimevarying2012,athanasopoulosStabilityStabilizationPeriodic2014}, its MJLS counterpart, called the Linear Time-Varying Periodic MJS (LTVPM), has received limited attention, largely confined to $H_2$ optimal control~\cite{maStochasticH2Optimal2012} and to structural properties such as observability and detectability~\cite{houSpectralTestsObservability2016}. Our previous work~\cite{shrivastavaRobustControllerSynthesis2024a} analyzes stability and LMI-based control design frameworks for such systems, but did not consider multiplicative noise. Reference~\cite{shrivastavaMarkovJumpApproachRobust2025} leverages this framework for MTE-robust stationkeeping maneuver planning, but it does not consider additional sources of uncertainty given the noiseless MJLS formulation. 

Therefore, in this work, we extend our previous work to a Markov jump-based control design framework with
multiplicative noise that is robust to MTEs, maneuver execution errors, and linearization errors.
The contributions of this work are threefold. First, we propose novel models for various uncertainties, such as multiplicative noise-based model for both maneuver
execution error and linearization error, and a three-mode MJLS modeling strategy for contingency TCM burns in impulsive propulsion case. Second, we develop
an LMI-based controller synthesis framework for LTVPM under
multiplicative noise that generates mean square stable policies respecting path constraints
with probability one. Third, we apply the developed controller synthesis framework to optimal stationkeeping
maneuver planning that is robust to MTEs, maneuver execution errors, initial dispersion errors, and
linearization errors. We demonstrate the robustness of the framework through NRHO
stationkeeping scenarios in the Earth-Moon and Saturn-Enceladus systems. 

The remainder of this paper is organized as follows. \Cref{sec:problem} formally defines the
stationkeeping problem as a stochastic optimal control problem with details on dynamics and maneuver execution error model. \Cref{sec:modeling} develops the MJLS modeling strategy, which enables us to model the original problem as an LTVPM optimal control problem. \Cref{sec:synthesis} presents the
stability theory for LTVPMs and an LMI-based controller synthesis framework to solve the original problem. \Cref{sec:results} discusses numerical results for two separate NRHOs in the Earth-Moon and
Saturn-Enceladus systems, and \Cref{sec:conclusion} concludes the paper.

\emph{Notation:} $\mathbb{R}$, $\mathbb{N}$, $\mathbb{Z}_{\geq 0}$ and $\mathbb{Z}_{a:b}$ denote the set of real numbers, natural numbers, non-negative integers and set of integers from $a$ to $b$, respectively. We define a linear space $\mathbb{M}^N$ with $N\in\mathbb{N}$ comprising a finite set of matrices $V\triangleq (V(1),\ldots V(N))$, where each $V(i)\in\mathbb{R}^{n\times n} $, $n\in\mathbb{N} \; \forall i\in\mathbb{Z}_{1:N}$.  $\{ V_k \}_{k\in\mathbb{Z}_{\geq 0}}$ is a sequence of such finite set of matrices, and $V_k\in\mathbb{M}^N$ refers to the $k^{\mathrm{th}}$ element in the sequence. Therefore, $V_k(i)$ refers to the $i^{\mathrm{th}}$ matrix from the $k^{\mathrm{th}}$ element of the sequence. Similarly, $\mathcal{O}_k^i(\cdot)$ refers to the $i^{\mathrm{th}}$ operator from the $k^{\mathrm{th}}$ element of the sequence $\{\mathcal{O}_k(\cdot)\}_{k\in\mathbb{Z}_{\geq 0}}$ whose each element is a finite set of operators $\mathcal{O}(\cdot)\triangleq(\mathcal{O}^1(\cdot),\ldots\mathcal{O}^N(\cdot))$. For a matrix $X$, $X^\top$ denotes its transpose. For a set of matrices $V\in\mathbb{M}^N$, $\mathrm{blkdiag}(V(i))$ denotes a block diagonal matrix comprising each matrix $V(i)$ of the finite set. If an operator is defined for a finite set of matrices $V(i)\in\mathbb{M}^N$, it represents an operator on the matrix $\mathrm{blkdiag}(V(i))$. For a vector $\mathbf{x}\in\mathbb{R}^n$, $\|\mathbf{x}\|$ indicates the $l_2$-norm of the vector. $\ast$ is used to denote symmetric elements in a hermitian matrix. $I_n$ and $0_n$ denote the identity and zero matrix of size $n\times n$ respectively. $\mathds{1}_{\{A\}}$ is the indicator function for the statement $A$: it is unity when A is true else 0. $\otimes$ denotes the Kronecker product. 

\section{Problem Statement}\label{sec:problem}

\subsection{Dynamics and Equations of Motion}

\subsubsection{Circular Restricted Three-Body Problem}

The Circular Restricted Three-body Problem (CR3BP) models the motion of a massless particle such as a spacecraft under the gravitational influence of two primaries, which are in a circular orbit around their common center of mass, called the barycenter. This is a commonly used model for preliminary mission design and assumes that the massless body does not affect the motion of the primaries. For the spacecraft in a rotating frame of reference about the barycenter of the primaries, the equations of motion are given by~\cite{spreenTrajectoryDesignTargeting2021}: 
\begin{equation}
    \label{eom}
        \begin{aligned}
    \ddot{x} &= x+ 2\dot{y}  -\frac{(1-\mu)(x+\mu)}{d_1^3} - \frac{\mu(x-1+\mu)}{d_2^3} \\
    \ddot{y} &=- 2\dot{x} +y  -\frac{(1-\mu)y}{d_1^3} - \frac{\mu y}{d_2^3} \\
    \ddot{z} &= -\frac{(1-\mu)z}{d_1^3} - \frac{\mu z}{d_2^3}
    \end{aligned}
\end{equation}
where  $d_1=\sqrt{(x + \mu)^2 + y^2 + z^2}$ and $d_2 = \sqrt{(x - 1 + \mu)^2 + y^2 + z^2}$ are the distances from the spacecraft to the two primary bodies, and $ \mu = \frac{m_2}{m_1 + m_2}$ is the mass ratio of the system. 

We define the corresponding state vector as $\mathbf{x}=[x, y, z, \dot{x}, \dot{y}, \dot{z}]^\top\in\mathbb{R}^{n_x}$ with $n_x=6$. The corresponding state space model with control input $\mathbf{u}\in\mathbb{R}^{n_u}$, $n_u=3$, is given by:
\begin{equation}\label{statespace}
    \dot{\mathbf{x}} = \mathbf{f}(\mathbf{x},\mathbf{u})=\mathbf{f}_0(\mathbf{x})+{B} \mathbf{u}(t)
\end{equation}
where $\mathbf{f}_0(\mathbf{x})$ is the corresponding nonlinear drift dynamics, and $B$ is the control input matrix, given by:
\begin{equation*}
    \renewcommand{\arraystretch}{0.9}
    \mathbf{f}_0(\mathbf{x}) = \begin{bmatrix}
    \dot{x} \\
    \dot{y} \\
    \dot{z} \\
    x+2\dot{y}- \tfrac{(1-\mu)(x+\mu)}{d_1^3} - \tfrac{\mu(x-1+\mu)}{d_2^3} \\
    -2\dot{x} + y - \tfrac{(1-\mu)y}{d_1^3} - \tfrac{\mu y}{d_2^3} \\
    -\tfrac{(1-\mu)z}{d_1^3} - \tfrac{\mu z}{d_2^3}
\end{bmatrix}, \quad B = \begin{bmatrix}
    \mathbf{0}_{3\times3}\\
    I_3
\end{bmatrix}
\end{equation*}

The control input $\mathbf{u}(t)$ can be used to model both impulsive and continuous propulsion. The control input is defined by an input sequence $\{\mathbf{u}_k\}_{k \in \mathbb{Z}_{\geq0}}$, and is given below:
\begin{equation}
\mathbf{u}(t)= \begin{cases}\sum_{k=0}^{T-1} \mathbf{u}_k \delta\left(t-t_k\right) & \text { (impulsive) } \\ \mathbf{u}_k, \quad \forall t \in\left[t_k, t_{k+1}\right) & \text { (continuous) }\end{cases}
\end{equation}
Here, $\delta(\cdot)$ represents the Dirac delta function and zero-order hold (ZOH) is assumed for continuous propulsion.

\subsubsection{Near-Rectilinear Halo Orbits}

In the CR3BP, there are five equilibrium points in the rotating frame of reference, which are called the libration points. Each of these libration points also has several corresponding families of periodic orbits. In particular, the halo orbits receive large attention thanks to their three-dimensional structure and proximity to the secondary body, which allows for good coverage of the secondary body. The near-rectilinear halo orbits (NRHOs), which are a subset of the halo orbit family, have since attracted significant attention for their stability properties and their synodic time-period resonance, which allows them to avoid eclipses.

The stability properties of periodic orbits can be studied using a stability index. In this work, we use the stability index $\nu_i=\frac{1}{2}(\lambda_i+1/\lambda_i)$ where $\lambda_i$ is an eigenvalue of the monodromy matrix~\cite{spreenTrajectoryDesignTargeting2021}. The monodromy matrix is a one-period state transition matrix for a periodic orbit, relating any initial perturbation in the state to its future state after one time period. Stable orbits have all their stability indices less than or equal to unity, i.e. $|\nu_i|\leq1\;\;\forall i$.

For the Earth-Moon system, the 9:2 NRHO from the southern $L_2$ family was chosen as the baseline orbit for the Gateway mission and some other technology demonstration missions~\cite{gardnerCAPSTONESummaryHighly2023}. The 9:2 ratio denotes its resonance with the synodic period. Similarly, the northern $L_1$ family NRHOs are being studied in great detail for their stability properties and their potential for science missions to Saturn's icy moon Enceladus~\cite{mackenzieEnceladusOrbilanderMission2021,booneHaloOrbitConfigurations2026, sidhoumLowthrustTrajectoryDesign2024a}. In this work, we consider the stationkeeping problem for both these orbits: the Earth-Moon 9:2 NRHO and a Saturn-Enceladus $L_1$ NRHO. The two reference orbits are plotted in \Cref{fig:ref_orbits}, and their parameters are given in \Cref{tab:orbit_params}.

\begin{figure}[htb!]
\centering
\begin{subfigure}[b]{0.42\textwidth}
\centering
\includegraphics[width=\textwidth]{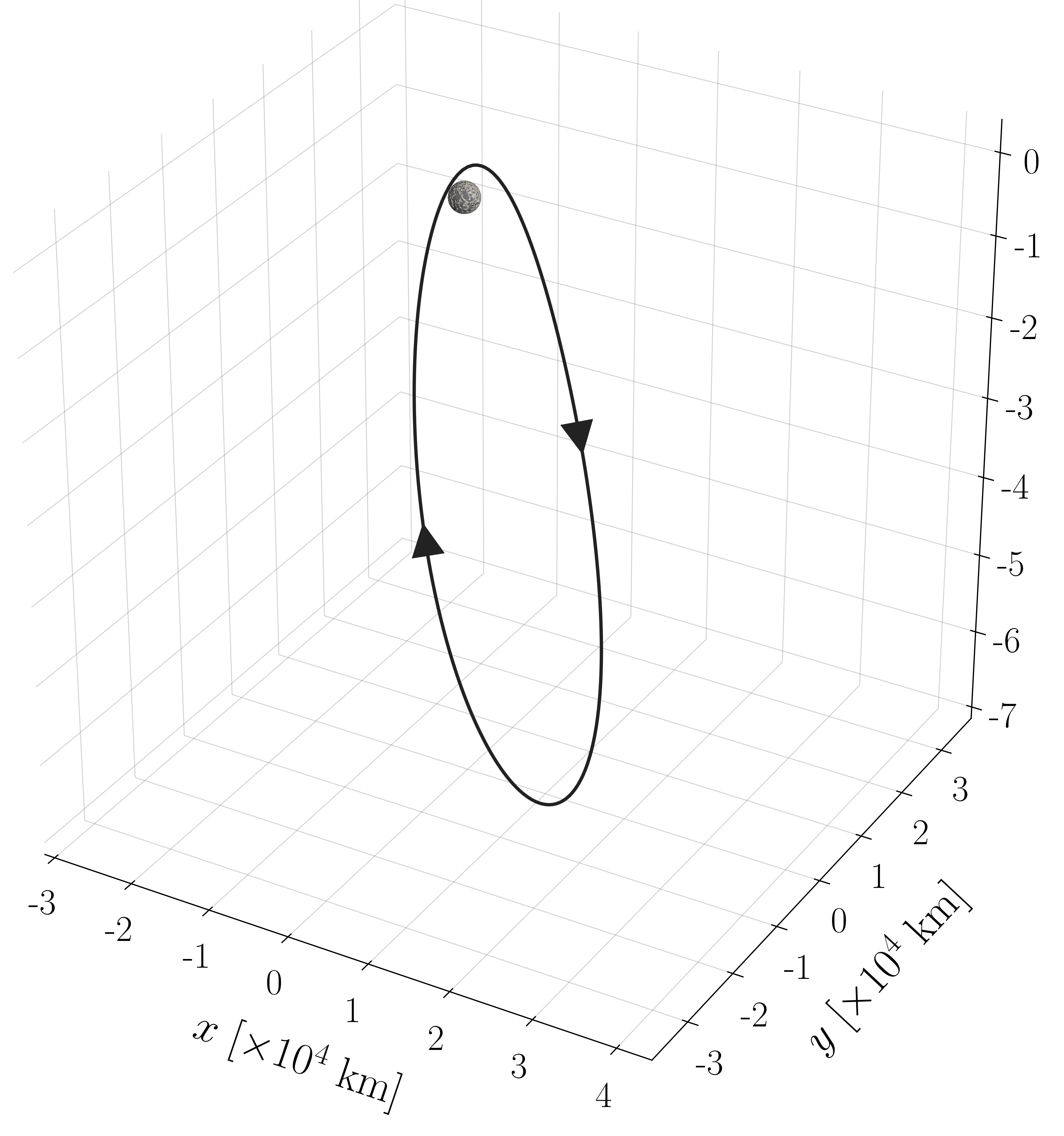}
\caption{Earth-Moon 9:2 NRHO.}
\label{fig:em_nrho}
\end{subfigure}
\hfill
\begin{subfigure}[b]{0.42\textwidth}
\centering
\includegraphics[width=\textwidth]{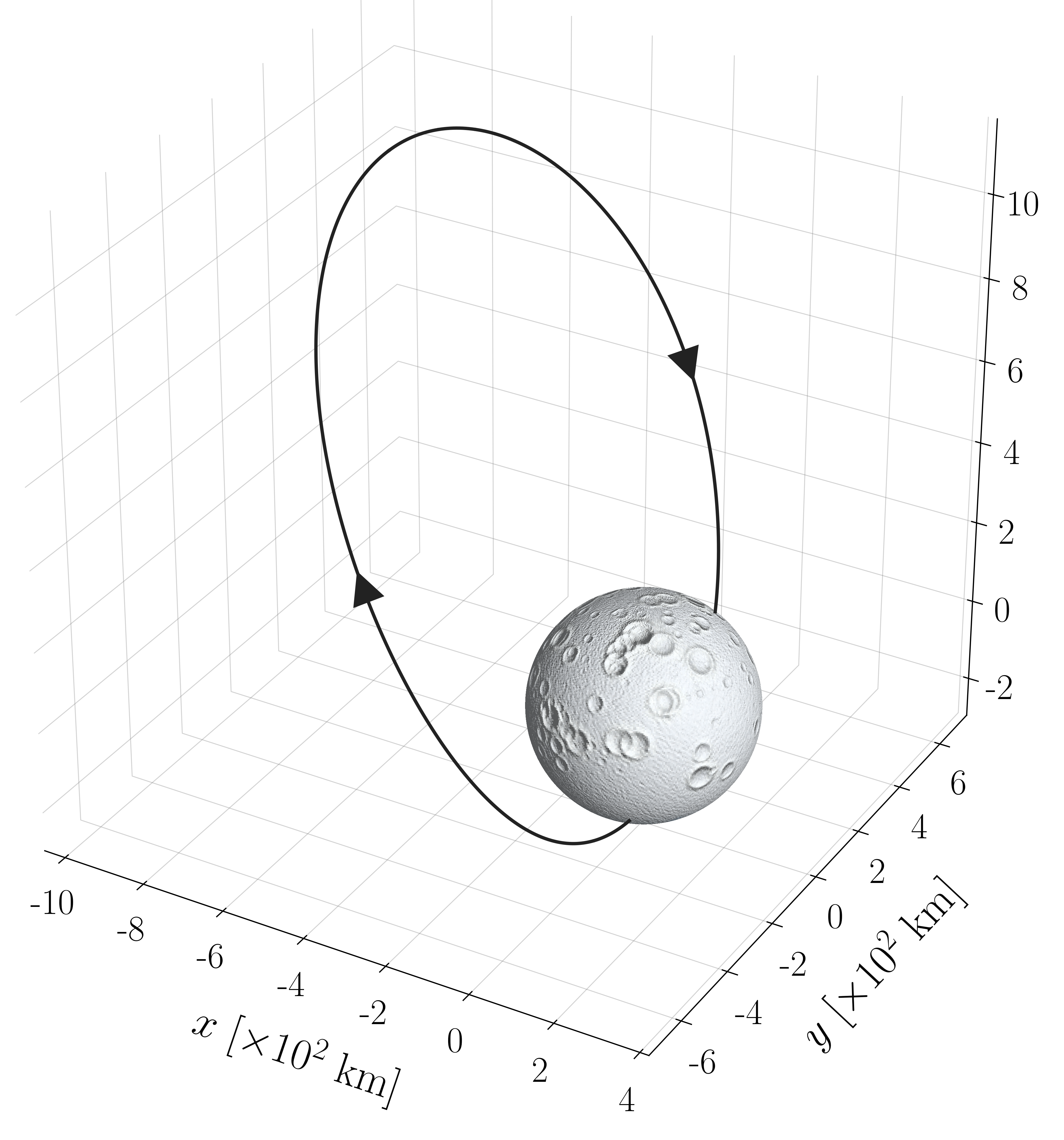}
\caption{Saturn-Enceladus $L_1$ NRHO.}
\label{fig:se_nrho}
\end{subfigure}
\caption{Baseline reference orbits, each shown in the secondary-centered rotating frame with the secondary body drawn to scale: (a) the Earth-Moon 9:2 southern $L_2$ NRHO; (b) the Saturn-Enceladus $L_1$ NRHO.}
\label{fig:ref_orbits}
\end{figure}
\subsection{Stochastic Hybrid System Framework}
To account for missed thrust events in the TCM policy design, we model the spacecraft as a stochastic hybrid system, which can switch between nominal and safe mode of operation. Imken et al.~\cite{imkenModelingSpacecraftSafe2018} study missed thrust event data from previous flight missions, and determine the exponential and Weibull distributions to be a good fit for the distribution of missed thrust events. The memoryless property of the exponential distribution motivates the use of the Markov Jump Systems (MJS) framework to model the resulting stochastic hybrid system. Further, we consider multiplicative noise in the MJS framework, which can model both maneuver execution errors and model linearization errors.

\subsubsection{Sampled-Data Markov Jump Systems}

In this work, we consider a subcategory of MJS called the sampled-data Markov jump system (SDMJS), in which the control and the mode are discrete-time variables that can only change at the finite sampling times, while the state evolves continuously in between. This matches spacecraft operations, where maneuvers are commanded at discrete TCM epochs. Moreover, for low-thrust propulsion, most direct methods assume zero-order hold for the control. 

Consider the probability space $(\Omega,\mathcal{F},\mathbb{P})$, where $\Omega$ is the sample space, $\mathcal{F}$ is the $\sigma$-algebra of Borel measurable events with natural filtration $\mathcal{F}_k\triangleq\sigma\{\mathbf{x}(t_0),r_0,\ldots,\mathbf{x}(t_k),r_k\}$, and $\mathbb{P}$ is the probability measure. The SDMJS is defined on this probability space as a pair of stochastic processes $(r_k,\mathbf{x}(t))$, where $r_k$ is a discrete-time Markov chain $r:\mathbb{Z}_{\geq0}\times\Omega\rightarrow\mathcal{N}$ and $\mathbf{x}(t)$ is a continuous-time stochastic process, evolving according to:
\begin{equation}\label{sdmjs}
    \mathcal{J} : \begin{cases}
    \dot{\mathbf{x}}(t) = \mathbf{f}_0(\mathbf{x}(t),\,r_k) + {B}(r_k)\,\mathbf{u}_k(r_k),
    & t \in [t_k,\,t_{k+1})
    \\[6pt]
   \mathbb{P}(r_{k+1} = j \mid r_k = i) = p_{ij}, \quad i,j \in \mathcal{N},
    & t = t_{k+1}
    \end{cases}
\end{equation}

If $\mathcal{N}$ is a singleton set, i.e. $r_k=1\;\forall k$, then \cref{sdmjs} is equivalent to the standard nonlinear dynamics equation with discrete control given in \cref{statespace}. Incorporating the mode dynamics generalizes the dynamics to allow for sudden switches in the dynamics function $\mathbf{f}$, or the control matrix ${B}$. This additional parameter $r_k$ is referred to as the active mode over the sampling interval $[t_k,t_{k+1})$, and it can take values from the discrete set $\mathcal{N}=\mathbb{Z}_{1:N}$, with $N$ being the total number of modes. The switching between these modes is governed by a homogeneous Markov chain with transition probability $p_{ij}=\mathbb{P}(r_{k+1}=j|r_k=i)$. We also define the transition probability matrix $\mathcal{P}$ whose $(i,j)^{\mathrm{th}}$ entry is $p_{ij}$. The mode probability distribution vector is given by $\boldsymbol{\pi}_k=[\pi_k(1)\ldots\pi_k(N)]$, and for a given initial mode distribution $\boldsymbol{\pi}_0$, it evolves as $\boldsymbol{\pi}_k=\boldsymbol{\pi}_0\mathcal{P}^k$.

\subsubsection{Maneuver Execution Error Model}

Even in nominal operations with no MTE occurrence, the executed spacecraft burns often have some error called the maneuver execution error (MEE). Conventional control noise models, like the Gates model~\cite{gatesSimplifiedModelMidcourse}, are typically incorporated in optimization-based mission design as additive noise~\cite{oguriChanceConstrainedControlSafe2024, kumagaiRobustCislunarLowThrust, ridderhofChanceconstrainedCovarianceControl2020}. However, the underlying error, even in the Gates model, scales with the magnitude of the commanded burn. Therefore, an additive representation is an oversimplification for MEE whenever the executed control has a control-dependent error. This motivates a multiplicative noise model that can be embedded in the control design itself, so that the error depends on the control magnitude directly.

Therefore, we consider a multiplicative noise model for the control input. This model has two parameters $\epsilon_p$ and $\epsilon_m$, which introduce noise in a direction perpendicular to, and along the control vector $\mathbf{u}_k$, respectively. $\epsilon_p$ can also be interpreted as a pointing error while $\epsilon_m$ indicates magnitude error. 

The pointing error is introduced in the plane perpendicular to the control vector $\mathbf{u}_k=[u_{k,x},u_{k,y},u_{k,z}]^\top$. This plane can be represented as $\mathbf{w}\times \mathbf{u}_k$ for a given vector $\mathbf{w}=[w_1, w_2, w_3]^\top$. We assume $w_i$ to be wide-sense stationary (WSS) white noise. Then, 
\begin{align}
    \mathbf{w}\times \mathbf{u}_k = (w_1T_1 + w_2T_2 + w_3T_3) \mathbf{u}_k
\end{align}
where
$$T_1=\begin{bmatrix}0 & 0 & 0\\0 & 0 & -1\\0 & 1 & 0 \end{bmatrix},\quad T_2=\begin{bmatrix}0 & 0 & 1\\0 & 0 & 0\\-1 &  0 & 0 \end{bmatrix}, \quad T_3=\begin{bmatrix}0 & -1 & 0\\1 & 0 & 0\\0 & 0 & 0 \end{bmatrix}$$
We now define the corresponding multiplicative noise matrix $\delta B_l$ as $\delta B_l=\begin{bsmallmatrix}
    \mathbf{0}_3\\\epsilon_p T_l
\end{bsmallmatrix}\; \forall l \in \{1,2,3\}$. These three terms represent errors in the plane perpendicular to the vector $\mathbf{u}_k$. The WSS white noise terms $w_i$ are incorporated as coefficients of $\delta B_l$ in the dynamics propagation.

For magnitude error, we introduce a fourth term in the multiplicative noise model for control. This is given by $\delta B_4 = \begin{bsmallmatrix}
    \mathbf{0}_3\\
    \epsilon_m I_3
\end{bsmallmatrix}$, and will augment the control vector to $\mathbf{u}_k+w_4\epsilon_m \mathbf{u}_k$, where $w$ is a WSS white noise. This term represents a magnitude error in the control vector.

Therefore, we represent the control input matrix with multiplicative noise as $B_{\Delta} = B + \sum_{l=1}^{4} w_l \delta B_l$, where $w_l$ are bounded WSS white noise terms, i.e. $\mathbb{E}[w_l]=0, \mathbb{E}[w_l^2]=1$, $\mathbb{E}[w_l w_m]=0$ $\forall l \neq m$, and $w_l^2\leq\delta_u$ almost surely. \Cref{fig:exec_error_model} shows a schematic describing the MEE model we employ in our work. \Cref{fig:exec_error_total} shows the relationship between nominal control input and executed burn, and \Cref{fig:exec_error_decomp} shows the composition of the error vector, with the pointing and magnitude error vectors shown in green and orange respectively.

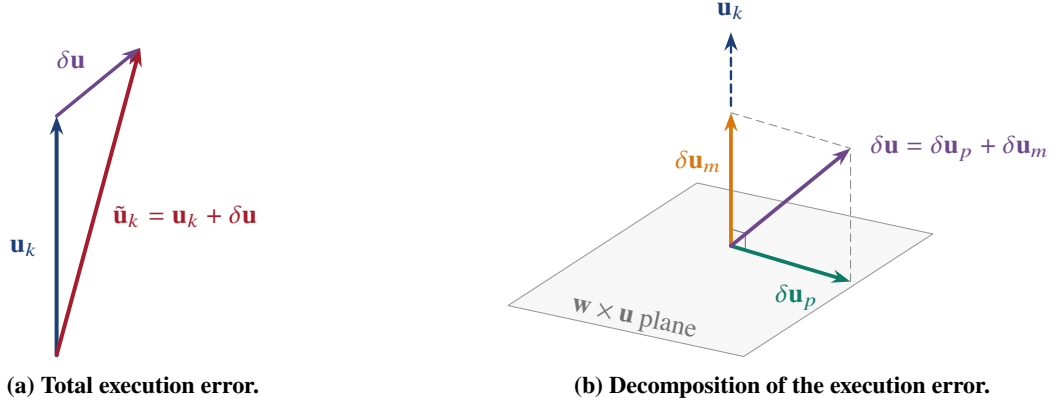
\begin{figure}[htb!]
    \centering
    \tdplotsetmaincoords{68}{120}
    \begin{subfigure}[b]{0.48\textwidth}
    \centering
    \begin{tikzpicture}[tdplot_main_coords,scale=1.1,line cap=round,line join=round,
        >={Stealth[length=7pt,width=5pt]}]
      \coordinate (O) at (0,0,0);
      \coordinate (T) at (0,0,3.1);          
      \coordinate (E) at (0,1.16,4.22);      
    
      \draw[cNom,line width=1.5pt,->] (O) -- (T)
            node[pos=0.45,left=3pt] {$\mathbf{u}_k$};
      \draw[cErr,line width=1.3pt,->] (T) -- (E)
            node[pos=0.5,above left=1pt] {$\delta\mathbf{u}$};
      \draw[cExec,line width=1.5pt,->] (O) -- (E)
            node[pos=0.45,right=3pt] {$\tilde{\mathbf{u}}_k=\mathbf{u}_k+\delta\mathbf{u}$};
    \end{tikzpicture}
    \caption{Total execution error.}
    \label{fig:exec_error_total}
    \end{subfigure}
    \hfill
    \begin{subfigure}[b]{0.48\textwidth}
    \centering
    \begin{tikzpicture}[tdplot_main_coords,line cap=round,line join=round,
        >={Stealth[length=7pt,width=5pt]}]
      \begin{scope}[canvas is xy plane at z=0]
        \fill[black!3,draw=black!40,thin] (-1.7,-1.5) rectangle (3.3,2.1);
      \end{scope}
      \begin{scope}[tdplot_screen_coords]
        \node[black!55,rotate=-12.2] at (-1.28,-0.95) {$\mathbf{w}\times\mathbf{u}$ plane};
      \end{scope}
    
      \coordinate (A) at (0,0,0);
      \coordinate (P) at (0.3,2.0,0);        
      \coordinate (M) at (0,0,1.9);          
      \coordinate (S) at (0.3,2.0,1.9);      
    
      \draw[black!55,thin] (0,0,0.24) -- (0.036,0.24,0.24) -- (0.036,0.24,0);
    
      \draw[black!45,densely dashed,thin] (P) -- (S);
      \draw[black!45,densely dashed,thin] (M) -- (S);
    
      \draw[cNom,densely dashed,line width=0.9pt,->] (0,0,2.0) -- (0,0,3.05)
            node[above=1pt] {$\mathbf{u}_k$};
    
      \draw[cPoint,line width=1.4pt,->] (A) -- (P)
            node[pos=0.55,below=3pt] {$\delta\mathbf{u}_p$};
      \draw[cMag,line width=1.4pt,->] (A) -- (M);
      \node[cMag,anchor=east] at (-0.02,-0.02,1.2) {$\delta\mathbf{u}_m$};
      \draw[cErr,line width=1.4pt,->] (A) -- (S);
      \node[cErr,anchor=west,xshift=4pt] at (0.3,2.0,1.9)
            {$\delta\mathbf{u}=\delta\mathbf{u}_p+\delta\mathbf{u}_m$};
    \end{tikzpicture}
    \caption{Decomposition of the execution error.}
    \label{fig:exec_error_decomp}
    \end{subfigure}
    \caption{Maneuver execution error model. The executed control $\tilde{\mathbf{u}}_k$ is the nominal
    control $\mathbf{u}_k$ plus an error $\delta\mathbf{u}$ (a), which splits into a pointing error
    $\delta\mathbf{u}_p$ in the $\mathbf{w}\times\mathbf{u}$ plane and a magnitude error
    $\delta\mathbf{u}_m$ along $\mathbf{u}_k$ (b).}
    \label{fig:exec_error_model}
    \end{figure}
\subsection{Nonlinear Stochastic Optimal Control Problem}
Now we outline the Nonlinear Stochastic Optimal Control Problem (NL-SOCP) we aim to solve in this work.
\subsubsection{Objective Function}
We consider a LQR-type quadratic cost function on both the control and state in expectation, given by:
\begin{equation}
    J = \mathbb{E} \left[ \int_0^{\infty} \{\delta\mathbf{x}^\top Q \delta\mathbf{x} + \mathbf{u}^\top R \mathbf{u} \}dt \right]
\end{equation}
{where $\delta\mathbf{x}(t)\triangleq\mathbf{x}(t)-\mathbf{\bar{x}}(t)$ is the state deviation from the baseline orbit $\mathbf{\bar{x}}(t)$, and $Q\succeq0$ and $R\succ0$ are the cost weights for the state deviation $\delta\mathbf{x}$ and the control $\mathbf{u}$ respectively.}
\subsubsection{Path Constraints}
Since we deal with stochastic state variables, we consider chance constraints, which should be satisfied with probability one, as given below:
\begin{align}
    \mathbb{P}(\delta\mathbf{x}\in\mathcal{X})=1, \quad \mathbb{P}(\mathbf{u}\in\mathcal{U})=1
\end{align}
where $\mathcal{X}$ and $\mathcal{U}$ are defined as $\mathcal{X} \triangleq \{\delta\mathbf{x}: \|H_k\delta\mathbf{x}\| \leq 1 \},\, H_k = x_{b,k}^{-1}[I_3\,\, 0_3]$ and $\mathcal{U} \triangleq \{\mathbf{u}: \|\mathbf{u}\|\leq u_{max}\}$. The baseline orbit $\mathbf{\bar{x}}(t)$ is the Earth-Moon 9:2 NRHO and the Saturn Enceladus northern $L_1$ NRHO in this work.
\subsubsection{Initial Dispersion}
We assume that the initial state deviation $\delta\mathbf{x}_0$ is distributed in some known bounded set $\mathcal{D}_0$ with probability one, which models the initial dispersion error. We further assume that the mean of the distribution of the initial state corresponds to the apolune of the NRHO but do not make any other assumptions on the nature of the distribution of the initial state. Mathematically, we have:
\begin{align}
    \mathbb{P}(\delta\mathbf{x}_0\in\mathcal{D}_0)=1, \quad \mathbb{E}[\delta\mathbf{x}_0] = 0
\end{align}
\subsubsection{Stability Condition}
We want the spacecraft to asymptotically converge to the baseline orbit. To assure this, we enforce mean square stability of the Markov jump system, defined below:
\begin{definition}
    An MJS is said to be mean square stable (MSS) iff:
    \begin{equation}
        \lim_{t\rightarrow \infty}\mathbb{E}[\|\delta\mathbf{x}(t)\|^2] = 0
    \end{equation}
\end{definition}
\subsubsection{NL-SOCP}
Under the initial condition, objective function, path and stability constraints discussed above, we now define the nonlinear stochastic optimal control problem we aim to solve in this work:
\begin{problem}\label{NL_SOCP}
    Solve the following nonlinear stochastic optimal control problem:
\begin{subequations}\label{eq:nl_socp}
\begin{alignat}{2}
    \underset{\mathbf{u}_k(r_k)}{\min} \quad & J= \mathbb{E} \left[ \int_0^{\infty} (\delta\mathbf{x}^\top Q \delta\mathbf{x} + \mathbf{u}^\top R \mathbf{u})dt \right] \\
    \text{s.t.} \quad & \dot{\mathbf{x}}(t)= \mathbf{f}_0(\mathbf{x}(t),r_k) + B_{\Delta}(r_k)\, \mathbf{u}_k(r_k), \;\; t\in[t_k,t_{k+1}) && \quad \triangleleft \text{SDMJS dynamics with MEE} \\
    & \mathbb{P}(r_{k+1}=j \mid r_k=i) = p_{ij}, \;\; i,j\in\mathcal{N}, \;\; t=t_{k+1} && \quad \triangleleft \text{Markov mode switching} \\
    &  \delta\mathbf{x}_0\in\mathcal{D}_0, \;\; \mathbb{E}[\delta\mathbf{x}_0] = 0, \;\; r_0\sim\boldsymbol{\pi}_0 && \quad \triangleleft \text{Initial state and mode distribution} \\
    & \mathbb{P}(\delta\mathbf{x}\in\mathcal{X})=1, \mathbb{P}(\mathbf{u}\in\mathcal{U})=1 && \quad \triangleleft \text{Path constraints} \\
    & \lim_{t\rightarrow \infty}\mathbb{E}[\|\delta\mathbf{x}(t)\|^2] = 0 && \quad \triangleleft \text{Stability constraints}
    \end{alignat}
\end{subequations}
\end{problem}
This is a computationally intractable problem, for several reasons: (i) the stochastic switching dynamics and MEE-induced multiplicative noise render scenario-based direct optimization techniques unscalable. At every discrete node, the switching dynamics will lead to an exponential number of cases, while multiplicative noise cannot be trivially applied to such frameworks. (ii) the nonlinear dynamics are difficult to handle in optimization frameworks, as they lead to a non-convex problem, and (iii) for a general class of control policies, the problem is an infinite-horizon, non-convex problem which makes it intractable.

In the following section, we outline a modeling strategy to convert the original nonlinear stochastic optimal control problem (\Cref{NL_SOCP}) to a linearized stochastic optimal control problem, and then utilize LMI-based controller synthesis framework to solve a linearized optimal control problem, which will provide an approximate solution to \Cref{NL_SOCP}.

\section{Modeling Strategy}\label{sec:modeling}
To solve \Cref{NL_SOCP} in a tractable manner, we elaborate the modeling strategy we undertake to convert the original nonlinear stochastic optimal control problem to a linearized stochastic optimal control problem. We follow the approach of first linearizing and discretizing the single-mode system about a baseline orbit, and then extending that discrete-time linear system to a discrete-time Markov jump linear system (MJLS), which can also model missed thrust events, MEE and state linearization errors.
\subsection{Linearization and Discretization}
To linearize and discretize a continuous-time system of the form $\dot{\mathbf{{x}}}= \mathbf{f}(\mathbf{x},\mathbf{u})$, we use the first-order Taylor expansion of $\mathbf{f}(\mathbf{x},\mathbf{u})$ about a reference trajectory $\mathbf{\bar{x}}(t),\mathbf{\bar{u}}(t)$ to get the discrete-time linear system of the form~\cite{oguriChanceConstrainedControlSafe2024}:
\begin{equation}
     \delta\mathbf{x}_{k+1} = A_k \delta\mathbf{x}_k + B_k  \mathbf{u}_k
\end{equation}
where $\delta\mathbf{x}_k\triangleq\mathbf{x}(t_k)-\mathbf{\bar{x}}(t_k)$ is the deviation from the reference trajectory $\mathbf{\bar{x}}(t)$, and the matrices $A_k$ and $B_k$ are given by~\cite{oguriChanceConstrainedControlSafe2024}:
\begin{equation}\label{lindisc}
\begin{aligned}
& A_k=\Phi\left(t_{k+1}, t_k\right), 
& B_k= \begin{cases}A_k B\left(t_k\right) & \text { (impulsive) } \\
A_k \int_{t_k}^{t_{k+1}} \Phi^{-1}\left(t, t_k\right) B(t) \mathrm{d} t & \text { (low thrust) }\end{cases}
\end{aligned}
\end{equation}
where $\Phi$ is the State Transition Matrix (STM) for the system~\cite{oguriChanceConstrainedControlSafe2024,kumagaiRobustCislunarLowThrust}. The MEE matrices $\delta B_l$ are also discretized analogously using \cref{lindisc}, giving the discrete-time noise matrices $B_{k,l}$.
\subsection{Linearization Error Modeling}

While we use a linearized system to describe the dynamics around the baseline orbit, it ignores the higher order effects, which might not be negligible depending on the deviation from the baseline orbit. Therefore, considering the linearized dynamics about a perturbed state will better approximate the nonlinear dynamics. \Cref{fig:unmodeled_stepoff} shows an orbit with an initial perturbation $\delta\mathbf{x}$, and downstream at time $k$, the linear map $\delta\mathbf{x}_{k+1}=A_k\delta\mathbf{x}_k$ is only true for the baseline orbit $\mathbf{\bar{x}}_k$ and its vicinity. If the perturbation $\delta\mathbf{x}$ grows larger, a more accurate linear map would be about the perturbed state $\mathbf{\bar{x}}_k+ \delta\mathbf{x}_k$ which is modeled by the slope $A_k +\delta A_k$. This improvement in linear approximation of the original nonlinear dynamics is also shown in \Cref{fig:unmodeled_jacobian} where the red slope, which indicates a one-dimensional projection of the new linear map, better approximates the true nonlinear dynamical flow (shown in grey) at $\mathbf{\bar{x}}_k+\delta\mathbf{x}_k$ compared to the previous linear map $A_k$ obtained at $\mathbf{\bar{x}}_k$.
\begin{figure}[htb!]
    \centering
    \begin{subfigure}[b]{0.48\textwidth}
    \centering
    \begin{tikzpicture}[line cap=round,line join=round,
        >={Stealth[length=7pt,width=5pt]},
        node/.style={circle,fill,inner sep=1.4pt}]
      \fill[black!25] (0,1.3) circle (0.24);
      \node[black!55,anchor=north] at (0,1.04) {\scriptsize Moon};

      \draw[cNom,line width=1.2pt]
        plot[smooth cycle,tension=0.7] coordinates
        {(0,-2.5) (1.25,-1.5) (1.05,0.6) (0,1.75) (-1.05,0.6) (-1.25,-1.5)};
      \node[cNom,anchor=east] at (-1.6,-0.6) {\small $\mathbf{\bar{x}}(t)$};

      \draw[cExec,line width=1.0pt,densely dashed,->]
        plot[smooth,tension=0.7] coordinates
        {(-0.34,-2.86) (-1.55,-1.6) (-1.28,0.7) (-0.03,1.85) (1.28,0.7) (1.6,-1.55) (0.75,-2.9)};

      \coordinate (Ap) at (0,-2.5);
      \coordinate (Pp) at (-0.34,-2.86);
      \node[node,cNom] at (Ap) {};
      \node[node,cExec] at (Pp) {};
      \draw[cErr,line width=1.0pt] (Ap) -- (Pp);
      \node[cErr,right=1pt] at ($(Ap)!0.6!(Pp)+(0.02,-0.1)$) {\small $\delta\mathbf{x}_0$};

      \coordinate (Bk)  at (1.05,0.6);
      \coordinate (Bk1) at (1.25,-1.5);
      \coordinate (Pk)  at (1.28,0.7);
      \coordinate (Pk1) at (1.6,-1.55);
      \node[node,cNom]  at (Bk)  {};
      \node[node,cNom]  at (Bk1) {};
      \node[node,cExec] at (Pk)  {};
      \node[node,cExec] at (Pk1) {};
      \node[cNom,anchor=north east] at ($(Bk)+(-0.04,-0.08)$)  {\small $\mathbf{\bar{x}}_k$};
      \node[cNom,anchor=east]       at ($(Bk1)+(-0.08,-0.02)$) {\small $\mathbf{\bar{x}}_{k+1}$};
      \node[cExec,anchor=west]      at ($(Pk)+(0.08,0)$)       {\small $\mathbf{\bar{x}}_k+\delta\mathbf{x}_k$};
      \node[cExec,anchor=west]      at ($(Pk1)+(0.08,-0.02)$)  {\small $\mathbf{\bar{x}}_{k+1}+\delta\mathbf{x}_{k+1}$};

      \draw[cNom,thin,->] ($(Bk)+(-0.6,-0.7)$) to[bend right=25] ($(Bk1)+(-0.55,0.4)$);
      \node[cNom] at (0.15,-0.45) {\small $A_k$};
      \draw[cExec,thin,->] ($(Pk)+(0.85,-0.35)$) to[bend left=25] ($(Pk1)+(0.7,0.35)$);
      \node[cExec,anchor=west] at (2.55,-0.45) {\small $A_k+\delta A_k$};
    \end{tikzpicture}
    \caption{Step-off along the unstable eigenvector.}
    \label{fig:unmodeled_stepoff}
    \end{subfigure}
    \hfill
    \begin{subfigure}[b]{0.48\textwidth}
    \centering
    \begin{tikzpicture}[line cap=round,line join=round,scale=1.15,
        >={Stealth[length=7pt,width=5pt]},
        node/.style={circle,fill,inner sep=1.4pt}]
      \draw[black!60,->] (-0.2,0) -- (4.7,0) node[below] {\small $\mathbf{x}_k$};
      \draw[black!60,->] (0,-0.2) -- (0,3.9) node[left] {\small $\mathbf{F}_k(\mathbf{x}_k)$};

      \draw[black!45,line width=1.1pt,domain=0.25:3.35,samples=80]
        plot (\x,{0.55+0.05*\x+0.028*\x*\x*\x*\x});
      \node[black!55,anchor=south east] at (3.0,3.35)
            {\small $\mathbf{x}_{k+1}=\mathbf{F}_k(\mathbf{x}_k)$};

      \coordinate (B) at (1.3,0.6949);
      \draw[cNom,line width=1.2pt] (0.4,{0.6949+0.2960*(0.4-1.3)}) -- (4.3,{0.6949+0.2960*3.0});
      \node[cNom,anchor=west] at (4.35,{0.6949+0.2960*3.0}) {\small slope $A_k$};
      \node[node,cNom] at (B) {};
      \node[cNom,anchor=north] at ($(B)+(-0.15,-0.12)$) {\small $\mathbf{\bar{x}}_k$};

      \coordinate (P) at (2.7,2.1743);
      \draw[cExec,line width=1.2pt,densely dashed] (1.8,{2.1743+2.2044*(1.8-2.7)}) -- (3.75,{2.1743+2.2044*1.05});
      \node[cExec,anchor=west] at (3.85,{2.1743+2.2044*1.05}) {\small slope $A_k+\delta A_k$};
      \draw[cNom,thin,densely dotted] (P) -- ($(P)+(1.35,{1.35*0.2960})$); 
      \node[node,cExec] at (P) {};
      \node[cExec,anchor=south east] at ($(P)+(-0.08,0.1)$) {\small $\mathbf{\bar{x}}_k+\delta\mathbf{x}_k$};

      \draw[cErr,line width=0.9pt,->] ($(P)+(16.5:1.15)$) arc[start angle=16.5,end angle=65.6,radius=1.15];
      \node[cErr,anchor=west] at ($(P)+(38:1.22)$) {\small $\delta A_k$};

      \draw[black!40,thin,densely dotted] (B) -- (1.3,0);
      \draw[black!40,thin,densely dotted] (P) -- (2.7,0);
      \draw[cErr,decorate,decoration={brace,mirror,amplitude=4pt}] (1.3,-0.1) -- (2.7,-0.1)
            node[midway,below=4pt] {\small $\delta\mathbf{x}_k$};
    \end{tikzpicture}
    \caption{One-dimensional projection of Jacobian perturbation.}
    \label{fig:unmodeled_jacobian}
    \end{subfigure}
    \caption{Linearization Error Schematic: (a) shows that linearizing about the perturbed reference yields a discrete map
    $A_k+\delta A_k$, which better represents the true dynamics $\mathbf{F}_k(\mathbf{x}_k)$ as indicated in (b).}
    \label{fig:unmodeled_dynamics}
\end{figure}
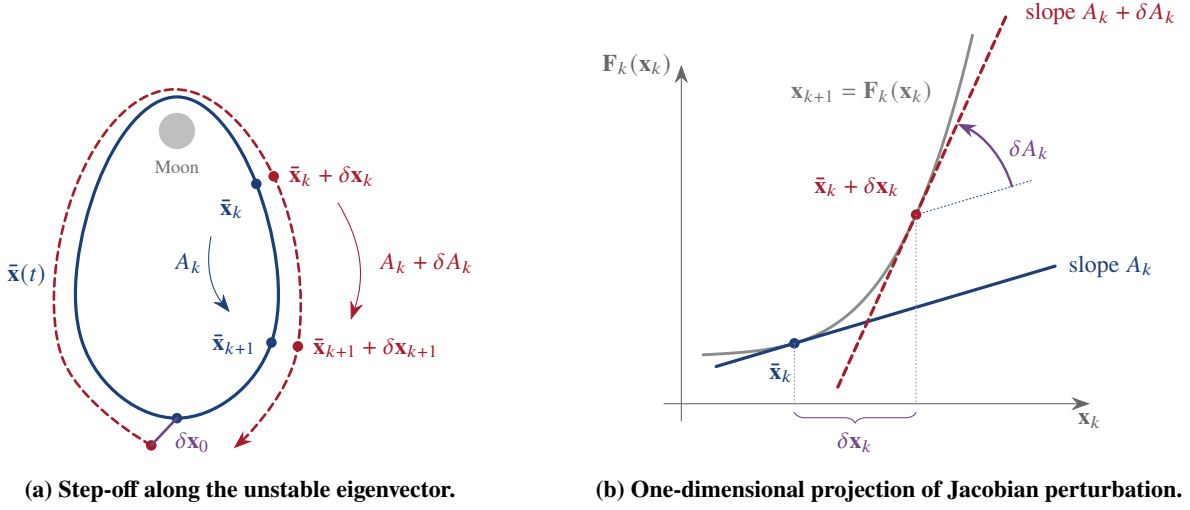

In particular, deviations in the direction of the unstable manifold of periodic orbits can be detrimental. Therefore, if we consider the corresponding eigenvector at each node on the reference orbit, it is the direction that is most sensitive to linearization errors; i.e., if the spacecraft is perturbed along this direction and our model does not represent the dynamics well enough, the controller might not be able to stabilize it, and the unstable natural dynamics can cause the spacecraft to diverge from the baseline orbit. 

To account for this, we incorporate the effect of linearization error in our model through state multiplicative noise terms. The noise terms are obtained by re-linearizing about the perturbed state $\mathbf{\bar{x}}_k+\delta\mathbf{x}_k$ to form $A^p_k=A_k + \delta A_k$, where $\delta\mathbf{x}_k$ is the perturbation. This initial perturbation is computed as $\delta\mathbf{x}_k=\delta d\cdot \mathbf{v}_{us}$, where $\delta d$ is the step-off distance (chosen beforehand,) and $\mathbf{v}_{us}$ is the unstable eigenvector at that node. The multiplicative noise term $\delta A_k$ is given by the difference between the new linearization about the perturbed state and the baseline orbit linearization discrete-time map $A_k$. We taper the step-off distances $\delta d$ based on its proximity to the Moon, which allows us to balance conservativeness with robustness while designing controllers. For a given set of step-off distances $\delta d_{p}$ and $\delta d_a$ given at periapsis and apoapsis respectively, the step-off distance $\delta d_k$ at node $k$ is given by:
\begin{equation}
    \delta d_k = \delta d_p + (\delta d_a - \delta d_p) \cdot \frac{d_{2,k}-d_{2,p}}{d_{2,a}-d_{2,p}}
\end{equation} 
where $d_{2,k}$ is the distance to the moon from baseline orbit node $\mathbf{\bar{x}}_k$, and $d_{2,p}$ and $d_{2,a}$ are its values at periapsis and apoapsis respectively. As with the execution-error terms, the associated scalar noises are normalized using bounded WSS noise terms. Our model offers a structured framework to model linearization error in the control design process.

\subsection{Markov Jump Linear Systems}

For the linearized system with MEE and state linearization errors modeled as multiplicative noise, we now define the corresponding stochastic switching system using the Markov Jump Linear System (MJLS) framework. Under the same probability space defined earlier for \cref{sdmjs}, we define a corresponding discrete-time MJLS with multiplicative noise as a pair of stochastic processes $(r_k,\delta\mathbf{x}_k)$, where $r_k$ is a discrete-time Markov chain and $\delta\mathbf{x}_k$ is a discrete-time stochastic process, evolving according to the following dynamics:
\begin{equation}\label{MJLS}
    \delta\mathbf{x}_{k+1} = \left(A_k(r_k)+\sum_{l=1}^{m_1} w^{\{x\}}_{k,l}(r_k)A_{k,l}(r_k)\right)\delta\mathbf{x}_k + \left(B_k(r_k)+\sum_{l=1}^{m_2} w^{\{u\}}_{k,l}(r_k)B_{k,l}(r_k)\right)\mathbf{u}_k
\end{equation}
where $A_k(r_k)$ and $B_k(r_k)$ are the state dynamics and control input matrices, respectively, for the active mode $r_k$, and $w^{\{x\}}_{k,l}$ and $w^{\{u\}}_{k,l}$ are bounded wide-sense stationary (WSS) white noise, which obey $(w^{\{x\}}_{k,l})^2\leq\delta_x$ and $(w^{\{u\}}_{k,l})^2\leq\delta_u$ almost surely respectively.
\subsubsection{Linear Time-Varying Periodic Markov Jump Systems (LTVPM)}
Since we deal with periodic orbits in this work, we will leverage a subclass of discrete-time MJLS, where the individual dynamics of each mode are periodic. Such systems are called LTVPM and their multiplicative noise version is defined below: 
\begin{definition}\label{def:ltvpm}
    An MJLS of the form \cref{MJLS} is called an LTVPM with period $T$ iff, $\forall i \in \mathcal{N}$ and $\forall n \in \mathbb{N}$, both the nominal system matrices and the multiplicative noise data repeat with period $T$, i.e. 
    \begin{equation*}
        \begin{aligned}
            A_{k+nT}(i)&=A_{k}(i), &\qquad B_{k+nT}(i)&=B_{k}(i),\\
            A_{k+nT,l}(i)&=A_{k,l}(i) \quad \forall\, l \in \mathbb{Z}_{1:m_1}, &\qquad B_{k+nT,l}(i)&=B_{k,l}(i) \quad \forall\, l \in \mathbb{Z}_{1:m_2}
        \end{aligned}
    \end{equation*}
\end{definition}
\subsubsection{Control Law}
In this work, we will consider a periodic mode-dependent linear feedback control policy, given by:
\begin{align}\label{control_law}
    \mathbf{u}_k(i)=K_k(i)\delta\mathbf{x}_k, \quad \forall k
\end{align}
where $K_k(i)$ is $T$-periodic $\forall i$, i.e. $K_{k+nT}(i) = K_k(i)\, \forall n\in\mathbb{N},i\in\mathcal{N}$.
Under \cref{control_law}, we can represent the MJLS dynamics as follows:
\begin{align}\label{concise_mjls}
    \delta\mathbf{x}_{k+1}= \Phi_k(i)\delta\mathbf{x}_k
\end{align}
where,
\begin{equation}\label{eq:Phi_k}
    \Phi_k(i)=A_k(i)+\sum_{l=1}^{m_1} w^{\{x\}}_{k,l}(i)A_{k,l}(i)+ \left(B_k(i)+\sum_{l=1}^{m_2} w^{\{u\}}_{k,l}(i)B_{k,l}(i)\right) K_k(i)
\end{equation}
\subsection{Stationkeeping as an LTVPM Control Problem}
We now elaborate on how we utilize our LTVPM framework to model the stationkeeping problem we defined in \cref{NL_SOCP}. 
\subsubsection{Failure Modeling}
To incorporate MTEs, we model the spacecraft as a two-mode MJLS, where the nominal mode assumes full control authority, and the failure mode assumes no control authority, representing a missed thrust event. We define two parameters $TTF_{90}$ and $TTR_{95}$ as the 90th percentile failure time and 95th percentile recovery time respectively. That is, once the spacecraft is operating nominally, at least one MTE occurs within $TTF_{90}$ with $90\%$ probability, and once an MTE begins, recovery occurs within $TTR_{95}$ with $95\%$ probability.  

The exponential distribution is memoryless, i.e. the probability of failure or recovery at any time $t$ is independent of the time spent in the current mode. Therefore, we can model the time to failure and the time to recovery as exponential distributions with parameters $\lambda_F$ and $\lambda_R$, respectively. Using the CDF for the exponential distribution, we can derive the parameters as:
\begin{equation}\label{expfit}
    \begin{aligned}
        1-e^{-\lambda_F\, TTF_{90}}=0.9,\qquad 1-e^{-\lambda_R\, TTR_{95}}=0.95
    \end{aligned}
\end{equation}
For a uniformly discretized grid with interval length $\delta t$, the corresponding transition probability matrix $\mathcal{P}$ of the Markov chain is given by:
\begin{equation}\label{tpm_formula}
    \mathcal{P}=\begin{bmatrix}
        e^{-\lambda_F\delta t} & 1-e^{-\lambda_F\delta t}\\
        1-e^{-\lambda_R\delta t} & e^{-\lambda_R\delta t}
    \end{bmatrix}
\end{equation}
\subsubsection{Impulsive TCM Modeling}\label{sec:backup}

Unlike low-thrust continuous propulsion, impulsive propulsion executes a burn only at specific times in an orbit. We model those opportunities by setting the control input matrix $B_k(i)$ to its nominal value at the TCM nodes and to zero elsewhere.

However, when a failure mode occurrence coincides with a planned TCM burn, the spacecraft will have to wait until the next TCM opportunity even if it recovers to nominal operations earlier. Therefore, in such scenarios, it is desirable to incorporate additional backup TCM opportunities, to offer more opportunities for recovery TCM burns after the spacecraft recovers to nominal operations. 

To achieve this, we utilize a three-mode MJLS, with an additional backup mode. This backup mode can only activate after recovery from failure mode, and it transitions to either nominal or failure mode immediately after. This allows for additional TCM opportunities to the spacecraft in case of missed burns. The corresponding Markov chain schematic for this three-mode framework and an illustrative scenario with MTE recovery and a backup-mode burn is shown in \Cref{fig:backup_framework}. In that scenario the MTE causes the nominal TCM at perilune to be missed, and the first two backup nodes cannot be executed either; after recovery, the correction burn is executed at the third backup node. Without the backup nodes, the spacecraft would have to wait until the next nominal TCM opportunity.
  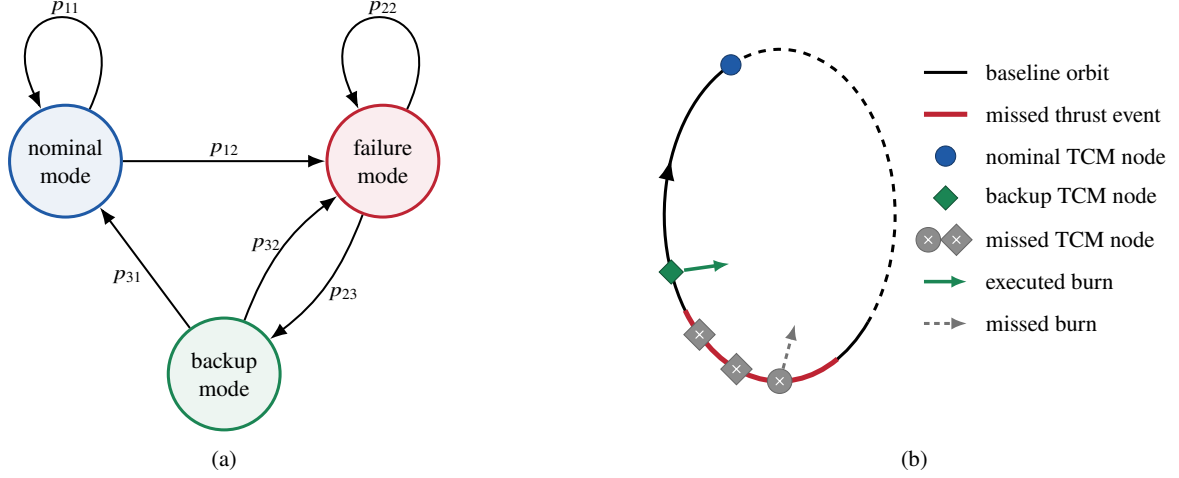
\begin{figure}[htbp!]
    \centering
    \resizebox{0.95\linewidth}{!}{%
  \begin{tikzpicture}[
      >={Latex[length=2.4mm,width=1.8mm]},
      font=\small,
      nomnode/.style ={circle, fill=nomblue, draw=nomblue!60!black, inner sep=0pt, minimum size=2.8mm},
      bknode/.style  ={diamond, fill=bkgreen, draw=bkgreen!60!black, inner sep=0pt, minimum size=3.4mm},
      xmark/.style   ={path picture={
          \draw[white, line width=0.45pt, line cap=round]
            ($(path picture bounding box.center)+(-0.05,-0.05)$) --
            ($(path picture bounding box.center)+( 0.05, 0.05)$)
            ($(path picture bounding box.center)+(-0.05, 0.05)$) --
            ($(path picture bounding box.center)+( 0.05,-0.05)$);}},
      skipnom/.style ={circle, fill=black!45, draw=black!60, inner sep=0pt, minimum size=3.4mm, xmark},
      skipbk/.style  ={diamond, fill=black!45, draw=black!60, inner sep=0pt, minimum size=4.4mm, xmark},
      mode/.style    ={circle, draw, very thick, minimum size=15.5mm, align=center, font=\small},
      plab/.style    ={font=\footnotesize, inner sep=1.5pt},
      blab/.style    ={font=\small, align=left, inner sep=1.5pt},
    ]
   
  \begin{scope}[local bounding box=panelA]
    \node[mode, draw=nomblue,  fill=nomblue!8]  (N) at (0,1.1)     {nominal\\mode};
    \node[mode, draw=failred,  fill=failred!8]  (F) at (4.4,1.1)   {failure\\mode};
    \node[mode, draw=bkgreen,  fill=bkgreen!8]  (B) at (2.2,-1.85) {backup\\mode};
   
    \draw[->, thick] (N) edge[loop above, looseness=7, in=115, out=65] node[plab, above] {$p_{11}$} (N);
    \draw[->, thick] (F) edge[loop above, looseness=7, in=115, out=65] node[plab, above] {$p_{22}$} (F);
   
    \draw[->, thick] (N) -- node[plab, above, align=center] {$p_{12}$} (F);
    \draw[->, thick] (F) edge[bend left=16] node[plab, below right=-1pt, align=center] {$p_{23}$} (B);
    \draw[->, thick] (B) edge[bend left=16] node[plab, above left=-2pt, align=center] {$p_{32}$} (F);
    \draw[->, thick] (B) -- node[plab, below left=-1pt, align=center] {$p_{31}$} (N);

  \end{scope}
   
  \begin{scope}[shift={(9.9,0.35)}, local bounding box=panelB]
    \def\a{1.6}   
    \def\b{2.3}   
   
    \draw[black, line width=1.2pt]
        ({\a*cos(-40)},{\b*sin(-40)})
        arc[start angle=-40, end angle=-245, x radius=\a, y radius=\b];
    \draw[black, line width=1.2pt, dash pattern=on 3pt off 3pt]
        ({\a*cos(-245)},{\b*sin(-245)})
        arc[start angle=-245, end angle=-400, x radius=\a, y radius=\b];
   
    \draw[failred, line width=2pt]
        ({\a*cos(-60)},{\b*sin(-60)})
        arc[start angle=-60, end angle=-145, x radius=\a, y radius=\b];
   
    \draw[black, line width=1.2pt, -{Latex[length=3mm, width=2.3mm]}]
        ({\a*cos(-196)},{\b*sin(-196)}) -- ({\a*cos(-200)},{\b*sin(-200)});
   
    \coordinate (SN) at ({\a*cos(-90)},{\b*sin(-90)});
    \draw[black!55, line width=1.2pt, dash pattern=on 2.2pt off 1.8pt,
          -{Latex[length=2.2mm,width=1.7mm]}]
        ($(SN)+(75:0.18)$) -- ($(SN)+(75:0.82)$);
    \node[skipnom] at (SN) {};
   
    \node[skipbk] at ({\a*cos(-112)},{\b*sin(-112)}) {};
    \node[skipbk] at ({\a*cos(-134)},{\b*sin(-134)}) {};
   
    \coordinate (RB) at ({\a*cos(-160)},{\b*sin(-160)});
    \draw[bkgreen, line width=1.4pt, -{Latex[length=2.2mm,width=1.7mm]}]
        ($(RB)+(8:0.18)$) -- ($(RB)+(8:0.82)$);
    \node[bknode] at (RB) {};
   
    \node[nomnode] (NN) at ({\a*cos(-245)},{\b*sin(-245)}) {};

    \begin{scope}[shift={(2.30,1.98)}]
      \draw[black, line width=1.2pt] (-0.30,0) -- (0.30,0);
      \node[blab, anchor=west] at (0.50,0) {baseline orbit};

      \draw[failred, line width=2pt] (-0.30,-0.58) -- (0.30,-0.58);
      \node[blab, anchor=west] at (0.50,-0.58) {missed thrust event};

      \node[nomnode] at (0,-1.16) {};
      \node[blab, anchor=west] at (0.50,-1.16) {nominal TCM node};

      \node[bknode] at (0,-1.74) {};
      \node[blab, anchor=west] at (0.50,-1.74) {backup TCM node};

      \node[skipnom] at (-0.24,-2.32) {};
      \node[skipbk]  at ( 0.15,-2.32) {};
      \node[blab, anchor=west] at (0.50,-2.32) {missed TCM node};

      \draw[bkgreen, line width=1.4pt, -{Latex[length=2.2mm,width=1.7mm]}]
          (-0.30,-2.90) -- (0.30,-2.90);
      \node[blab, anchor=west] at (0.50,-2.90) {executed burn};

      \draw[black!55, line width=1.2pt, dash pattern=on 2.2pt off 1.8pt,
            -{Latex[length=2.2mm,width=1.7mm]}] (-0.30,-3.48) -- (0.30,-3.48);
      \node[blab, anchor=west] at (0.50,-3.48) {missed burn};
    \end{scope}
  \end{scope}
   
  \coordinate (tagline) at ($(current bounding box.south)+(0,-0.4)$);
  \node[font=\small] at (panelA.south |- tagline) {(a)};
  \node[font=\small] at (panelB.south |- tagline) {(b)};
   
  \end{tikzpicture}%
    }
    \caption{Schematic of the backup TCM node framework: (a)~the three-mode Markov jump system, and
    (b)~an MTE and recovery sequence along the baseline orbit.}
    \label{fig:backup_framework}
\end{figure}

The transition probabilities for the three-mode Markov chain follow the same exponential failure and recovery model of \cref{expfit}. The nominal and failure self-transition probabilities over a sampling interval $\delta t$ are unchanged from the two-mode case, $p_{11}=e^{-\lambda_F\delta t}$ and $p_{22}=e^{-\lambda_R\delta t}$, so that failures and recoveries are governed by the same rates $\lambda_F,\lambda_R$ derived from $TTF_{90}$ and $TTR_{95}$. The distinguishing feature is that a recovering spacecraft does not return directly to the nominal mode; instead it transitions into the transient backup mode, i.e. $p_{23}=1-e^{-\lambda_R\delta t}$. The backup mode must then transition either back to nominal or into another MTE, with the same outgoing probabilities as the nominal mode, $p_{31}=e^{-\lambda_F\delta t}$ and $p_{32}=1-e^{-\lambda_F\delta t}$. The resulting transition probability matrix, with modes ordered as nominal, failure, and backup, is:
\begin{equation}\label{tpm_formula_backup}
    \mathcal{P}=\begin{bmatrix}
        e^{-\lambda_F\delta t} & 1-e^{-\lambda_F\delta t} & 0\\
        0 & e^{-\lambda_R\delta t} & 1-e^{-\lambda_R\delta t}\\
        e^{-\lambda_F\delta t} & 1-e^{-\lambda_F\delta t} & 0
    \end{bmatrix}
\end{equation}

\subsubsection{Linearized Stochastic Optimal Control Problem}

We now define the linearized stochastic optimal control problem below as:

\begin{problem}\label{problem:l_socp}
    Solve the following optimization problem:
\begin{subequations}\label{eq:l_socp}
\begin{alignat}{2}
    \underset{K_k(i)}{\min} \quad & J= \mathbb{E}\left[\sum_{k=0}^{\infty} \delta\mathbf{x}_k^\top Q(i) \delta\mathbf{x}_k + \mathbf{u}_k^\top(i) R(i) \mathbf{u}_k(i)\right] \label{prob1_J}\\
    \text{s.t.} \quad & \delta\mathbf{x}_{k+1} = \Phi_k(i) \delta\mathbf{x}_k && \quad \triangleleft \text{Linearized MJLS dynamics} \\
    &  \delta\mathbf{x}_0\in\mathcal{D}_0 && \quad \triangleleft \text{Initial state distribution} \\
    & \mathbb{P}(\delta\mathbf{x}\in\mathcal{X})=1, \mathbb{P}(\mathbf{u}\in\mathcal{U})=1 && \quad \triangleleft \text{Path constraints} \\
    & \lim_{k\rightarrow \infty}\mathbb{E}[\|\delta\mathbf{x}_k\|^2] = 0 && \quad \triangleleft \text{Stability constraints}
    \end{alignat}
\end{subequations}
\end{problem}
This remains an intractable problem, primarily because conventional direct optimization techniques enforce a single dynamical system model at every node, and cannot account for the stochastic mode-switching nature of LTVPM. This is further exacerbated with multiplicative noise. Therefore, to solve \cref{problem:l_socp} in a tractable manner, we develop a new stability theory for LTVPM with multiplicative noise, and propose an LMI-based controller synthesis framework below.

\section{LMI-based Controller Synthesis for LTVPM}\label{sec:synthesis}

To aid the stability analysis of LTVPM, we first define the following moment variables, which are commonly used in MJLS literature:  
\begin{align}
    \pi_k(i) \triangleq \mathbb{E} \left[\mathds{1}_{\{r_k=i\}}\right], \quad
    & q_k(i) \triangleq \mathbb{E} \left[\delta\mathbf{x}_k \mathds{1}_{r_k=i} \right],   \notag \\
    X_k(i)\triangleq\mathbb{E}[\delta\mathbf{x}_k\delta\mathbf{x}^\top_k\mathds{1}_{r_k=i}],\quad& X_k\triangleq [X_k(1)\ \ldots\ X_k(N)]\notag
\end{align}
Such partial expectation quantities lead to easier algebraic derivation, as we will see in the proofs ahead. We also introduce the operators used throughout the stability analysis, which generalize their noise-free counterparts~\cite{shrivastavaRobustControllerSynthesis2024a} to the multiplicative-noise setting. For any $V\in\mathbb{M}^N$,
\begin{equation}\label{operator_defs}
    \begin{aligned}
        \mathcal{E}^i(V) &\triangleq \sum_{j\in\mathcal{N}}p_{ij}V(j)\\
        \mathcal{T}_k^j(V) &\triangleq \sum_{i\in\mathcal{N}}p_{ij}\left\{\phi_k(i)V(i)\phi_k^\top(i)+\sum_{l=1}^{m_1}A_{k,l}(i)V(i)A_{k,l}^\top(i)+\sum_{l=1}^{m_2}B_{k,l}(i)K_k(i)V(i)K_k^\top(i)B_{k,l}^\top(i)\right\}\\
        \mathcal{L}_k^i(V) &\triangleq \phi_k^\top(i)\mathcal{E}^i(V)\phi_k(i)+\sum_{l=1}^{m_1}A_{k,l}^\top(i)\mathcal{E}^i(V)A_{k,l}(i)+\sum_{l=1}^{m_2}K_k^\top(i)B_{k,l}^\top(i)\mathcal{E}^i(V)B_{k,l}(i)K_k(i)
    \end{aligned}
\end{equation}
where $\phi_k(i)=A_k(i)+B_k(i)K_k(i)=\mathbb{E}[\Phi_k(i)]$ is the noise-free part of the closed-loop dynamics. We collect these mode-dependent operators into the finite sets of operators $\mathcal{T}_k(V)\triangleq\left(\mathcal{T}_k^1(V),\ldots,\mathcal{T}_k^N(V)\right)$ and $\mathcal{L}_k(V)\triangleq\left(\mathcal{L}_k^1(V),\ldots,\mathcal{L}_k^N(V)\right)$, and define their one-period compositions as: 
\begin{equation}\label{one_period_ops}
    \mathcal{G}_T \triangleq \mathcal{T}_{T-1}\circ\mathcal{T}_{T-2}\circ\cdots\circ\mathcal{T}_0, \qquad
    \mathcal{H}_T \triangleq \mathcal{L}_0\circ\mathcal{L}_1\circ\cdots\circ\mathcal{L}_{T-1}
\end{equation}

\begin{proposition}\label{prop:adjoint}
    \cite[Proposition 3.2]{costaDiscreteTimeMarkovJump2005} The operators defined in \cref{operator_defs} satisfy $\mathcal{T}_k^\star(V)=\mathcal{L}_k(V)$ for all $V\in\mathbb{M}^N$.
\end{proposition}
\begin{proof}
    The extension of proof in~\cite{costaDiscreteTimeMarkovJump2005} to the multiplicative-noise operators of \cref{operator_defs} is immediate, since each noise term is of the same form.
\end{proof}
To analyze the mean-square stability, we first consider the first two moments' propagation:
\begin{theorem}\label{thm:moment_propagation}
    The first- and second-moment propagation equations for \cref{concise_mjls} are given by:
    \begin{equation}\label{moment_propagation}
        \begin{aligned}
            q_{k+1}(j)&=\sum_{i\in\mathcal{N}}p_{ij}\{\phi_k(i)q_k(i)\}\\
            X_{k+1}(j)&= \mathcal{T}_k^j(X_{k})
        \end{aligned}
    \end{equation}
    where $\phi_k(i)=A_k(i)+B_k(i)K_k(i)=\mathbb{E}[\Phi_k(i)]$ denotes the noise-free part of the closed-loop dynamics.
\end{theorem}

\begin{proof}
    We first derive the mean propagation. By definition,
    \begin{equation}
        q_{k+1}(j)\triangleq\mathbb{E}[\delta\mathbf{x}_{k+1}\mathds{1}_{r_{k+1}=j}]=\sum_{i\in\mathcal{N}}p_{ij}\mathbb{E}[\Phi_k(i)\delta\mathbf{x}_k\mathds{1}_{r_k=i}]=\sum_{i\in\mathcal{N}}p_{ij}\{\phi_k(i)q_k(i)\}
    \end{equation}
    where the last step follows because the multiplicative noises are zero-mean and independent of both $\delta\mathbf{x}_k$ and the Markov chain, so that $\mathbb{E}[\Phi_k(i)]=\phi_k(i)$.

    Similarly, the second moment propagation is given by:
    \begin{equation}
        X_{k+1}(j)\triangleq\mathbb{E}[\delta\mathbf{x}_{k+1}\delta\mathbf{x}^\top_{k+1}\mathds{1}_{r_{k+1}=j}]=\sum_{i\in\mathcal{N}}p_{ij}\mathbb{E}[\Phi_k(i)\delta\mathbf{x}_k\delta\mathbf{x}_k^\top\Phi_k^\top(i)\mathds{1}_{r_k=i}]= \mathcal{T}_k^j(X_{k})
    \end{equation}
    where the last step takes the expectation over the WSS white noise.
\end{proof}
\begin{assumption}\label{as:zero_mean}
    In this work, we deal with zero-mean MJLS, i.e. $\mathbb{E}[\delta\mathbf{x}_k\mathds{1}_{r_k=i}]=0$. In stationkeeping applications, we design control policies for the spacecraft to return to the baseline orbit, which is also the reference for linearization and is used as the mean trajectory. This allows us to model all perturbations through the second moment, motivating the use of zero-mean MJLS.
\end{assumption}
\begin{theorem}
    An LTVPM of the form \cref{concise_mjls} is mean square stable (MSS) iff its corresponding operator $\mathcal{G}_T$ is contractive.
\end{theorem}
\begin{proof}
    See Lemma 1 in Ref.~\citenum{houSpectralTestsObservability2016}, as the second moment evolves linearly (\cref{moment_propagation}) even under multiplicative noise.
\end{proof}
Next, we discuss a Lyapunov-like stability condition for LTVPM with multiplicative noise that guarantees MSS. 
\begin{theorem}
    An LTVPM of the form \cref{concise_mjls} is mean square stable (MSS) iff $\exists$ a $T$-periodic finite set of positive definite matrices $P_k\in\mathbb{M}^N$ such that $\forall i \in \mathcal{N}$, we have:
    \begin{equation}\label{mss_eq1}
        \begin{aligned}
            P_k - \mathcal{L}_k(P_{k+1}) \succ 0, \quad \forall k \in \mathbb{Z}_{0:T-1}
        \end{aligned}
    \end{equation}
\end{theorem}

\begin{proof}
    We first show that \cref{mss_eq1} implies MSS.
    Iteratively applying \cref{mss_eq1} over one period, we get:
    \begin{align}
        P_0\succ \mathcal{L}_0(P_1)\succ \mathcal{L}_0(\mathcal{L}_1(P_2))\succ \cdots \succ [\mathcal{L}_0 \circ \mathcal{L}_1 \circ \cdots \circ \mathcal{L}_{T-1}](P_0) = \mathcal{H}_T (P_0) \succ 0
    \end{align}
    Therefore, this implies that the linear operator $\mathcal{H}_T$ is contractive, i.e. its spectral radius is less than 1. Now, from \cref{prop:adjoint}, we know that $\mathcal{T}_k^{\star} = \mathcal{L}_k$, and since the adjoint of a contractive operator is itself contractive, $\mathcal{G}_T$ as defined in \cref{one_period_ops} is also contractive, which implies MSS.

    For the \textit{only if} part, we first assume that the system is MSS. We then define the closed-loop state transition matrix for a given mode history $\{r_{k_0}\ldots r_{k-1}\}$ and a given noise realization as:
\begin{equation}\label{Phi_transition}
    \Phi(k,k_0) = \Phi_{k-1}(r_{k-1})\ldots\Phi_{k_0}(r_{k_0})
\end{equation}
By definition, $\Phi(k_0,k_0)\triangleq I$. For the same mode  and noise realization, the following relation stands: $\Phi(k_2,k_0)=\Phi(k_2,k_1)\Phi(k_1,k_0)\;\;\:\forall k_2\geq k_1 \geq k_0$.

Now, for $s\in\mathbb{Z}_{\geq0}$, let us define the truncated sum
\begin{equation}\label{eq:gamma}
    \Gamma_{k,s}(i)\triangleq\mathbb{E}\left[\sum_{t=0}^{s-1}\Phi(k+t,k)^\top\Phi(k+t,k)\;\middle|\;r_k=i\right], \qquad \Gamma_{k,0}(i)=0
\end{equation}
Consider $\Gamma_{k,s+1}(i)$, and since $\Phi(k+1+t,k)=\Phi(k+1+t,k+1)\Phi_k(i)$, we get
\begin{equation}\label{eq:gamma_rec}
\begin{aligned}
    \Gamma_{k,s+1}(i) &= I + \sum_{t=0}^{s-1}\mathbb{E}\left[\Phi_k^\top(i)\,\Phi(k+1+t,k+1)^\top\Phi(k+1+t,k+1)\,\Phi_k(i)\;\middle|\;r_k=i\right]\\
    &= I + \mathbb{E}\left[\Phi_k^\top(i)\,\Gamma_{k+1,s}(r_{k+1})\,\Phi_k(i)\;\middle|\;r_k=i\right]
     \;=\; I + \sum_{j\in\mathcal{N}}p_{ij}\,\mathbb{E}\left[\Phi_k^\top(i)\,\Gamma_{k+1,s}(j)\,\Phi_k(i)\right]\\
    &= I + \phi_k^\top(i)\mathcal{E}^i(\Gamma_{k+1,s})\phi_k(i) + \sum_{l=1}^{m_1} A_{k,l}^\top(i)\mathcal{E}^i(\Gamma_{k+1,s})A_{k,l}(i) + \sum_{l=1}^{m_2} K_k^\top(i)B_{k,l}^\top(i)\mathcal{E}^i(\Gamma_{k+1,s})B_{k,l}(i)K_k(i)\\
    &= \mathcal{L}_k^i(\Gamma_{k+1,s}) + I
\end{aligned}
\end{equation}

We now show that $\Gamma_{k,s}$ converges as $s\rightarrow\infty$. Each summand in \cref{eq:gamma} is positive semidefinite, so $\Gamma_{k,s}$ is nondecreasing in $s$. Iterating \cref{eq:gamma_rec} over one period, and using the $T$-periodicity of the system, i.e. $\Gamma_{k+T,s}=\Gamma_{k,s}$, we get:
\begin{equation}\label{eq:gamma_period}
    \Gamma_{k,s+T} = \mathcal{H}_T^{(k)}(\Gamma_{k,s}) + \Gamma_{k,T}, \qquad \mathcal{H}_T^{(k)} \triangleq \mathcal{L}_k\circ\mathcal{L}_{k+1}\circ\cdots\circ\mathcal{L}_{k+T-1}
\end{equation}
where the constant term follows by setting $s=0$, for which $\Gamma_{k,0}=0$. Since the system is MSS, $\mathcal{G}_T$ is contractive, and hence so is its adjoint $\mathcal{H}_T$ from \cref{prop:adjoint}. Iterating \cref{eq:gamma_period} then gives the series $\Gamma_{k,sT}=\sum_{t=0}^{s-1}\left(\mathcal{H}_T^{(k)}\right)^{t}(\Gamma_{k,T})$, so the limit
\begin{equation}\label{Pkdef}
    P_k(i)\triangleq\lim_{s\rightarrow\infty}\Gamma_{k,s}(i) = \sum_{t=0}^\infty \mathbb{E}\left[\Phi(k+t,k)^\top\Phi(k+t,k) \mid r_k=i \right]
\end{equation}
exists, where the second equality follows from the monotone convergence theorem~\cite{durrettProbabilityTheoryExamples2019}. Taking $s\rightarrow\infty$ in \cref{eq:gamma_rec}, we get
\begin{equation}\label{thm4finaleq}
    P_k(i) = \mathcal{L}_k^i(P_{k+1}) + I
\end{equation}
which implies $P_k(i) - \mathcal{L}_k^i(P_{k+1})=I\succ 0$. Since $P_k(i)\succeq\Gamma_{k,1}(i)=I\succ0$ and $P_{k+T}=P_k$, this is the required $T$-periodic finite set of positive definite matrices.
\end{proof}
In our previous work~\cite{shrivastavaRobustControllerSynthesis2024a}, we utilized this Lyapunov-like stability constraint to bound the cost using the performance analysis lemma. We extend the performance analysis lemma now to LTVPM with multiplicative noise: 
\begin{lemma}\label{lemma_performance_noise}
    (Performance analysis lemma with Multiplicative noise) Suppose there exist $T$-periodic mode-dependent positive semidefinite matrices $M_k(i)$ and positive definite matrices $P_k(i)$, and a positive scalar constant $\beta>0$ such that 
    \begin{equation}\label{per_lemma_2}
        \begin{aligned}
            P_k(i) \succ \mathcal{L}_k^i(P_{k+1}) + \tfrac{M_k(i)}{\beta},\quad\forall k \in \mathbb{Z}_{0:T-1}, i\in\mathcal{N}
        \end{aligned}
    \end{equation}
then the corresponding MJLS is MSS, with $\mathbb{E}  \left[\sum_{k=0}^\infty \delta\mathbf{x}_k^\top M_k(i)\delta\mathbf{x}_k\right] \leq \beta\, \mathbb{E}[\delta\mathbf{x}_0^\top P_0(i) \delta\mathbf{x}_0]$.
\end{lemma}
\begin{proof}
    Multiplying both sides of \cref{per_lemma_2} by $\delta\mathbf{x}_k^\top$ and $\delta\mathbf{x}_k$, we get:
    \begin{align} \label{lemma_p0}
        \delta\mathbf{x}_k^\top P_k(i)\delta\mathbf{x}_k \geq \delta\mathbf{x}_k^\top \mathcal{L}_k^i(P_{k+1})\delta\mathbf{x}_k + \frac{\delta\mathbf{x}_k^\top M_k(i)\delta\mathbf{x}_k}{\beta}
    \end{align}
    We first show that the first term on the RHS, $\delta\mathbf{x}_k^\top \mathcal{L}_k^i(P_{k+1})\delta\mathbf{x}_k$, is equivalent to $\mathbb{E}[\delta\mathbf{x}^\top_{k+1}P_{k+1}\delta\mathbf{x}_{k+1}|\mathcal{F}_k]$:
    \begin{align*}
        \mathbb{E}[\delta\mathbf{x}^\top_{k+1}P_{k+1}\delta\mathbf{x}_{k+1}|\mathcal{F}_k] = \mathbb{E}[\delta\mathbf{x}_k^\top \Phi^\top_k(i)P_{k+1}\Phi_k(i) \delta\mathbf{x}_k|\mathcal{F}_k]
    \end{align*}
    Since $\delta\mathbf{x}_k$ is $\mathcal{F}_k$-measurable, $\mathbb{E}[\delta\mathbf{x}_k^\top \Phi^\top_k(i)P_{k+1}\Phi_k(i) \delta\mathbf{x}_k|\mathcal{F}_k] = \delta\mathbf{x}_k^\top \mathbb{E}[\Phi^\top_k(i)P_{k+1}\Phi_k(i)|\mathcal{F}_k] \delta\mathbf{x}_k$. Expanding the inner expectation, we get:
    \begin{align*}
        &\mathbb{E}[\Phi^\top_k(i)P_{k+1}\Phi_k(i)|\mathcal{F}_k] = \mathbb{E}\left[ \left\{\phi_k(i) + \sum_{l=1}^{m_1}w^{\{x\}}_{k,l}(i)A_{k,l}(i) + \sum_{l=1}^{m_2}w^{\{u\}}_{k,l}(i)B_{k,l}(i)K_k(i)\right\}^\top P_{k+1}\right.\\
        &\hspace{4.5cm}\left.\times\left\{\phi_k(i) + \sum_{l=1}^{m_1}w^{\{x\}}_{k,l}(i)A_{k,l}(i) + \sum_{l=1}^{m_2}w^{\{u\}}_{k,l}(i)B_{k,l}(i)K_k(i)\right\} \,\middle|\, \mathcal{F}_k \right]\\
        &= \mathbb{E}\left[\left\{\phi_k(i)^\top P_{k+1}\phi_k(i) + \sum_{l=1}^{m_1}A_{k,l}(i)^\top P_{k+1}A_{k,l}(i) + \sum_{l=1}^{m_2}K_k^\top(i)B_{k,l}(i)^\top P_{k+1}B_{k,l}(i)K_k(i) \right\} \,\middle|\, \mathcal{F}_k \right]\\
        &= \phi_k(i)^\top \mathcal{E}^i(P_{k+1})\phi_k(i) + \sum_{l=1}^{m_1}A_{k,l}(i)^\top \mathcal{E}^i(P_{k+1})A_{k,l}(i)+ \sum_{l=1}^{m_2}K_k^\top(i)B_{k,l}(i)^\top \mathcal{E}^i(P_{k+1})B_{k,l}(i)K_k(i)
    \end{align*}
    which follows from the WSS noise terms being zero-mean, mutually uncorrelated with unit second moments, and independent of $\mathcal{F}_k$. The last step takes the expectation over $r_{k+1}$, which gives $\mathcal{L}_k^i(P_{k+1})$ by \cref{operator_defs}. Therefore, $\delta\mathbf{x}_k^\top \mathcal{L}_k^i(P_{k+1})\delta\mathbf{x}_k=\mathbb{E}[\delta\mathbf{x}^\top_{k+1}P_{k+1}\delta\mathbf{x}_{k+1}|\mathcal{F}_k]$. 
    Using this in \cref{lemma_p0}, we get:
    \begin{align}\label{lemma_p2}
        \delta\mathbf{x}_k^\top P_k(i)\delta\mathbf{x}_k \geq \mathbb{E}[\delta\mathbf{x}^\top_{k+1}P_{k+1}\delta\mathbf{x}_{k+1}|\mathcal{F}_k] + \frac{\delta\mathbf{x}_k^\top M_k(i)\delta\mathbf{x}_k}{\beta}
    \end{align}
    Taking the expectation of \cref{lemma_p2} and using law of total expectation, we get:
    \begin{align}
        \mathbb{E}[\delta\mathbf{x}_k^\top P_k(i)\delta\mathbf{x}_k] &\geq \mathbb{E}[\mathbb{E}[\delta\mathbf{x}^\top_{k+1}P_{k+1}\delta\mathbf{x}_{k+1}|\mathcal{F}_k]] + \frac{\mathbb{E}[\delta\mathbf{x}_k^\top M_k(i)\delta\mathbf{x}_k]}{\beta}\notag\\
        &= \mathbb{E}[\delta\mathbf{x}^\top_{k+1}P_{k+1}\delta\mathbf{x}_{k+1}] + \frac{\mathbb{E}[\delta\mathbf{x}_k^\top M_k(i)\delta\mathbf{x}_k]}{\beta}\label{lemma_p3}
    \end{align}
    Summing \cref{lemma_p3} from $k=0$ to $k=T-1$, we get:
    \begin{align}
        \mathbb{E}[\delta\mathbf{x}_0^\top P_0(i)\delta\mathbf{x}_0] \geq \mathbb{E}[\delta\mathbf{x}_T^\top P_T(i)\delta\mathbf{x}_T] + \frac{\sum_{k=0}^{T-1}\mathbb{E}[\delta\mathbf{x}_k^\top M_k(i)\delta\mathbf{x}_k]}{\beta}\label{lemma_p4}
    \end{align}
    Summing \cref{lemma_p4} over $s$ time periods, we get:
    \begin{align}
        \mathbb{E}[\delta\mathbf{x}_0^\top P_0(i)\delta\mathbf{x}_0] \geq \mathbb{E}[\delta\mathbf{x}_{sT}^\top P_{sT}(i)\delta\mathbf{x}_{sT}] + \frac{\sum_{k=0}^{sT-1}\mathbb{E}[\delta\mathbf{x}_k^\top M_k(i)\delta\mathbf{x}_k]}{\beta}\label{lemma_p5}
    \end{align}
    Since $M_k(i)\succeq0$ and \cref{per_lemma_2} implies MSS, the term $\mathbb{E}[\delta\mathbf{x}_{sT}^\top P_{sT}(i)\delta\mathbf{x}_{sT}]$ vanishes as $s\rightarrow\infty$. Taking that limit therefore gives:
    \begin{equation*}
        J \leq \beta \cdot \mathbb{E}[\delta\mathbf{x}_0^\top P_0(i) \delta\mathbf{x}_0] \qedhere
    \end{equation*}
\end{proof}
The condition \cref{per_lemma_2} is not directly usable for synthesis, since it is nonlinear in the pair $(P_k(i),K_k(i))$. \Cref{cor:performance_lmi} recasts it as an LMI in the change of variables $S_k(i)=P_k^{-1}(i)$ and $Y_k(i)=K_k(i)S_k(i)$.
\begin{corollary}\label{cor:performance_lmi}
    (LMI form of the performance analysis lemma) Let $M_k(i)=Q(i)+K_k^\top(i)R(i)K_k(i)$ be the stage cost of \cref{problem:l_socp}. If there exist a scalar $\beta>0$, a $T$-periodic finite set of positive definite matrices $S_k\in\mathbb{M}^N$, and a $T$-periodic finite set of matrices $Y_k(i)\in\mathbb{R}^{n_u\times n_x}$ satisfying
    \begin{equation}\label{MSS_LMI}
    \begin{bsmallmatrix}
    S_k(i) & \ast & \ast & \cdots & \ast & \ast & \cdots & \ast & \ast & \ast \\
    \ell_{i}^{\top}\left(A_{k}(i) S_k(i)+B_{k}(i) Y_k(i)\right) & \mathbf{S}_{k+1} & \ast & \cdots & \ast & \ast & \cdots & \ast & \ast & \ast \\
    \ell_{i}^{\top}A_{k,1}(i) S_k(i) & 0 & \mathbf{S}_{k+1} & \cdots & \ast & \ast & \cdots & \ast & \ast & \ast \\
    \vdots & \vdots & \vdots & \ddots & \vdots & \vdots & & \vdots & \vdots & \vdots \\
    \ell_{i}^{\top}A_{k,m_1}(i) S_k(i) & 0 & 0 & \cdots & \mathbf{S}_{k+1} & \ast & \cdots & \ast & \ast & \ast \\
    \ell_{i}^{\top}B_{k,1}(i) Y_k(i) & 0 & 0 & \cdots & 0 & \mathbf{S}_{k+1} & \cdots & \ast & \ast & \ast \\
    \vdots & \vdots & \vdots & & \vdots & \vdots & \ddots & \vdots & \vdots & \vdots \\
    \ell_{i}^{\top}B_{k,m_2}(i) Y_k(i) & 0 & 0 & \cdots & 0 & 0 & \cdots & \mathbf{S}_{k+1} & \ast & \ast \\
    Q^{1/2}(i)S_k(i) & 0 & 0 & \cdots & 0 & 0 & \cdots & 0 & \beta I & \ast \\
    R^{1/2}(i)Y_k(i) & 0 & 0 & \cdots & 0 & 0 & \cdots & 0 & 0 & \beta I
    \end{bsmallmatrix}
    \succ 0
    \end{equation}
    $\forall k \in \mathbb{Z}_{0:T-1}, i\in\mathcal{N}$, where $\ell_i\triangleq\begin{bsmallmatrix}\sqrt{p_{i1}}I & \ldots & \sqrt{p_{iN}}I\end{bsmallmatrix}$ and $\mathbf{S}_{k+1}\triangleq\mathrm{blkdiag}(S_{k+1}(1),\ldots,S_{k+1}(N))$, then \cref{per_lemma_2} holds with $P_k(i)=S_k^{-1}(i)$ and $K_k(i)=Y_k(i)S_k^{-1}(i)$. Consequently, the closed-loop LTVPM is MSS and $J\leq\beta\,\mathbb{E}[\delta\mathbf{x}_0^\top P_0(i)\delta\mathbf{x}_0]$.
\end{corollary}
\begin{proof}
    Under the control law \cref{control_law}, the stage cost of \cref{problem:l_socp} is $\delta\mathbf{x}_k^\top Q(i)\delta\mathbf{x}_k+\mathbf{u}_k^\top R(i)\mathbf{u}_k=\delta\mathbf{x}_k^\top M_k(i)\delta\mathbf{x}_k$, so the cost bound of \cref{lemma_performance_noise} is exactly the claimed bound on $J$. It therefore suffices to show that \cref{MSS_LMI} is equivalent to \cref{per_lemma_2} under the stated change of variables, i.e. $S_k(i)=P_k^{-1}(i)$ and $Y_k(i)=K_k(i)S_k(i)$.

    Since $\ell_i\mathbf{S}_{k+1}^{-1}\ell_i^\top=\sum_{j\in\mathcal{N}}p_{ij}S_{k+1}^{-1}(j)=\mathcal{E}^i(P_{k+1})$, substituting the definition of $\mathcal{L}_k^i$ from \cref{operator_defs} into \cref{per_lemma_2} gives
    \begin{equation}\label{cor_p1}
        \begin{aligned}
            P_k(i) &- \phi_k^\top(i)\ell_i\mathbf{S}_{k+1}^{-1}\ell_i^\top\phi_k(i)
            - \sum_{l=1}^{m_1}A_{k,l}^\top(i)\ell_i\mathbf{S}_{k+1}^{-1}\ell_i^\top A_{k,l}(i)\\
            &- \sum_{l=1}^{m_2}K_k^\top(i)B_{k,l}^\top(i)\ell_i\mathbf{S}_{k+1}^{-1}\ell_i^\top B_{k,l}(i)K_k(i)
            - \tfrac{1}{\beta}\left(Q(i)+K_k^\top(i)R(i)K_k(i)\right) \succ 0
        \end{aligned}
    \end{equation}
    Pre- and post-multiplying \cref{cor_p1} by $S_k(i)$, we get
    \begin{equation}\label{cor_p2}
        \begin{aligned}
            S_k(i) &- \left(A_k(i)S_k(i)+B_k(i)Y_k(i)\right)^\top\ell_i\mathbf{S}_{k+1}^{-1}\ell_i^\top\left(A_k(i)S_k(i)+B_k(i)Y_k(i)\right)\\
            &- \sum_{l=1}^{m_1}S_k(i)A_{k,l}^\top(i)\ell_i\mathbf{S}_{k+1}^{-1}\ell_i^\top A_{k,l}(i)S_k(i) - \sum_{l=1}^{m_2}Y_k^\top(i)B_{k,l}^\top(i)\ell_i\mathbf{S}_{k+1}^{-1}\ell_i^\top B_{k,l}(i)Y_k(i)\\
            &- \left(Q^{1/2}(i)S_k(i)\right)^\top(\beta I)^{-1}\left(Q^{1/2}(i)S_k(i)\right)
            - \left(R^{1/2}(i)Y_k(i)\right)^\top(\beta I)^{-1}\left(R^{1/2}(i)Y_k(i)\right) \succ 0
        \end{aligned}
    \end{equation}
    \Cref{cor_p2} is the Schur complement of the block-diagonal matrix $\mathrm{blkdiag}\left(\mathbf{S}_{k+1},\ldots,\mathbf{S}_{k+1},\beta I,\beta I\right)$ in \cref{MSS_LMI}. Since $S_{k+1}(j)\succ0$ and $\beta>0$, \cref{MSS_LMI} holds if and only if \cref{cor_p2} does, and hence if and only if \cref{per_lemma_2} does.
\end{proof}
\Cref{cor:performance_lmi} enforces MSS while bounding the cost. To enforce path constraints, we now introduce the notion of a $T$-periodic robust positively invariant (RPI) sequence of sets, which will allow us to enforce path constraints with probability one at all time steps. 
\begin{definition}[$T$-periodic RPI sequence]\label{def:rpi}
    Let $\{\mathscr{E}_k\}_{k=0}^{T-1}$ be a $T$-periodic sequence of  sets, i.e.  $\mathscr{E}_{k+nT}\triangleq\mathscr{E}_k$ for all $n\in\mathbb{N}$. The sequence is called \emph{$T$-periodic robust positively invariant} (RPI) for the closed-loop LTVPM \cref{concise_mjls} if, for every $k\in\mathbb{Z}_{\geq0}$, every active mode $i\in\mathcal{N}$, and every admissible noise realization satisfying $(w^{\{x\}}_{k,l}(i))^2\leq\delta_x$, $l\in\mathbb{Z}_{1:m_1}$, and $(w^{\{u\}}_{k,l}(i))^2\leq\delta_u$, $l\in\mathbb{Z}_{1:m_2}$,
    \begin{equation*}
        \delta\mathbf{x}_k\in\mathscr{E}_k \;\Longrightarrow\; \delta\mathbf{x}_{k+1}=\Phi_k(i)\delta\mathbf{x}_k\in\mathscr{E}_{k+1}.
    \end{equation*}
\end{definition}
Since we define the RPI sequence of sets for every noise realization, it holds along every trajectory of the Markov chain and noise realization. We now leverage this concept in \Cref{lemma_invariant_sets} to provide LMI conditions which admit such a sequence of $T$-periodic robust positively invariant sets.
\begin{lemma}\label{lemma_invariant_sets}
    (Robust Invariant Sets for LTVPM with Multiplicative Noise) For a finite sequence of tuples of positive definite matrices $\{P_k(i)\}_{i \in \mathcal{N}}$ with $P_T(i)\triangleq P_0(i)$, the sequence $\{\mathscr{E}_k\}_{k=0}^{T-1}$ defined by $\mathscr{E}_k=\{\delta\mathbf{x}:\delta\mathbf{x}^\top P_k(i) \delta\mathbf{x} \leq 1, \; \forall i \in \mathcal{N}\}$ is $T$-periodic robust positively invariant in the sense of \cref{def:rpi}, and satisfies $\mathcal{D}_0\subseteq\mathscr{E}_0$ for the initial-condition hull generated by $\{\delta\mathbf{x}_{0,v}\}_{v \in \mathbb{Z}_{1:c}}$, if there exist two sequences of non-negative scalars $\{\tau_{l}\}_{l\in\mathbb{Z}_{1:m_1}}$ and $\{\eta_{l}\}_{l\in\mathbb{Z}_{1:m_2}}$ such that the following holds $\forall k\in\mathbb{Z}_{0:T-1}$ and $\forall (i,j)\in\mathcal{N}\times\mathcal{N}$:
    \begin{align}
        &\hspace{6cm}\begin{bsmallmatrix}\label{l3_t1}
            1 & \ast \\
            \delta\mathbf{x}_{0, v} & S_0(i)
        \end{bsmallmatrix}
        \succeq 0\\
        &\begin{bsmallmatrix}\label{l3_t2}
                S_k(i) - \tau_{\text{total}} I_{n_x} & \ast & \dots & \ast & \ast & \dots & \ast & \ast \\
                0_{n_x} & \tau_{1} I_{n_x} & \dots & \ast & \ast & \dots & \ast & \ast \\
                \vdots & \vdots & \ddots & \vdots & \vdots & \ddots & \vdots & \vdots \\
                0_{n_x} & 0_{n_x} & \dots & \tau_{m_1} I_{n_x} & \ast & \dots & \ast & \ast \\
                0_{n_x} & 0_{n_x} & \dots & 0_{n_x} & \eta_{1} I_{n_x} & \dots & \ast & \ast \\
                \vdots & \vdots & \ddots & \vdots & \vdots & \ddots & \vdots & \vdots \\
                0_{n_x} & 0_{n_x} & \dots & 0_{n_x} & 0_{n_x} & \dots & \eta_{m_2} I_{n_x} & \ast \\
                A_k(i) S_k(i) + B_k(i) Y_k(i) & A_{k,1}(i) S_k(i) & \dots & A_{k,m_1}(i) S_k(i) & B_{k,1}(i) Y_k(i) & \dots & B_{k,m_2}(i) Y_k(i) & S_{k+1}(j)
            \end{bsmallmatrix}
            \succeq 0
    \end{align}
    where $S_k(i)=P_k(i)^{-1}$, $S_T(i)\triangleq S_0(i)$, and $Y_k(i)=K_k(i)S_k(i)$.
\end{lemma}
\begin{proof}
    Using Schur complement, we can rewrite \cref{l3_t1} as:
    \begin{align}
        \delta\mathbf{x}_{0, v}^\top P_0(i) \delta\mathbf{x}_{0, v} \leq 1
    \end{align}
    Now, we can use a Lyapunov-like inequality to guarantee that $\delta\mathbf{x}_{k+1}\in\mathscr{E}_{k+1}$ if $\delta\mathbf{x}_k\in\mathscr{E}_k$, which is given by:
    \begin{align}
        \delta\mathbf{x}_{k+1}^\top P_{k+1}(i_{k+1}) \delta\mathbf{x}_{k+1} \leq \delta\mathbf{x}_k^\top P_k(i_k) \delta\mathbf{x}_k \leq \ldots \leq \delta\mathbf{x}_0^\top P_0(i_0) \delta\mathbf{x}_0 \leq 1
    \end{align}
    where $i_k$ refers to the mode at time-step $k$. At time-step $k$, we have:
    \begin{align*}
        \delta\mathbf{x}_{k+1}^\top P_{k+1}(j) \delta\mathbf{x}_{k+1} -\delta\mathbf{x}_{k}^\top P_k(i) \delta\mathbf{x}_{k}&\leq 0\\
        \text{or, } \hspace{2cm}\Phi^\top_k(i) P_{k+1}(j) \Phi_k(i) - P_k(i) &\preceq 0
    \end{align*}
    Substituting $\Phi_k(i)$ from \cref{eq:Phi_k}, pre- and post-multiplying with $S_k(i)=P_k(i)^{-1}$, and using $Y_k(i)=K_k(i)S_k(i)$, we get:
    \begin{align}
        \psi_k(i)^\top S^{-1}_{k+1}(j) \psi_k(i) - S_k(i) \preceq 0
    \end{align}
    where $\psi_k(i) = A_k(i) S_k(i) + B_k(i) Y_k(i) + \sum_{l=1}^{m_1} w^{\{x\}}_{k,l}(i) A_{k,l}(i) S_k(i) + \sum_{l=1}^{m_2} w^{\{u\}}_{k,l}(i) B_{k,l}(i) Y_k(i)$. This can be written succinctly as:
    \begin{align}\label{l3_sp_1}
        (\boldsymbol{\xi}\otimes I_{n_x}) ^\top \mathbf{M} (\boldsymbol{\xi}\otimes I_{n_x}) \preceq 0
    \end{align}
    where:
    \begin{align}
        \boldsymbol{\xi} = \begin{bmatrix}
            1 & w^{\{x\}}_{k,1}(i) & \cdots & w^{\{x\}}_{k,m_1}(i) & w^{\{u\}}_{k,1}(i) & \cdots & w^{\{u\}}_{k,m_2}(i)
        \end{bmatrix}^\top
    \end{align}
    and $\mathbf{M}$ is given by:
    \begin{align}
        \mathbf{M} =
    \begin{bsmallmatrix}
    \Psi_0^\top S_{k+1}^{-1}(j) \Psi_0 - S_k(i) & \ast & \dots & \ast & \ast & \dots & \ast \\
    \Psi_{A,1}^\top S_{k+1}^{-1}(j) \Psi_0 & \Psi_{A,1}^\top S_{k+1}^{-1}(j) \Psi_{A,1} & \dots & \ast & \ast & \dots & \ast \\
    \vdots & \vdots & \ddots & \vdots & \vdots & \ddots & \vdots \\
    \Psi_{A,m_1}^\top S_{k+1}^{-1}(j) \Psi_0 & \Psi_{A,m_1}^\top S_{k+1}^{-1}(j) \Psi_{A,1} & \dots & \Psi_{A,m_1}^\top S_{k+1}^{-1}(j) \Psi_{A,m_1} & \ast & \dots & \ast \\
    \Psi_{B,1}^\top S_{k+1}^{-1}(j) \Psi_0 & \Psi_{B,1}^\top S_{k+1}^{-1}(j) \Psi_{A,1} & \dots & \Psi_{B,1}^\top S_{k+1}^{-1}(j) \Psi_{A,m_1} & \Psi_{B,1}^\top S_{k+1}^{-1}(j) \Psi_{B,1} & \dots & \ast \\
    \vdots & \vdots & \ddots & \vdots & \vdots & \ddots & \vdots \\
    \Psi_{B,m_2}^\top S_{k+1}^{-1}(j) \Psi_0 & \Psi_{B,m_2}^\top S_{k+1}^{-1}(j) \Psi_{A,1} & \dots & \Psi_{B,m_2}^\top S_{k+1}^{-1}(j) \Psi_{A,m_1} & \Psi_{B,m_2}^\top S_{k+1}^{-1}(j) \Psi_{B,1} & \dots & \Psi_{B,m_2}^\top S_{k+1}^{-1}(j) \Psi_{B,m_2}
    \end{bsmallmatrix}
    \end{align}
    where $\Psi_0 = A_k(i) S_k(i) + B_k(i) Y_k(i), \, \Psi_{A,l} = A_{k,l}(i) S_k(i),$ and $\Psi_{B,l} = B_{k,l}(i) Y_k(i)$.
Now, $\mathbf{M}$ can be written as:
\begin{align}
    \mathbf{M} = \mathbf{\Psi}^\top S_{k+1}^{-1}(j) \mathbf{\Psi} - \mathscr{S}_k(i)
\end{align}
where $\mathbf{\Psi} = \begin{bmatrix} \Psi_0 & \Psi_{A,1} & \dots & \Psi_{A,m_1} & \Psi_{B,1} & \dots & \Psi_{B,m_2} \end{bmatrix}$ is a block row matrix of size $n_x \times (1+m_1+m_2)n_x$, and $\mathscr{S}_k(i) = \mathrm{blkdiag}(S_k(i), 0_{n_x}, \dots, 0_{n_x})$ is a block diagonal matrix of size $(1+m_1+m_2)n_x \times (1+m_1+m_2)n_x$.
Now, from the almost-sure bounds $(w^{\{x\}}_{k,l})^2\leq\delta_x$ and $(w^{\{u\}}_{k,l})^2\leq\delta_u$ on the white noise, we get:
\begin{align}\label{l3_sp_2}
    (\boldsymbol{\xi}\otimes I_{n_x})^\top \mathbf{Q}_{A,l} (\boldsymbol{\xi}\otimes I_{n_x}) \preceq 0, \quad (\boldsymbol{\xi}\otimes I_{n_x})^\top \mathbf{Q}_{B,l} (\boldsymbol{\xi}\otimes I_{n_x}) \preceq 0
\end{align}
where
\begin{subequations}
\begin{align}
\mathbf{Q}_{A,l} &= \mathrm{blkdiag}\left(
    -\delta_x I_{n_x}, \quad 
    \underbrace{0_{n_x}, \dots, 0_{n_x}}_{l-1 \text{ blocks}}, \quad
    \underbrace{I_{n_x}}_{(1+l, 1+l) \text{ block}}, \quad
    \underbrace{0_{n_x}, \dots, 0_{n_x}}_{m_1+m_2-l \text{ blocks}}
\right)\\
\mathbf{Q}_{B,l} &= \mathrm{blkdiag}\left(
    -\delta_u I_{n_x}, \quad 
    \underbrace{0_{n_x}, \dots, 0_{n_x}}_{m_1+l-1 \text{ blocks}}, \quad
    \underbrace{I_{n_x}}_{(1+m_1+l, 1+m_1+l) \text{ block}}, \quad
    \underbrace{0_{n_x}, \dots, 0_{n_x}}_{m_2-l \text{ blocks}}
\right)
\end{align}
\end{subequations}
Now using the lossy S-procedure (Section 2.6.3 in~\cite{boydLinearMatrixInequalities1994}) for multiple constraints, we know that if \cref{l3_sp_2} holds, then \cref{l3_sp_1} holds, if $\exists \; \tau_{l} \geq 0, \; \forall l \in \{1, \dots, m_1\},$ and $\eta_{l} \geq 0 \; \forall l \in \{1, \dots, m_2\}$ such that:
\begin{align}\label{l3_sp_3}
    \mathbf{M} - \sum_{l=1}^{m_1} \tau_{l} \mathbf{Q}_{A,l} - \sum_{l=1}^{m_2} \eta_{l} \mathbf{Q}_{B,l} \preceq 0
\end{align}
which can be written as:
\begin{align}
    (\mathscr{S}_k(i) + \sum_{l=1}^{m_1} \tau_{l} \mathbf{Q}_{A,l} + \sum_{l=1}^{m_2} \eta_{l} \mathbf{Q}_{B,l} )-\mathbf{\Psi}^\top S_{k+1}^{-1}(j) \mathbf{\Psi} \succeq 0
\end{align}
Using Schur complement, we get:
\begin{align}\label{l3_sp_4}
    \begin{bmatrix}
        \mathscr{S}_k(i) + \sum_{l=1}^{m_1} \tau_{l} \mathbf{Q}_{A,l} + \sum_{l=1}^{m_2} \eta_{l} \mathbf{Q}_{B,l} & \ast \\
        \mathbf{\Psi} & S_{k+1}(j)
    \end{bmatrix}
    \succeq 0
\end{align}
Define $\tau_{\text{total}} = \delta_x \sum_{l=1}^{m_1} \tau_{l} + \delta_u \sum_{l=1}^{m_2} \eta_{l}$. Now, expanding the matrix in \cref{l3_sp_4}, we get:
\begin{equation*}
    \begin{bsmallmatrix}
        S_k(i) - \tau_{\text{total}} I_{n_x} & \ast & \dots & \ast & \ast & \dots & \ast & \ast \\
        0_{n_x} & \tau_{1} I_{n_x} & \dots & \ast & \ast & \dots & \ast & \ast \\
        \vdots & \vdots & \ddots & \vdots & \vdots & \ddots & \vdots & \vdots \\
        0_{n_x} & 0_{n_x} & \dots & \tau_{m_1} I_{n_x} & \ast & \dots & \ast & \ast \\
        0_{n_x} & 0_{n_x} & \dots & 0_{n_x} & \eta_{1} I_{n_x} & \dots & \ast & \ast \\
        \vdots & \vdots & \ddots & \vdots & \vdots & \ddots & \vdots & \vdots \\
        0_{n_x} & 0_{n_x} & \dots & 0_{n_x} & 0_{n_x} & \dots & \eta_{m_2} I_{n_x} & \ast \\
        A_k(i) S_k(i) + B_k(i) Y_k(i) & A_{k,1}(i) S_k(i) & \dots & A_{k,m_1}(i) S_k(i) & B_{k,1}(i) Y_k(i) & \dots & B_{k,m_2}(i) Y_k(i) & S_{k+1}(j)
    \end{bsmallmatrix}
    \succeq 0 \qedhere
\end{equation*}
\end{proof}
\Cref{lemma_invariant_sets} provides an LMI-based condition, which can be used to search for controllers that render the sequence $\{\mathscr{E}_k\}$ $T$-periodic robust positively invariant. Combined with \cref{lemma_performance_noise}, we can derive an LMI-based condition for a controller, which is MSS, stays in the RPI sets, and satisfies path constraints with probability 1 while upper bounding the cost. The satisfaction of path constraints with probability one is guaranteed by the RPI sequence provided by \cref{lemma_invariant_sets}, which holds the state within the invariant ellipsoids for every admissible bounded-noise realization and mode sequence. We note that \cref{lemma_invariant_sets} finds this RPI sequence using Lyapunov matrices $P_k(i)$, which are also used in \cref{lemma_performance_noise}. This allows us to solve $\cref{problem:l_socp}$ using the following theorem:
\begin{theorem}\label{LMI1}
    The solution to the following semidefinite program (SDP) synthesizes a stabilizing controller which satisfies the path constraints in \Cref{problem:l_socp}, while minimizing an upper bound $\beta$ on the quadratic cost function in \cref{prob1_J}. The controller is given by $K_k(i)=\bar{Y}_k(i)\bar{S}_k^{-1}(i)$, where $\bar{Y}_k(i)$ and $\bar{S}_k(i)$ are the optimal values of the following SDP:
    \begin{subequations}
    \begin{alignat}{2}
        \min_{\substack{\beta,\,Y_k(i),\,S_k(i)\\[2pt] \tau_{l},\,\eta_{l}}} \quad & \beta \notag\\
        \text{s.t.} \quad & \cref{l3_t1}, \cref{l3_t2}, \cref{MSS_LMI} \notag\\
        & \begin{bsmallmatrix}
            u_{max}(i)^2 I &  \ast \\
            Y_k(i)^{\top} & S_k(i)
        \end{bsmallmatrix} \succeq 0 \label{t4}\\
        & I- H_k(i) S_k(i) H_k^\top(i) \succeq 0 \label{t5}
    \end{alignat}
    \end{subequations}
    $\forall\, k\in\mathbb{Z}_{0:T-1}$ and $\forall\, i,j\in\mathcal{N}$, with $S_T(i)\triangleq S_0(i)$,
    where $H_k(i)=x_{b,k}^{-1}\begin{bsmallmatrix}I_3 & 0_3\end{bsmallmatrix}$ is the state-constraint matrix in $\mathcal{X}=\{\delta\mathbf{x}:\|H_k(i)\delta\mathbf{x}\|\leq 1\}$, $S_k\in\mathbb{M}^N$ and $Y_k(i)\in\mathbb{R}^{n_u\times n_x}$. 
    \end{theorem}
\begin{proof}
    We show that each of the desired properties are satisfied by the SDP defined in \Cref{LMI1}:
    
    \emph{Invariance.} Constraints \cref{l3_t1} and \cref{l3_t2} are exactly the conditions of \cref{lemma_invariant_sets}. Hence, with $P_k(i)=S_k(i)^{-1}$, the sequence $\mathscr{E}_k=\{\delta\mathbf{x}:\delta\mathbf{x}^\top P_k(i)\delta\mathbf{x}\leq 1,\;\forall i\in\mathcal{N}\}$ is $T$-periodic RPI in the sense of \cref{def:rpi} and satisfies $\mathcal{D}_0\subseteq\mathscr{E}_0$. Since $\mathbb{P}(\delta\mathbf{x}_0\in\mathcal{D}_0)=1$ and the invariance of \cref{def:rpi} holds for every admissible noise realization, we have
    \begin{align}\label{eq:einvariance_wp1}
        \mathbb{P}(\delta\mathbf{x}_k\in\mathscr{E}_k)=1, \qquad \forall k\in\mathbb{Z}_{\geq0}.
    \end{align}
    We now show that if $\mathscr{E}_k$ is a feasible sequence of RPI sets, it implies satisfaction of the state and control constraints.

    \emph{Control constraint.} Applying the Schur complement to \cref{t4} and using $Y_k(i)=K_k(i)S_k(i)$ gives
    \begin{align}
        u_{max}(i)^2 I - Y_k(i) S_k(i)^{-1} Y_k(i)^\top \succeq 0
        \;\;\Longleftrightarrow\;\; K_k(i)S_k(i)K_k(i)^\top \preceq u_{max}(i)^2 I.
    \end{align}
    Hence, for any $\delta\mathbf{x}_k\in\mathscr{E}_k$, the commanded control $\mathbf{u}_k=K_k(i)\delta\mathbf{x}_k$ satisfies
    \begin{align}
        \|\mathbf{u}_k\|^2=\delta\mathbf{x}_k^\top K_k(i)^\top K_k(i)\delta\mathbf{x}_k
        \leq \max_{\delta\mathbf{x}^\top P_k(i)\delta\mathbf{x}\leq 1}\delta\mathbf{x}^\top K_k(i)^\top K_k(i)\delta\mathbf{x}
        = \lambda_{\max}\!\left(K_k(i)S_k(i)K_k(i)^\top\right)\leq u_{max}(i)^2,
    \end{align}
    so that $\mathbf{u}_k\in\mathcal{U}$ whenever $\delta\mathbf{x}_k\in\mathscr{E}_k$.

    \emph{State constraint.} With the state-constraint set written as $\mathcal{X}=\{\delta\mathbf{x}:\|H_k(i)\delta\mathbf{x}\|\leq 1\}$, from \cref{t5} we have
    \begin{align}
        I - H_k(i)S_k(i)H_k(i)^\top \succeq 0
        \;\;\Longleftrightarrow\;\; H_k(i)^\top H_k(i)\preceq P_k(i),
    \end{align}
    which is precisely the ellipsoid-containment condition $\mathscr{E}_k\subseteq\mathcal{X}$. Thus $\delta\mathbf{x}_k\in\mathscr{E}_k\Rightarrow\delta\mathbf{x}_k\in\mathcal{X}$.

    Combining these two implications with \cref{eq:einvariance_wp1} gives $\mathbb{P}(\delta\mathbf{x}_k\in\mathcal{X})=1$ and $\mathbb{P}(\mathbf{u}_k\in\mathcal{U})=1$ for all $k$, and for all possible mode and WSS noise realizations; that is, the state and control path constraints are satisfied with probability one. Finally, \cref{MSS_LMI} certifies, through \cref{cor:performance_lmi} and \cref{lemma_performance_noise}, mean square stability and the upper bound $\beta$ on the quadratic cost, which completes the proof.
\end{proof}
\section{Numerical Experiments}\label{sec:results}
\subsection{Problem Set-up}
\subsubsection{Stationkeeping Orbit}
We demonstrate the proposed controller on two baseline NRHOs: the Earth-Moon 9:2 southern $L_2$ NRHO baselined for the Gateway mission, and a Saturn-Enceladus $L_1$ NRHO studied for Enceladus orbiter missions~\cite{mackenzieEnceladusOrbilanderMission2021,booneHaloOrbitConfigurations2026}. Initial conditions for both reference periodic orbits are obtained via differential correction. \Cref{tab:orbit_params} lists the parameters, and initial conditions used for each orbit.
\begin{table}[hbt!]
\caption{\label{tab:orbit_params} Baseline NRHO orbit parameters}
\centering
\begin{tabular}{lccc}
\hline
Parameter & Units & Earth-Moon 9:2 NRHO & Saturn-Enceladus $L_1$ NRHO \\\hline
Mass ratio, $\mu$ & ND & $1.215058\times10^{-2}$ & $1.901150\times10^{-7}$ \\
Length unit, $l^*$ & km & 389{,}703 & 238{,}400 \\
Time unit, $t^*$ & days & 4.4327 & 0.2187 \\
Orbital period, $T$ & ND (days) & 1.502088 (6.658) & 2.282109 (0.499) \\
$x_0$ & ND & 1.021337 & 0.997411 \\
$z_0$ & ND & $-0.181621$ & 0.004875 \\
$\dot{y}_0$ & ND & $-0.101760$ & 0.005464 \\\hline
\end{tabular}
\end{table}
\subsubsection{Failure Modeling}
For both orbits, missed thrust events are modeled using the exponential failure/recovery framework of \cref{expfit}, parameterized by $TTF_{90}$ and $TTR_{95}$. \Cref{tab:failure_params} summarizes these simulation parameters. Every sample is initialized in the failure mode, so the controller must recover from an outage that is already in progress at $t=0$. Both impulsive cases use the three-mode backup framework of \Cref{sec:backup} to offer additional contingency TCM opportunities.
\begin{table}[hbt!]
\caption{\label{tab:failure_params} Failure and recovery modeling parameters}
\centering
\begin{tabular}{llccc}
\hline
 & & \multicolumn{2}{c}{Earth-Moon 9:2 NRHO} & Saturn-Enceladus $L_1$ NRHO \\\cline{3-4}
Parameter & Unit & Impulsive & Low-Thrust & Impulsive \\\hline
Nodes per orbit, $n_T$ & --- & 41 & 41 & 51 \\
$TTF_{90}$ & orbits (days) & 4 (26.63) & 4 (26.63) & 4 (2.00) \\
$TTR_{95}$ & orbits (days) & 0.1 (0.666) & 0.1 (0.666) & 0.1 (0.050) \\
Initial mode distribution, $\boldsymbol{\pi}_0$ & --- & $[0,\,1,\,0]$ & $[0,\,1]$ & $[0,\,1,\,0]$
\\\hline
\end{tabular}
\end{table}
\subsubsection{Controller Design Parameters}
The LMI synthesis of \Cref{problem:l_socp} uses a minimum-effort cost, $Q(i)=\mathbf{0}$ and $R(i)=I_3$. The initial dispersion error at apolune is modeled as a uniform distribution over $\mathcal{D}_0$, a cuboid centered at the origin in position and velocity space bounded by $ic_x$ and $ic_v$. The state and control path constraints are enforced through $x_{b,k}$ and $u_{max}$ respectively. \Cref{tab:controller_params} lists all the parameters used in the numerical examples. Note that the backup TCM nodes are placed immediately after each nominal TCM node, and placed consecutively but do not cover the entire arc, which allows for sufficient time gap between two successive TCM burns even in case of MTE recovery. In most space missions, TCM commands are often executed after performing orbit determination (OD,) so each TCM must have sufficient time after the previous TCM to perform OD and prepare for the next burn. By not covering the whole arc with backup TCM nodes, we ensure that no TCMs will be performed in quick succession under any MTE realization. \Cref{fig:tcm_schedule} illustrates the resulting schedule for the Earth-Moon impulsive case, showing the five nominal TCM nodes (shown in blue) per orbit with 5 backup TCM opportunities per arc, shown as green diamonds. A representative scenario is shown in the schematic where the first two TCM burns are executed nominally but a MTE occurs after the second nominal TCM burn, leading to a missed nominal TCM burn. The spacecraft recovers soon after, and is immediately able to execute a backup TCM burn to compensate for the missed TCM burn. In a nominal case with no backup TCM opportunities, the spacecraft would have to wait until the next planned location to execute the nominal TCM, which would have led to increased fuel costs and potentially divergence from the baseline orbit.
\begin{figure}[htbp!]
    \centering
    \resizebox{\linewidth}{!}{%
  \begin{tikzpicture}[
      x=0.34cm, y=1cm,
      font=\normalsize,
      opp nom/.style  ={circle,  draw=nomblue, fill=nomblue!16, line width=0.5pt, inner sep=0pt, minimum size=2.5mm},
      fire nom/.style ={circle,  draw=nomblue!55!black, fill=nomblue, line width=0.5pt, inner sep=0pt, minimum size=2.9mm},
      opp bk/.style   ={diamond, draw=bkgreen, fill=bkgreen!16, line width=0.5pt, inner sep=0pt, minimum size=2.4mm},
      fire bk/.style  ={diamond, draw=bkgreen!55!black, fill=bkgreen, line width=0.5pt, inner sep=0pt, minimum size=2.8mm},
      miss/.style     ={circle,  draw=black!45, fill=black!12, line width=0.5pt, inner sep=0pt, minimum size=3.6mm,
                        font=\scriptsize, text=failred},
      lane/.style     ={line width=0.6pt, black!25},
      hist/.style     ={line width=1.3pt},
      jump/.style     ={line width=0.9pt, black!35},
      lab/.style      ={font=\normalsize, inner sep=1.5pt},
    ]
    \def\YN{2.05}   
    \def\YB{0.25}   
    \def\YF{1.15}   
    \def\YA{-0.35}  

    \foreach \k in {6.5,14.5,22.5,30.5,38.5}{
      \fill[black!6] (\k,-0.10) rectangle (\k+2,2.48);
    }

    \foreach \y in {\YN,\YB,\YF}{ \draw[lane] (0.4,\y) -- (41.6,\y); }
    \node[lab, anchor=east, nomblue!70!black] at (-0.6,\YN) {nominal, $i=1$};
    \node[lab, anchor=east, bkgreen!70!black] at (-0.6,\YB) {backup, $i=3$};
    \node[lab, anchor=east, failred!70!black] at (-0.6,\YF) {failure, $i=2$};

    \draw[hist, nomblue] (0.5,\YN) -- (12.5,\YN);
    \draw[jump] (12.5,\YN) -- (12.5,\YF);
    \draw[hist, failred] (12.5,\YF) -- (19.5,\YF);
    \draw[jump] (19.5,\YF) -- (19.5,\YB);
    \draw[hist, bkgreen] (19.5,\YB) -- (20.5,\YB);
    \draw[jump] (20.5,\YB) -- (20.5,\YN);
    \draw[hist, nomblue] (20.5,\YN) -- (41.5,\YN);

    \foreach \k in {1,9,25,33,41}{ \node[fire nom] at (\k,\YN) {}; }
    \node[miss] at (17,\YN) {$\times$};

    \foreach \k in {2,3,4,5,6,10,11,12,13,14,18,19,21,22,26,27,28,29,30,34,35,36,37,38}{
      \node[opp bk] at (\k,\YB) {};
    }
    \node[fire bk] at (20,\YB) {};

    \draw[decorate, decoration={brace, amplitude=4pt}, black!55, line width=0.6pt]
        (0.6,2.66) -- (6.4,2.66)
        node[midway, above=4pt, lab, black!60] {TCM opportunities: $1$ nominal $+$ $5$ backup};
    \draw[decorate, decoration={brace, amplitude=4pt}, black!55, line width=0.6pt]
        (22.6,2.66) -- (24.4,2.66)
        node[midway, above=4pt, lab, black!60] {OD and burn-planning window};

    \draw[black!55, line width=0.6pt] (0.5,\YA) -- (41.6,\YA);
    \foreach \k in {1,...,41}{ \draw[black!35, line width=0.4pt] (\k,\YA) -- (\k,\YA-0.06); }
    \node[lab, anchor=west, black!70] at (42.0,\YA) {$t$};

    \node[opp nom]                                 (lgA) at (1.4,-1.15) {};
    \node[opp bk,  anchor=west] (lgB) at ([xshift=1.6mm]lgA.east) {};
    \node[lab,     anchor=west] (lgC) at ([xshift=1.8mm]lgB.east) {TCM opportunity};
    \node[fire nom,anchor=west] (lgD) at ([xshift=6mm]  lgC.east) {};
    \node[fire bk, anchor=west] (lgE) at ([xshift=1.6mm]lgD.east) {};
    \node[lab,     anchor=west] (lgF) at ([xshift=1.8mm]lgE.east) {executed burn};
    \node[miss,    anchor=west] (lgG) at ([xshift=6mm]  lgF.east) {$\times$};
    \node[lab,     anchor=west] (lgH) at ([xshift=1.8mm]lgG.east) {missed TCM};
    \coordinate (lgI) at ([xshift=6mm]lgH.east);
    \draw[hist, nomblue] (lgI) -- ++(3mm,0);
    \draw[hist, failred] ([xshift=3mm]lgI) -- ++(3mm,0);
    \draw[hist, bkgreen] ([xshift=6mm]lgI) -- ++(3mm,0);
    \node[lab, anchor=west] at ([xshift=10.8mm]lgI) {mode history $r_k$};
  \end{tikzpicture}%
    }
    \caption{TCM node schedule over one revolution of the Earth-Moon 9:2 NRHO: one lane per mode with
    the TCM opportunities it offers, shaded bands for the coast arcs, and the mode history.}
    \label{fig:tcm_schedule}
  \end{figure}
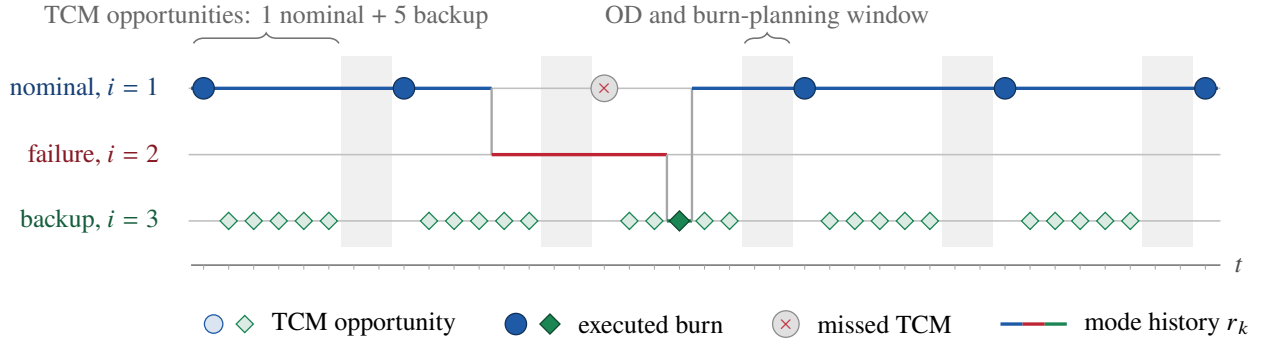

  In the Saturn Enceladus case, we also taper the state bound ($x_{b,k}$) to ensure the spacecraft is above the Enceladus surface at perilune, while offering less tight bounds near apolune. For a given set of taper parameters $x_b^{a}$ and $x_b^{p}$, the state bound $x_{b,k}$ at time $k$ which is $d_{E,k}$ distance away from Enceladus, can be calculated as:
  \begin{equation}
      x_{b,k}= x_b^{p} + \bigl(x_b^{a} - x_b^{p}\bigr)\cdot\frac{d_{E,k} - d_{E,p}}{d_{E,a} - d_{E,p}}
  \end{equation}
  where $d_{E,p}$ and $d_{E,a}$ are the distances from the perilune and apolune to Enceladus.

The WSS white noise terms $w^{\{x\}}_{k,l}$ and $w^{\{u\}}_{k,l}$ are modeled as symmetric triangular random variables on $[-a,a]$. Their compact support makes the almost-sure bounds of \Cref{lemma_invariant_sets} valid, and since this distribution has variance $a^2/6$, the unit second moment assumed in \cref{MJLS} fixes $\delta_x=\delta_u=a^2=6$. The pointing gain is calibrated from $\theta_{p,\mathrm{rms}}$, the RMS pointing error angle, which for small deflections equals the perpendicular execution error as a fraction of the commanded burn magnitude. The pointing error RMS and max angle are related to $\epsilon_p$ as:
\begin{equation}\label{eq:eps_p}
    \epsilon_p=\frac{\theta_{p,\mathrm{rms}}}{\sqrt{2}},\qquad\qquad \theta_{p,\max}=\epsilon_p\sqrt{3\delta_u}=\theta_{p,\mathrm{rms}}\sqrt{\frac{3\delta_u}{2}},
\end{equation}
where $\theta_{p,\max}$ is the worst case pointing error. For the magnitude error, $\epsilon_m$ represents the RMS relative magnitude error, with worst case relative magnitude error $\epsilon_m\sqrt{\delta_u}$.
\begin{remark}
Enforcing \Cref{lemma_invariant_sets} exactly requires the set invariance constraint to hold for every mode pair, including the failure mode self-transition $(i,j)=(2,2)$, for which $B_k(2)=0$. Composing it over one period requires the uncontrolled dynamics to admit a $T$-periodic RPI sequence of sets, which does not always hold. To address this, we relax the invariance set constraint for $(i,j)=(2,2)$ mode transition as $C_i\succeq-\epsilon I$ where $\epsilon$ is a small positive number chosen by the user and $C_i$ is the LMI invariance condition in \Cref{l3_t2}. Nonetheless, given the low probability of an MTE lasting an entire orbit, the path constraints are still satisfied with a probability close to unity.
\end{remark}

\begin{table}[hbt!]
\caption{\label{tab:controller_params} Controller design parameters}
\centering
\small
\begin{tabular}{lccc}
\hline
 & \multicolumn{2}{c}{Earth-Moon 9:2 NRHO} & Saturn-Enceladus $L_1$ NRHO \\\cline{2-3}
Parameter & Impulsive & Low-Thrust & Impulsive \\\hline
Control bound, $u_{max}$ & 5 m/s & $2\times10^{-5}$ m/s$^2$ & 15 m/s \\
State bound, $x_{b,k}$ (km) & 1000 & 1000 & 2000/30 (tapered) \\
IC position bound, $ic_x$ (km) & 100 & 100 & 50 \\
IC velocity bound, $ic_v$ (km/s) & $10^{-6}$ & $10^{-6}$ & $10^{-6}$ \\
Nominal TCM nodes per orbit & 5 & all nodes (continuous) & 7 \\
Backup TCM nodes per coasting arc & 5 & --- & 2 \\
State-uncertainty step-off (km) & 50/2.5 (tapered) & 50/2.5 (tapered) & 40/2.0 (tapered) \\
Pointing error, $\theta_{p,\mathrm{rms}}$ / $\theta_{p,\max}$ (deg) & 1.0 / 3.0 & 1.0 / 3.0 & 1.0 / 3.0 \\
Magnitude error, $\epsilon_m$ / $\epsilon_m\sqrt{\delta_u}$ (\%) & 1.0 / 2.45 & 1.0 / 2.45 & 1.0 / 2.45 \\
WSS-noise bounds, $\delta_x,\,\delta_u$ & 6, 6 & 6, 6 & 6, 6 \\\hline
\end{tabular}
\end{table}

\subsection{Simulation Results}

For each case, we use \Cref{LMI1} to design control policies, and then evaluate them through nonlinear Monte Carlo simulations. In each simulation, the initial dispersion error is drawn uniformly from the cuboid $\mathcal{D}_0$ defined by $ic_x$ and $ic_v$, and mode switching follows the transition probability matrix of \cref{tpm_formula} for the low-thrust case and \cref{tpm_formula_backup} for the two impulsive cases. The executed control includes the multiplicative pointing and magnitude errors of \Cref{tab:controller_params} and is saturated at $u_{max}$. The WSS multiplicative noise scalars for MEE are drawn from the symmetric triangular distribution described above. We now discuss the results of the three cases below: 

\subsubsection{Earth-Moon 9:2 NRHO: Impulsive Stationkeeping}

The impulsive Earth-Moon stationkeeping simulation is performed for 50 days ($\approx 7.5$ revolutions) with five nominal TCM opportunities per revolution and five backup nodes per coasting arc. \Cref{fig:em_imp_3d} shows the resulting Monte Carlo sample trajectories against the reference orbit (with deviations amplified $20$ times for visual clarity), and the corresponding burns are shown in \Cref{fig:em_imp_u}. The Monte Carlo trajectories show a clear trend of stabilizing the spacecraft under multiple MTEs, MEE, linearization error and initial dispersion error. Furthermore, the path constraints are all respected for each sample, as shown in \Cref{fig:em_imp_dev}, which plots the position deviation $\|\delta \mathbf{r}\|$ against the max deviation bound $x_{b,k}$. As expected, the controller stabilizes the spacecraft, and while deviations rise slightly post-perilune passage, there is an overall decreasing trend in deviation from baseline orbit. \Cref{fig:em_imp_state} shows the state deviation history for both position and velocity (on a symmetric-log axis, i.e. logarithmic in magnitude, linear near zero, and sign-preserving), which exhibit similar trends of long-term decreasing behavior towards convergence to the baseline orbit despite multiple sources of uncertainties. 

\begin{figure}[htb!]
\centering
\begin{subfigure}[b]{0.42\textwidth}
\centering
\includegraphics[width=\textwidth]{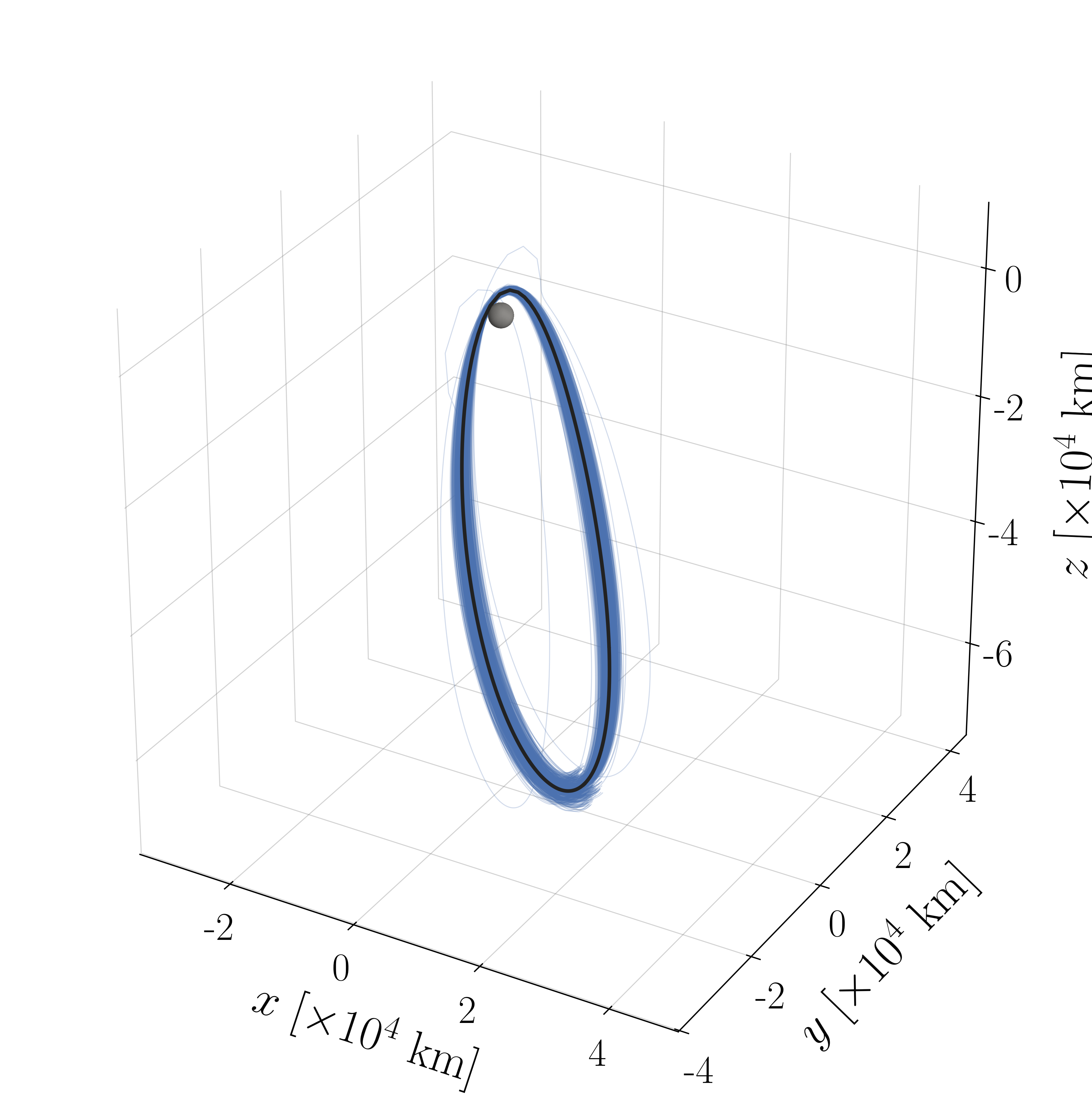}
\caption{Monte Carlo ensemble.}
\label{fig:em_imp_3d}
\end{subfigure}
\hfill
\begin{subfigure}[b]{0.56\textwidth}
\centering
\includegraphics[width=\textwidth]{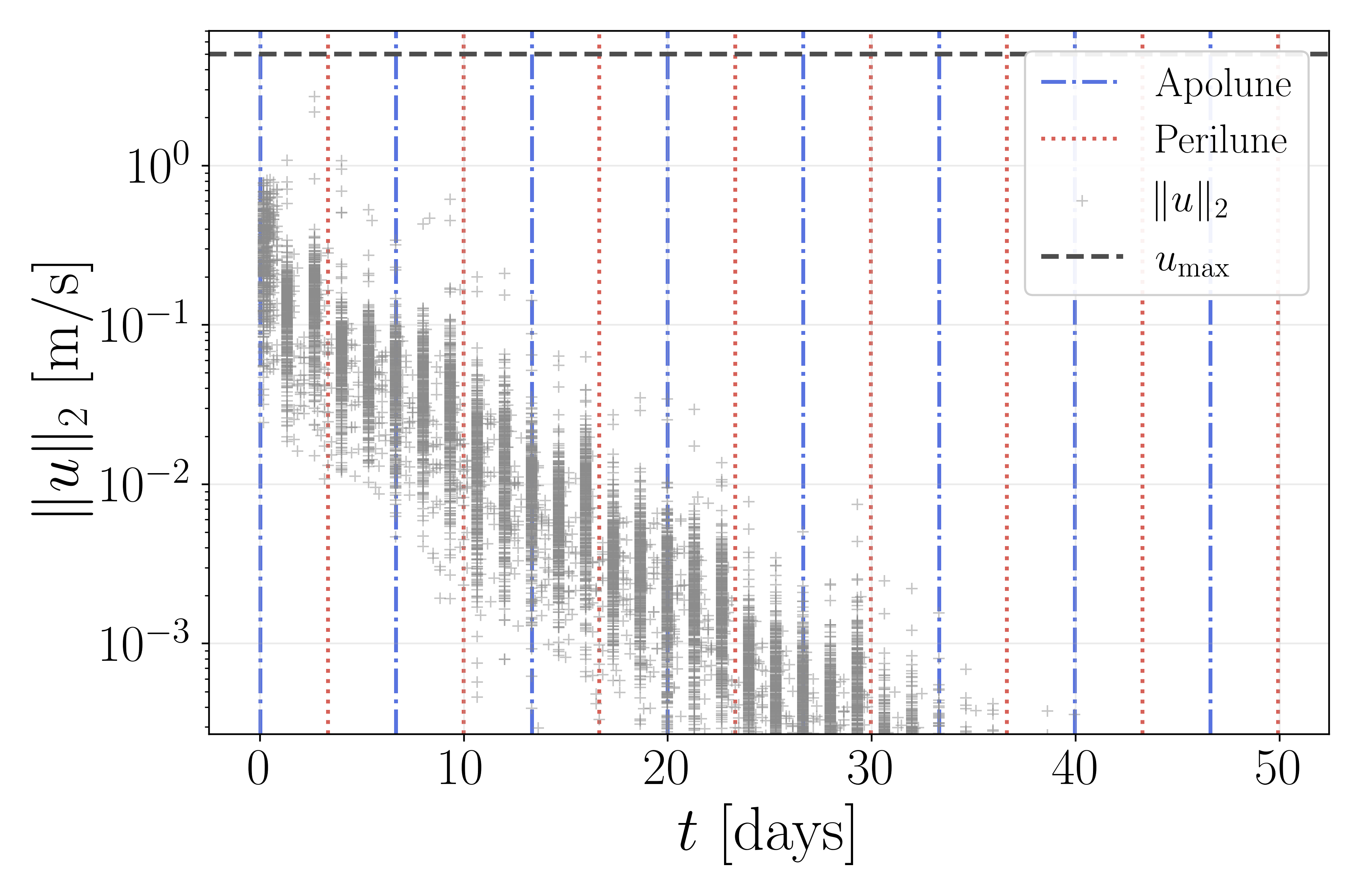}
\caption{Control history.}
\label{fig:em_imp_u}
\end{subfigure}
\caption{Earth-Moon impulsive: (a) 500 Monte Carlo trajectories with deviations magnified $20\times$; (b) the corresponding control histories, with apolune and perilune passages marked and the control bound $u_{max}$ dashed.}
\label{fig:em_imp_ensemble}
\end{figure}

\begin{figure}[htb!]
\centering
\begin{subfigure}[b]{0.46\textwidth}
\centering
\includegraphics[width=\textwidth]{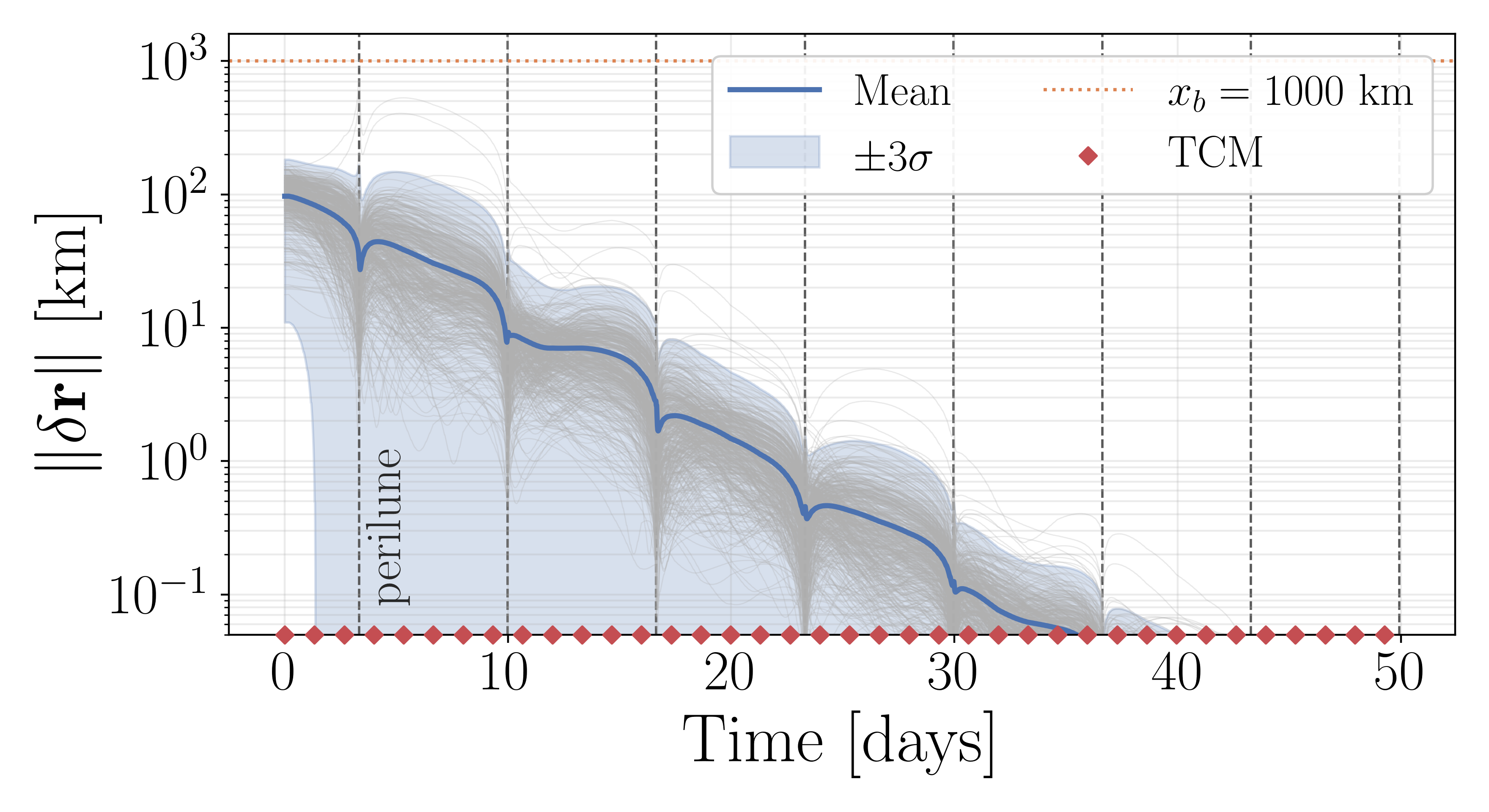}
\caption{Position deviation.}
\label{fig:em_imp_dev}
\end{subfigure}
\hfill
\begin{subfigure}[b]{0.52\textwidth}
\centering
\includegraphics[width=\textwidth]{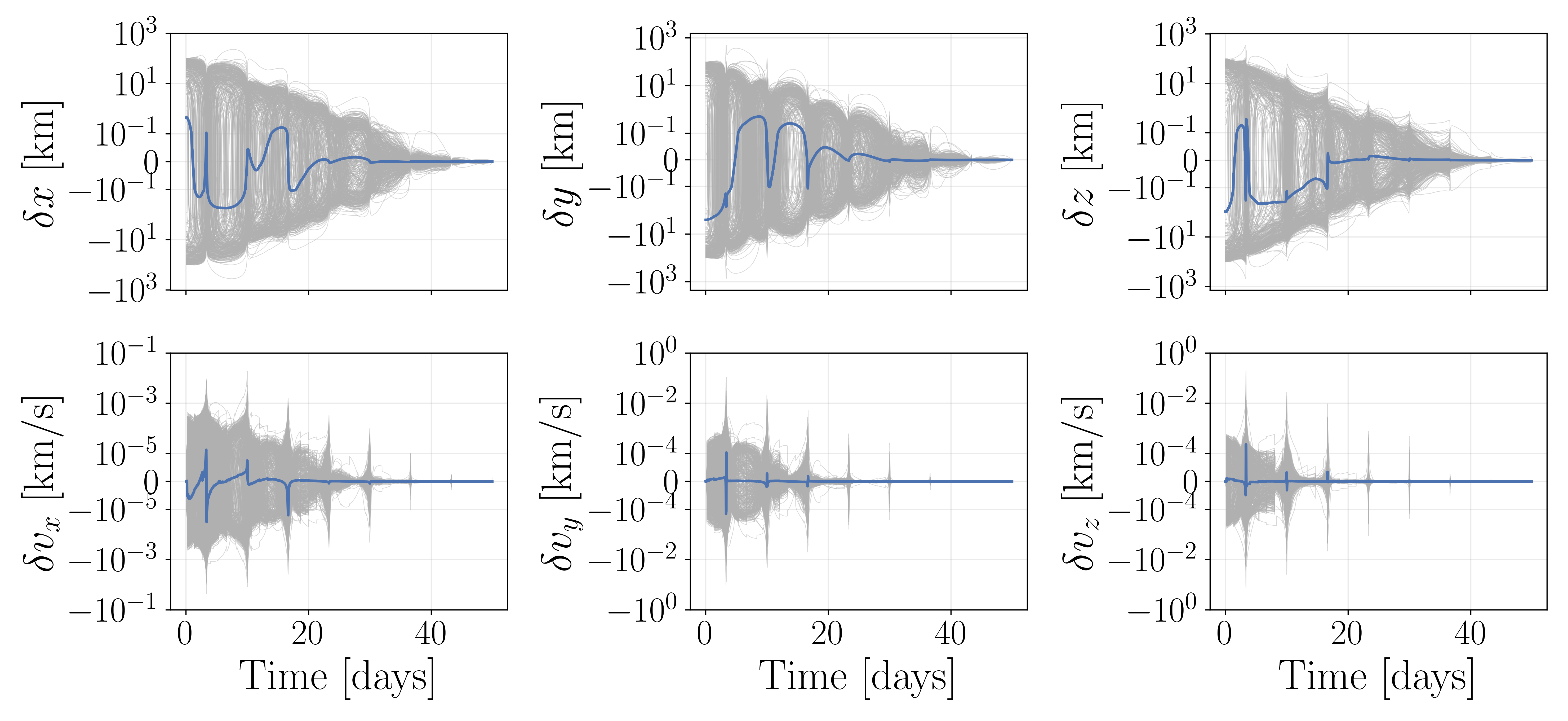}
\caption{State history.}
\label{fig:em_imp_state}
\end{subfigure}
\caption{Earth-Moon impulsive: (a) position deviation of all samples (gray), mean in blue, a shaded $\pm 3\sigma$ band, the state bound $x_{b,k}$, perilune passages, and TCM nodes; (b) per-component deviations, mean in blue.}
\label{fig:em_imp_deviation}
\end{figure}

\subsubsection{Earth-Moon 9:2 NRHO: Low-Thrust Stationkeeping}

Low-thrust engines execute long continuous burns over each arc, which makes them more susceptible to MTEs. In this case, each Monte Carlo sample is simulated for 80 days, which is approximately $12$ orbit revolutions for the Earth-Moon 9:2 NRHO. \Cref{fig:em_lt_3d} shows the Monte Carlo trajectories (with deviations amplified $20$ times), and the corresponding low thrust control history is shown in \Cref{fig:em_lt_u}. Similar to the impulsive case, we also plot the position deviation with the state bound $x_{b,k}$ in \cref{fig:em_lt_dev} and the corresponding state histories are plotted in \cref{fig:em_lt_state} (on a symmetric-log axis). We note the similar trend of increased deviations post perilune passage, and the long-term monotonic trend towards convergence, despite multiple sources of uncertainty. 

\begin{figure}[htb!]
\centering
\begin{subfigure}[b]{0.42\textwidth}
\centering
\includegraphics[width=\textwidth]{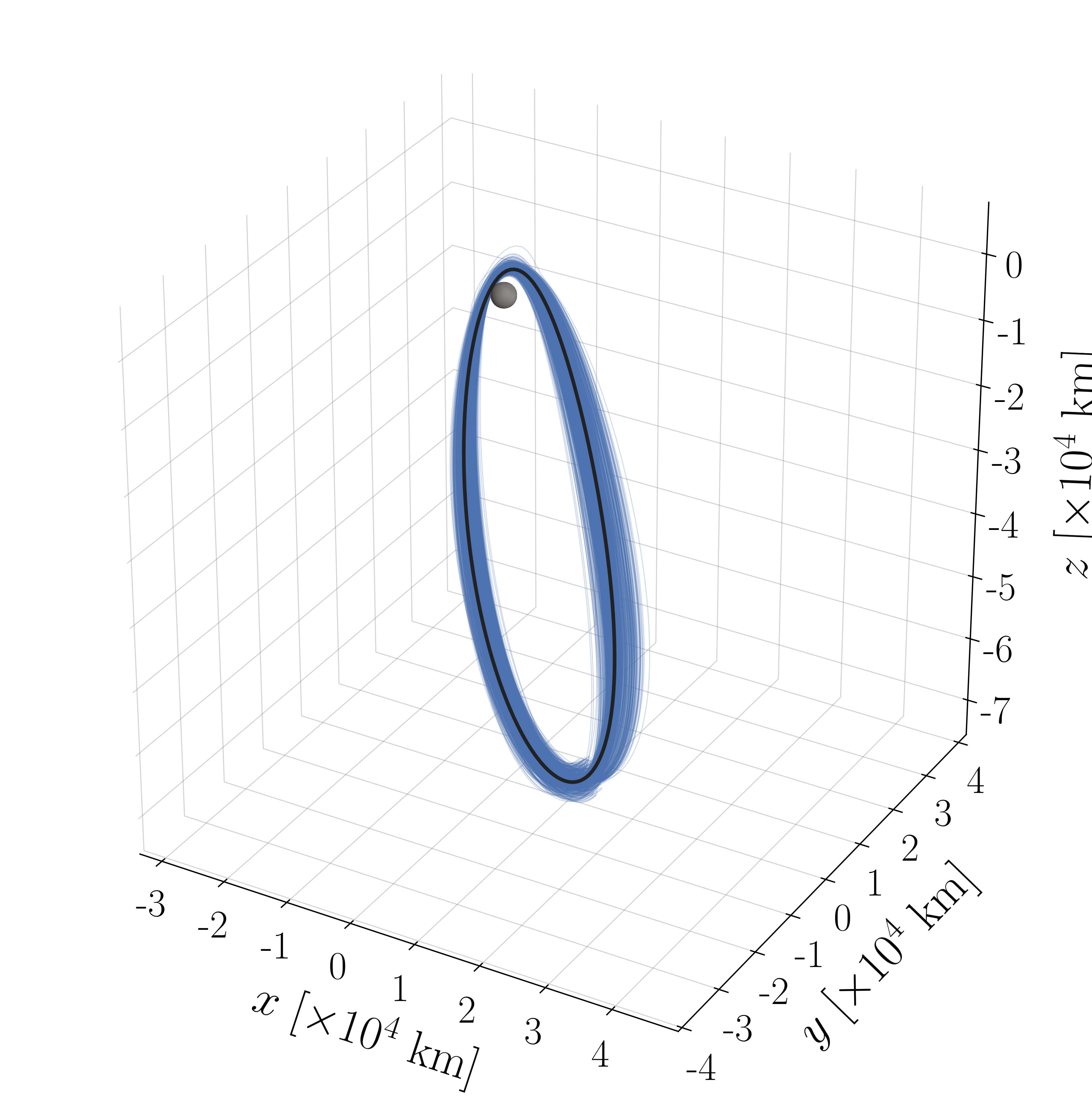}
\caption{Monte Carlo ensemble.}
\label{fig:em_lt_3d}
\end{subfigure}
\hfill
\begin{subfigure}[b]{0.56\textwidth}
\centering
\includegraphics[width=\textwidth]{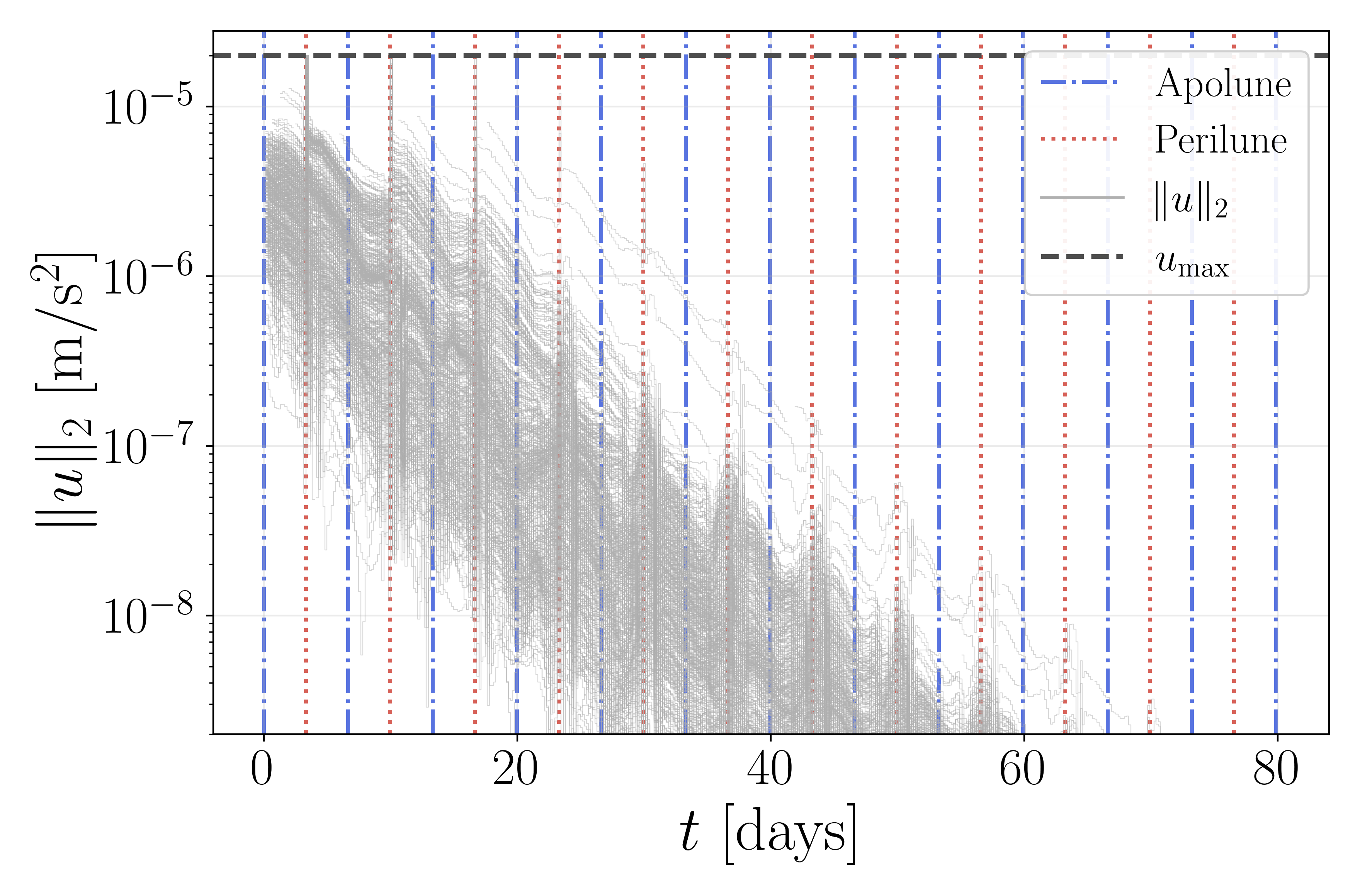}
\caption{Control history.}
\label{fig:em_lt_u}
\end{subfigure}
\caption{Earth-Moon low-thrust: (a) 500 Monte Carlo trajectories with deviations magnified $20\times$; (b) the corresponding control histories, with apolune and perilune passages marked and the control bound $u_{max}$ dashed.}
\label{fig:em_lt_ensemble}
\end{figure}

\begin{figure}[htb!]
\centering
\begin{subfigure}[b]{0.46\textwidth}
\centering
\includegraphics[width=\textwidth]{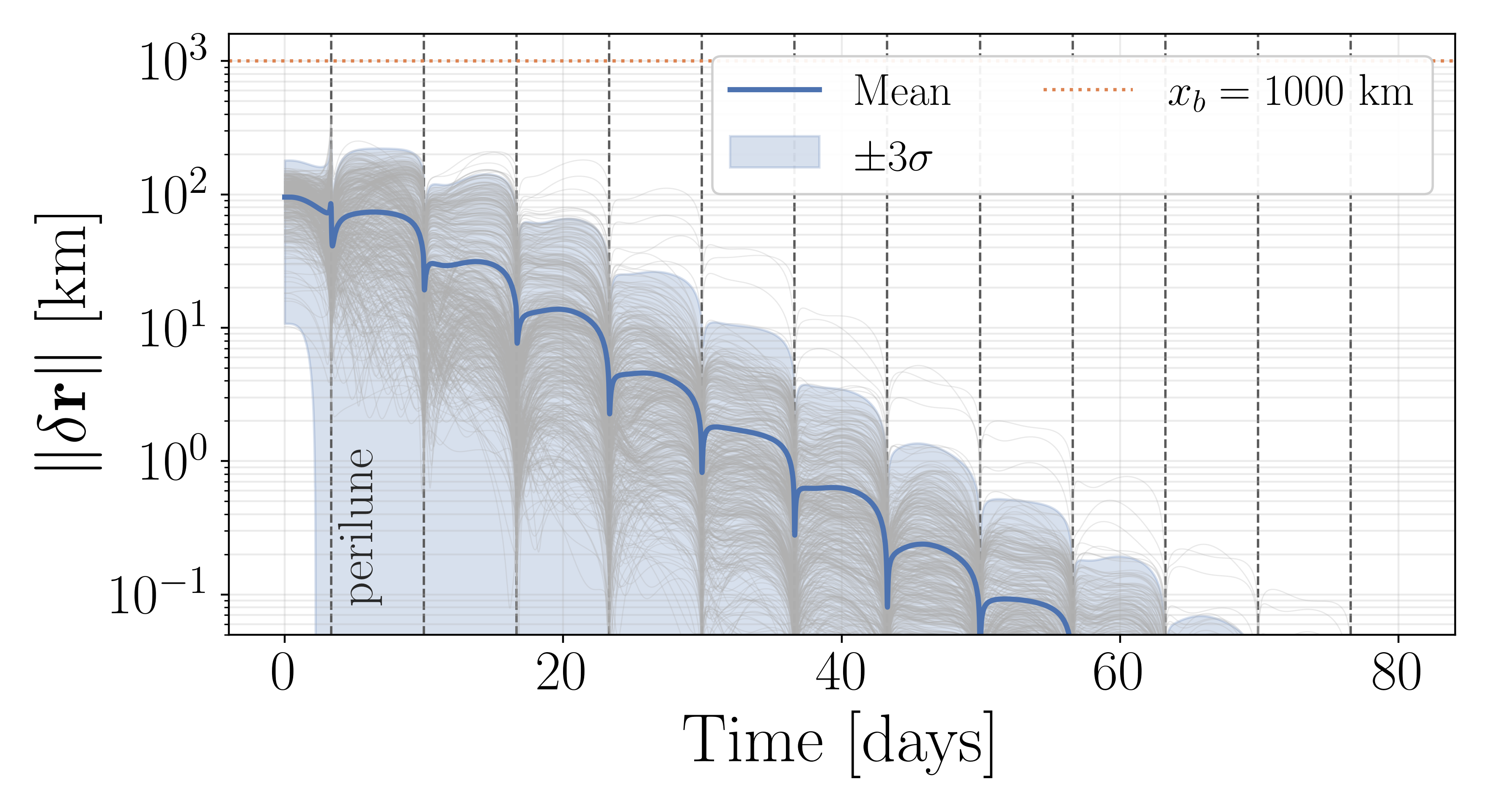}
\caption{Position deviation.}
\label{fig:em_lt_dev}
\end{subfigure}
\hfill
\begin{subfigure}[b]{0.52\textwidth}
\centering
\includegraphics[width=\textwidth]{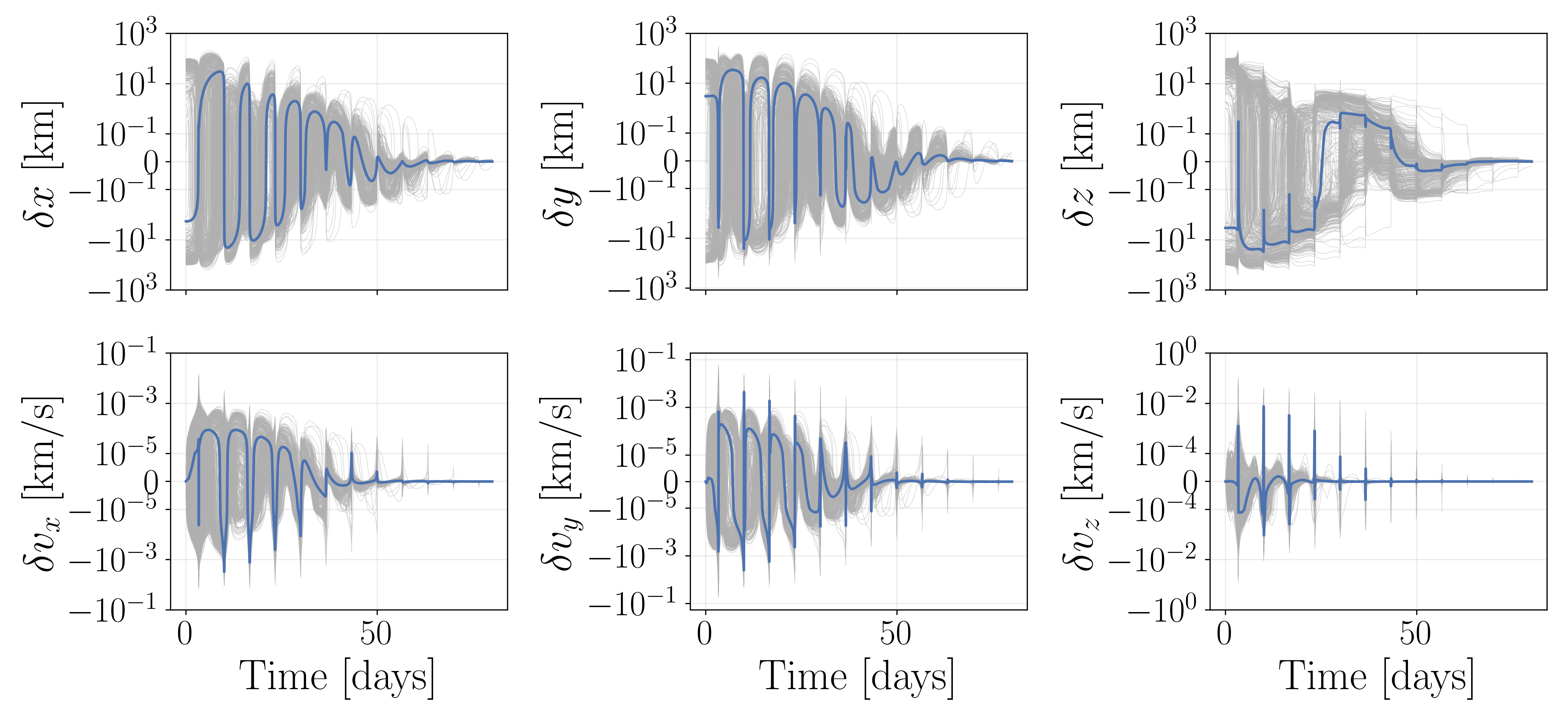}
\caption{State history.}
\label{fig:em_lt_state}
\end{subfigure}
\caption{Earth-Moon low-thrust: (a) position deviation of all samples (gray) with the mean in blue, a shaded $\pm 3\sigma$ band, the state bound $x_{b,k}$ dashed, and perilune passages; (b) per-component deviations, mean in blue.}
\label{fig:em_lt_deviation}
\end{figure}

In \Cref{fig:em_lt_worst,fig:em_lt_usample}, we isolate a single sample to study how the MTEs affect the nominal stationkeeping operations in this case. We focus on the first 30 days of the stationkeeping campaign, and it shows that despite five MTEs, the spacecraft is able to converge back to nominal operations. The two consecutive MTEs at the beginning cause a spike in the control effort, which then decreases gradually as the spacecraft converges to the baseline orbit. 

\begin{figure}[htb!]
\centering
\begin{subfigure}[b]{0.42\textwidth}
\centering
\includegraphics[width=\textwidth]{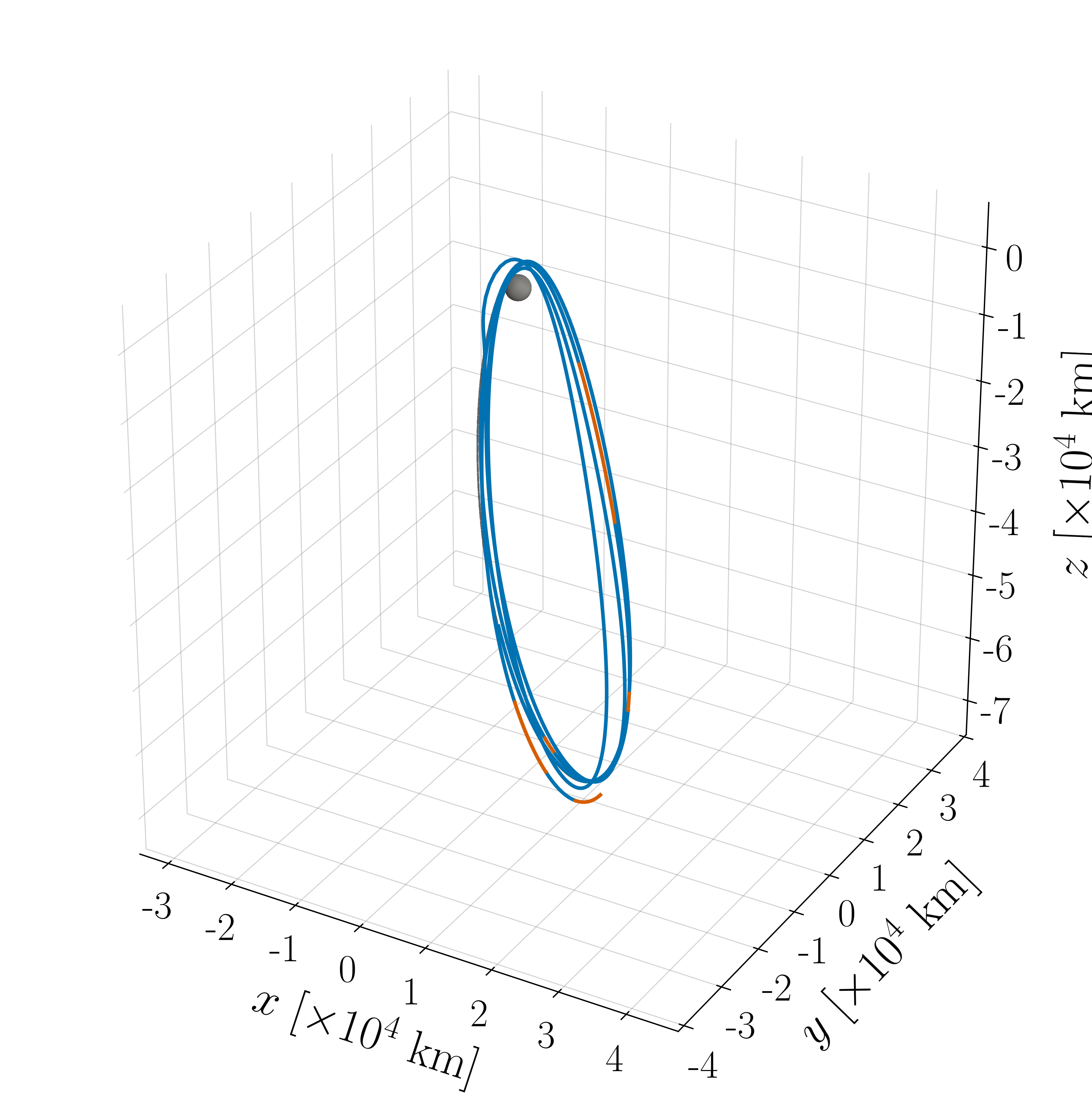}
\caption{Sample trajectory.}
\label{fig:em_lt_worst}
\end{subfigure}
\hfill
\begin{subfigure}[b]{0.56\textwidth}
\centering
\includegraphics[width=\textwidth]{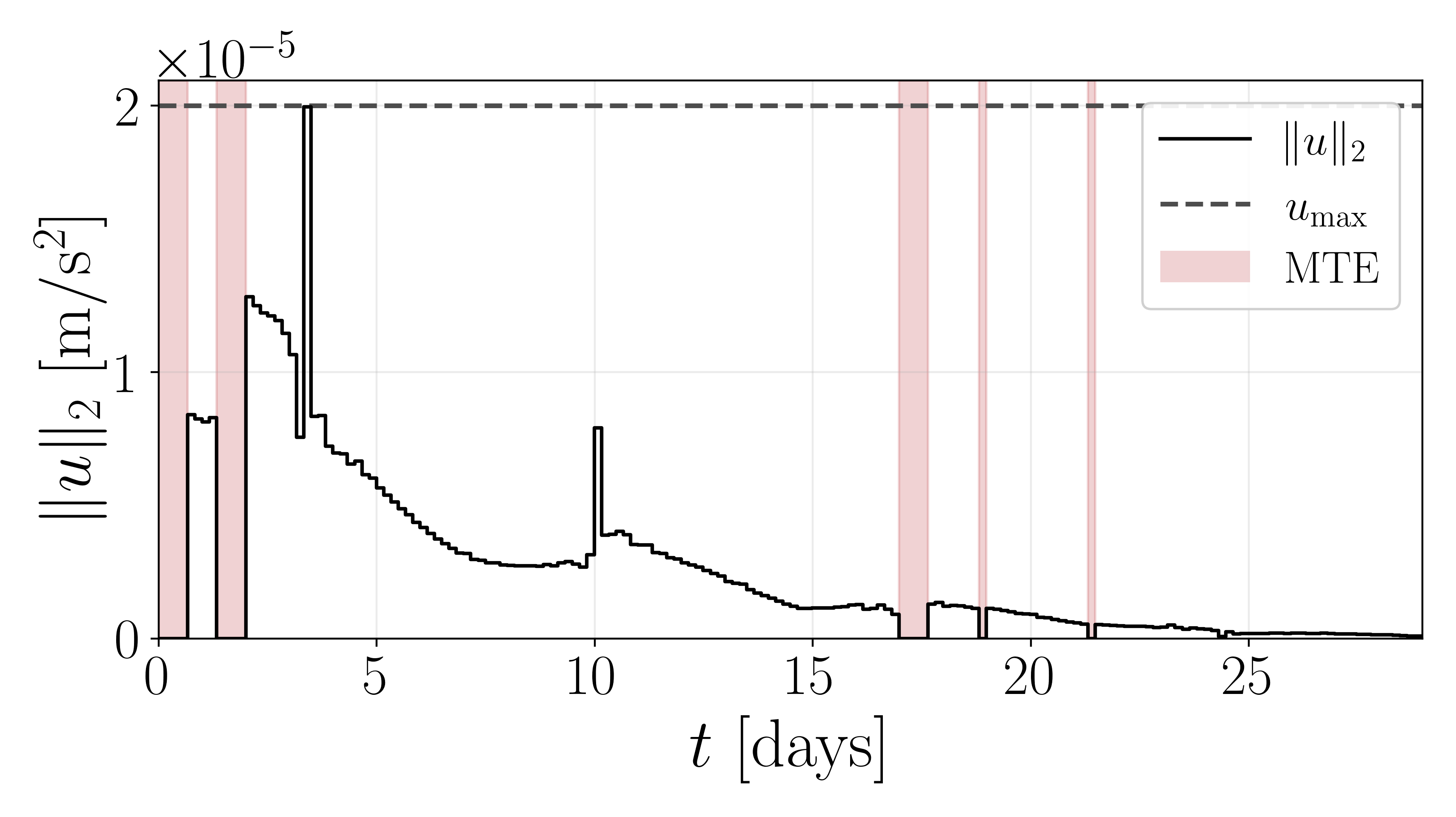}
\caption{Sample thrust profile.}
\label{fig:em_lt_usample}
\end{subfigure}
\caption{Earth-Moon low-thrust: (a) a single sample colored by mode, nominal (blue) and missed thrust (orange), with deviations magnified $20\times$; (b) its control history, with missed-thrust intervals shaded.}
\label{fig:em_lt_sample}
\end{figure}

\subsubsection{Saturn-Enceladus $L_1$ NRHO: Impulsive Stationkeeping}

The Saturn-Enceladus case is the most constrained: the orbit period is under 12 hours, each periapsis passage brings the spacecraft within a few hundred kilometers of the Enceladus surface. We taper the state bound accordingly to protect against a collision with the Enceladus surface under uncertainty. We perform the stationkeeping Monte Carlo simulations for 10 revolutions ($\approx 5$ days) with seven nominal TCM opportunities per revolution and two backup nodes per arc. \Cref{fig:se_imp_3d} shows the corresponding Monte Carlo sample trajectories, with deviations amplified $2$ times for visual clarity, and \Cref{fig:se_imp_u} gives the corresponding burn magnitudes. The simulations show good convergence to the baseline orbit despite multiple MTEs and other uncertainties. 

\begin{figure}[htb!]
\centering
\begin{subfigure}[b]{0.42\textwidth}
\centering
\includegraphics[width=\textwidth]{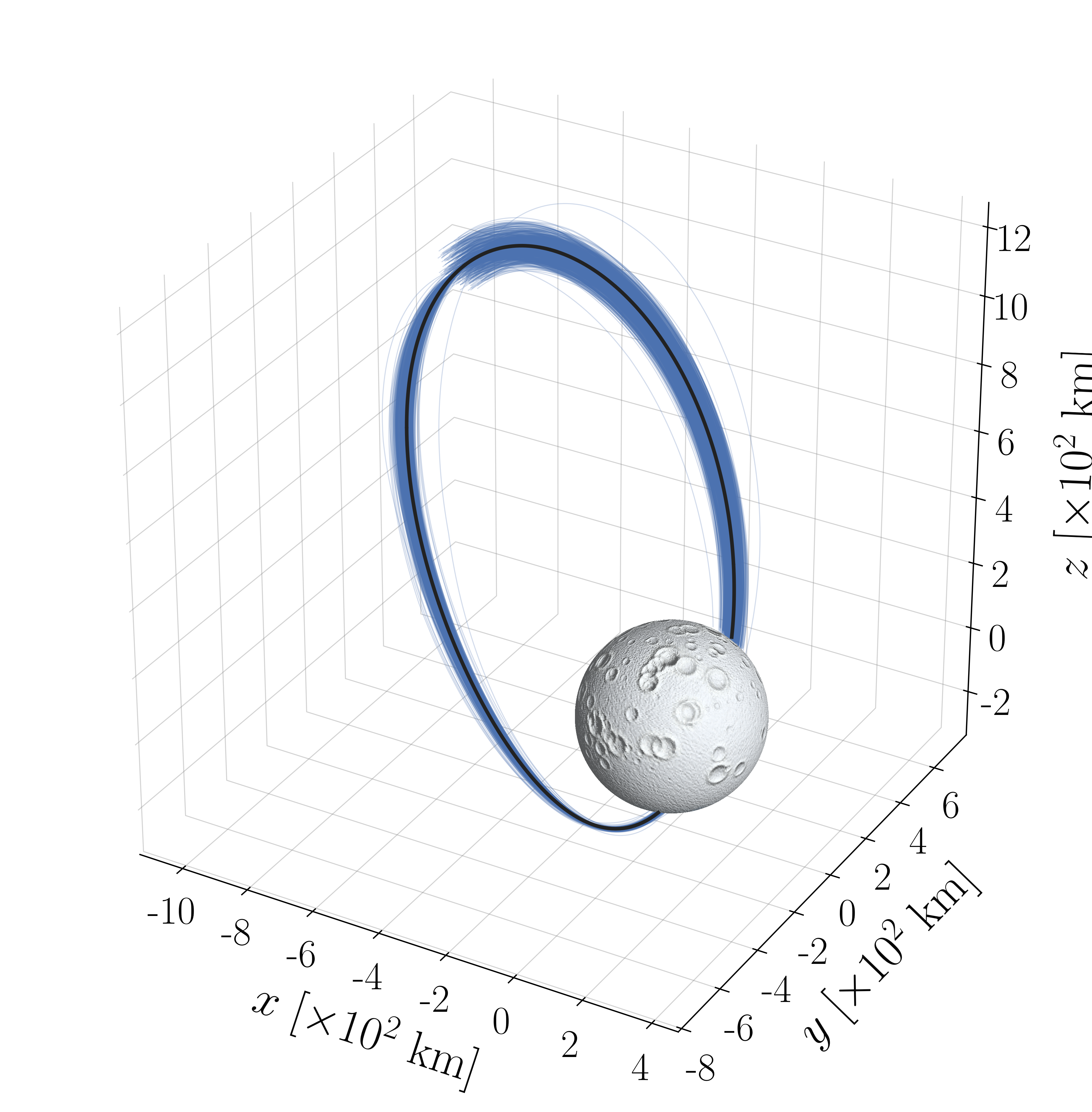}
\caption{Monte Carlo ensemble.}
\label{fig:se_imp_3d}
\end{subfigure}
\hfill
\begin{subfigure}[b]{0.56\textwidth}
\centering
\includegraphics[width=\textwidth]{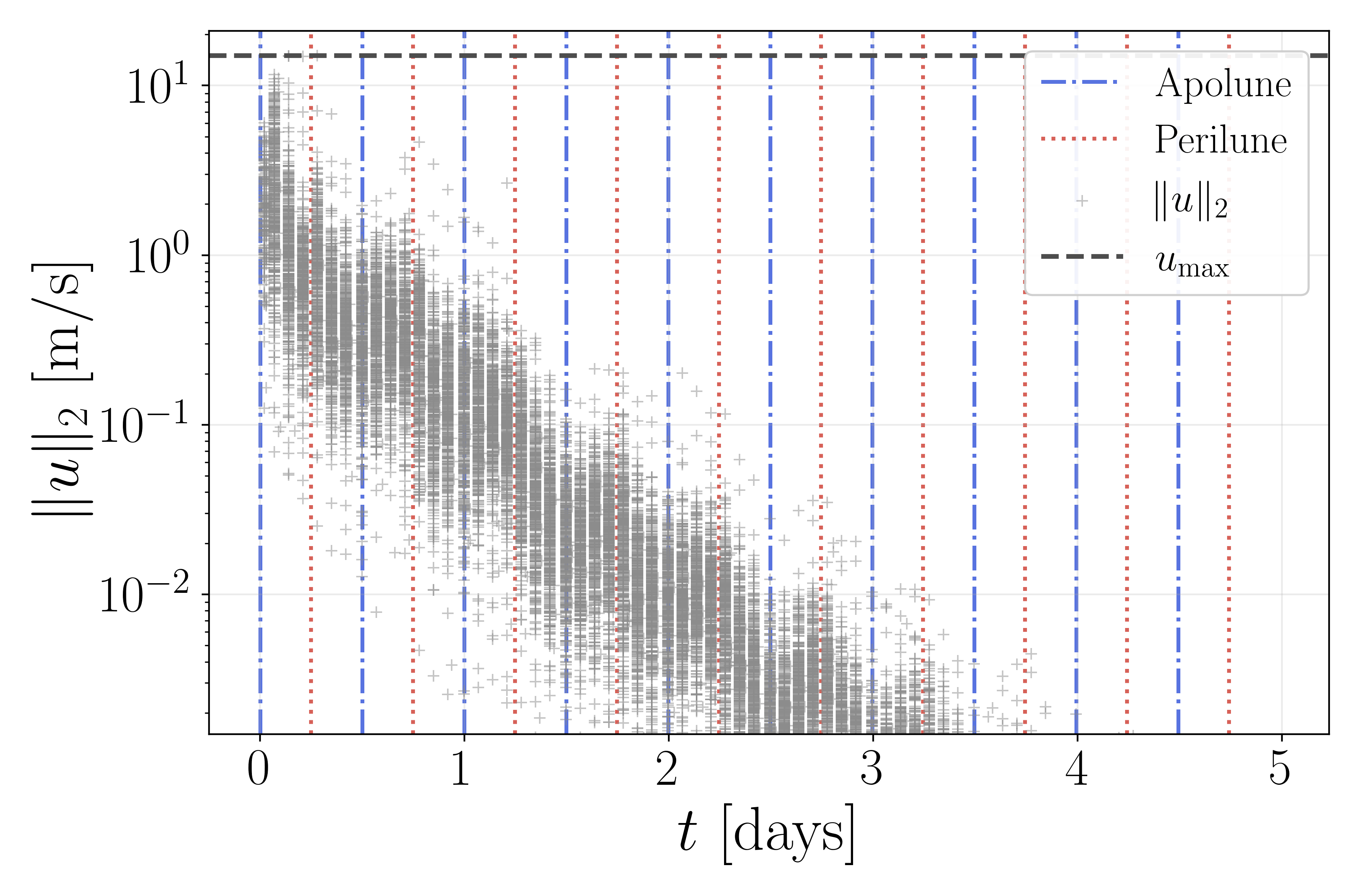}
\caption{Control history.}
\label{fig:se_imp_u}
\end{subfigure}
\caption{Saturn-Enceladus impulsive: (a) 500 Monte Carlo trajectories about the $L_1$ NRHO with deviations magnified $2\times$; (b) burn magnitudes, with apoapsis and periapsis passages marked and $u_{max}$ dashed.}
\label{fig:se_imp_ensemble}
\end{figure}

\Cref{fig:se_imp_dist} reports the absolute distance to the Enceladus center $d_E$, with the periapsis inset confirming that every sample retains a positive surface clearance; the closest approach over all 500 samples is 262.9 km from the Enceladus center, i.e. 10.8 km above the surface. \Cref{fig:se_imp_state} gives the per-component deviations (on a symmetric-log axis). We also isolate a single sample to study how the robust policy responds to multiple MTEs; the resulting trajectory and burn history are shown in \Cref{fig:se_imp_worst,fig:se_imp_usample}. In this sample, the spacecraft faces a long MTE at the beginning of the simulation, and executes a burn soon after recovery through the backup node framework. Successive TCM burns further stabilize the orbit, until a second MTE occurs, followed by a quick recovery. This causes an increase in TCM burn magnitude right after, but decreases as the spacecraft stabilizes to the baseline orbit.

\begin{figure}[htb!]
\centering
\begin{subfigure}[b]{0.46\textwidth}
\centering
\includegraphics[width=\textwidth]{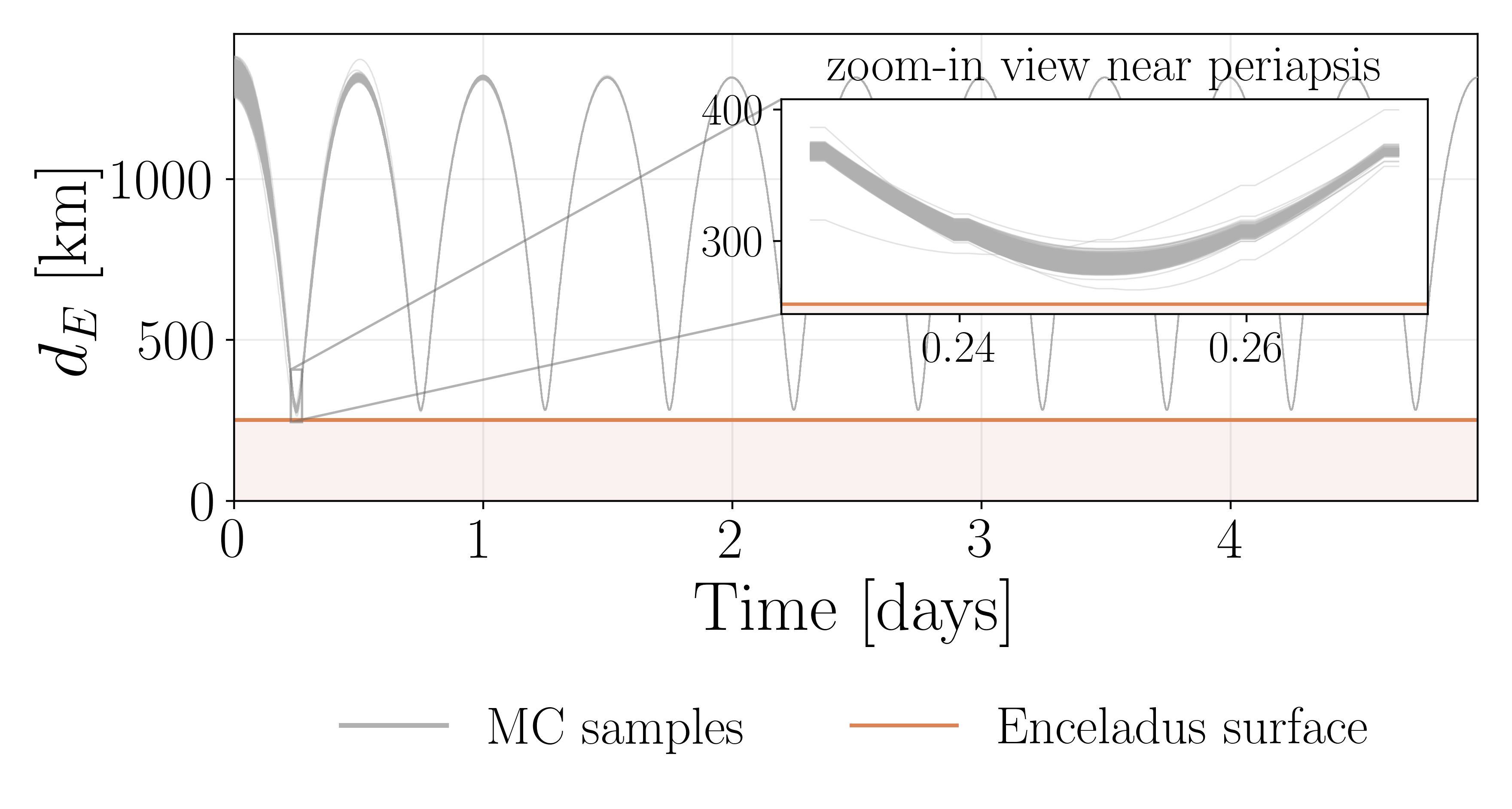}
\caption{Distance to Enceladus.}
\label{fig:se_imp_dist}
\end{subfigure}
\hfill
\begin{subfigure}[b]{0.52\textwidth}
\centering
\includegraphics[width=\textwidth]{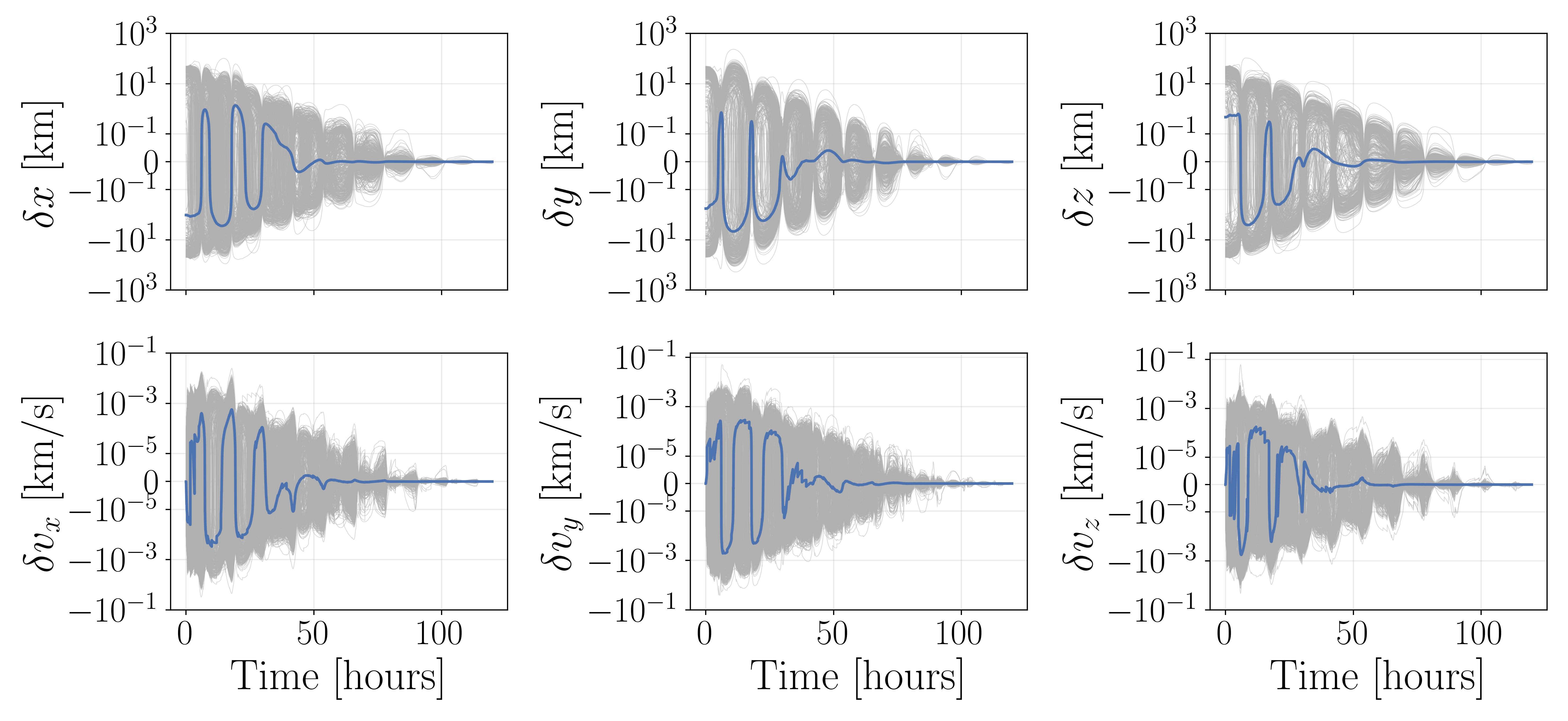}
\caption{State history.}
\label{fig:se_imp_state}
\end{subfigure}
\caption{Saturn-Enceladus impulsive: (a) distance from the Enceladus center for every sample, with the surface drawn and a zoom on the tightest periapsis passage; (b) per-component deviations, mean in blue.}
\label{fig:se_imp_deviation}
\end{figure}

\begin{figure}[htb!]
\centering
\begin{subfigure}[b]{0.42\textwidth}
\centering
\includegraphics[width=\textwidth]{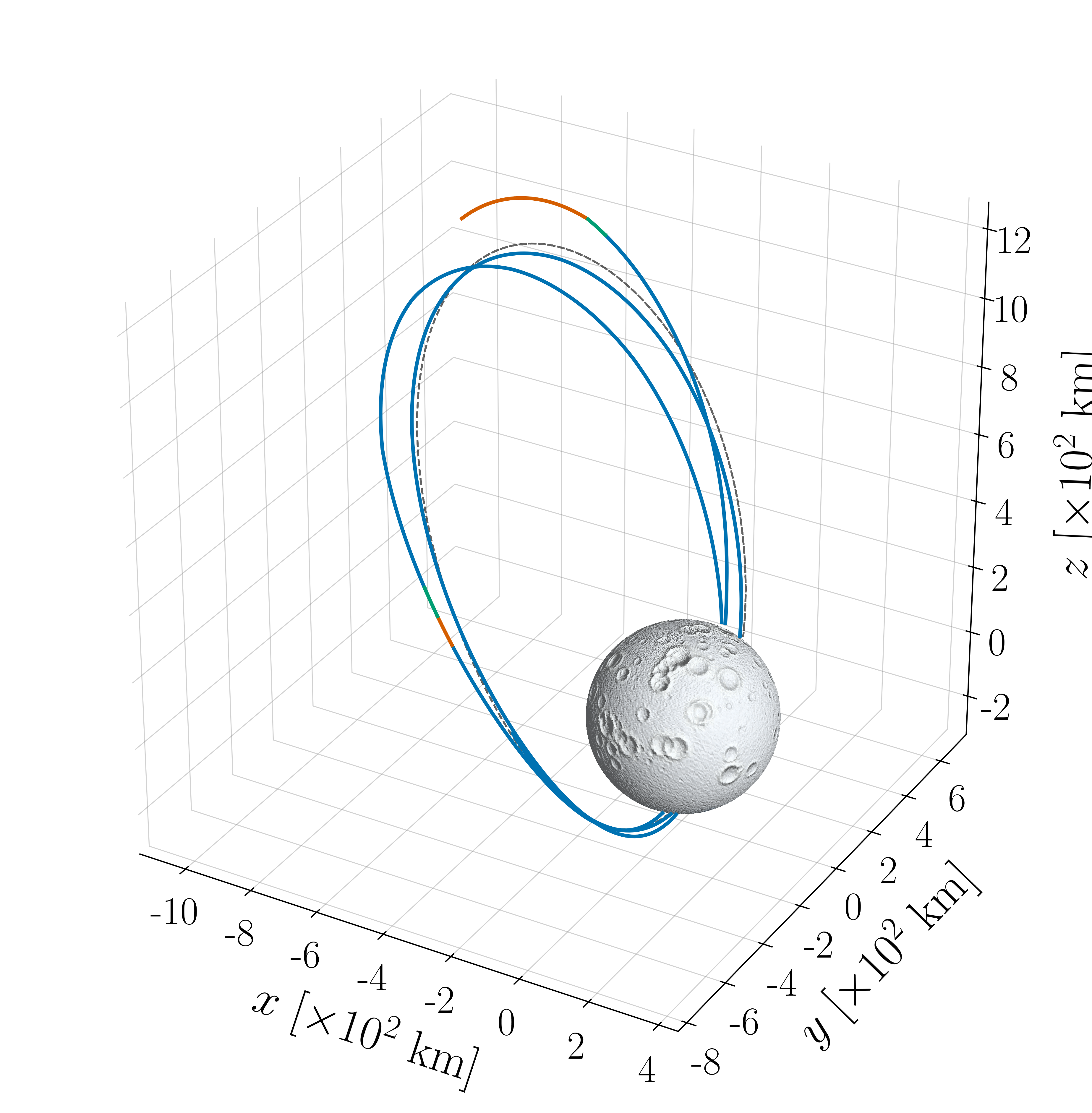}
\caption{Sample trajectory.}
\label{fig:se_imp_worst}
\end{subfigure}
\hfill
\begin{subfigure}[b]{0.56\textwidth}
\centering
\includegraphics[width=\textwidth]{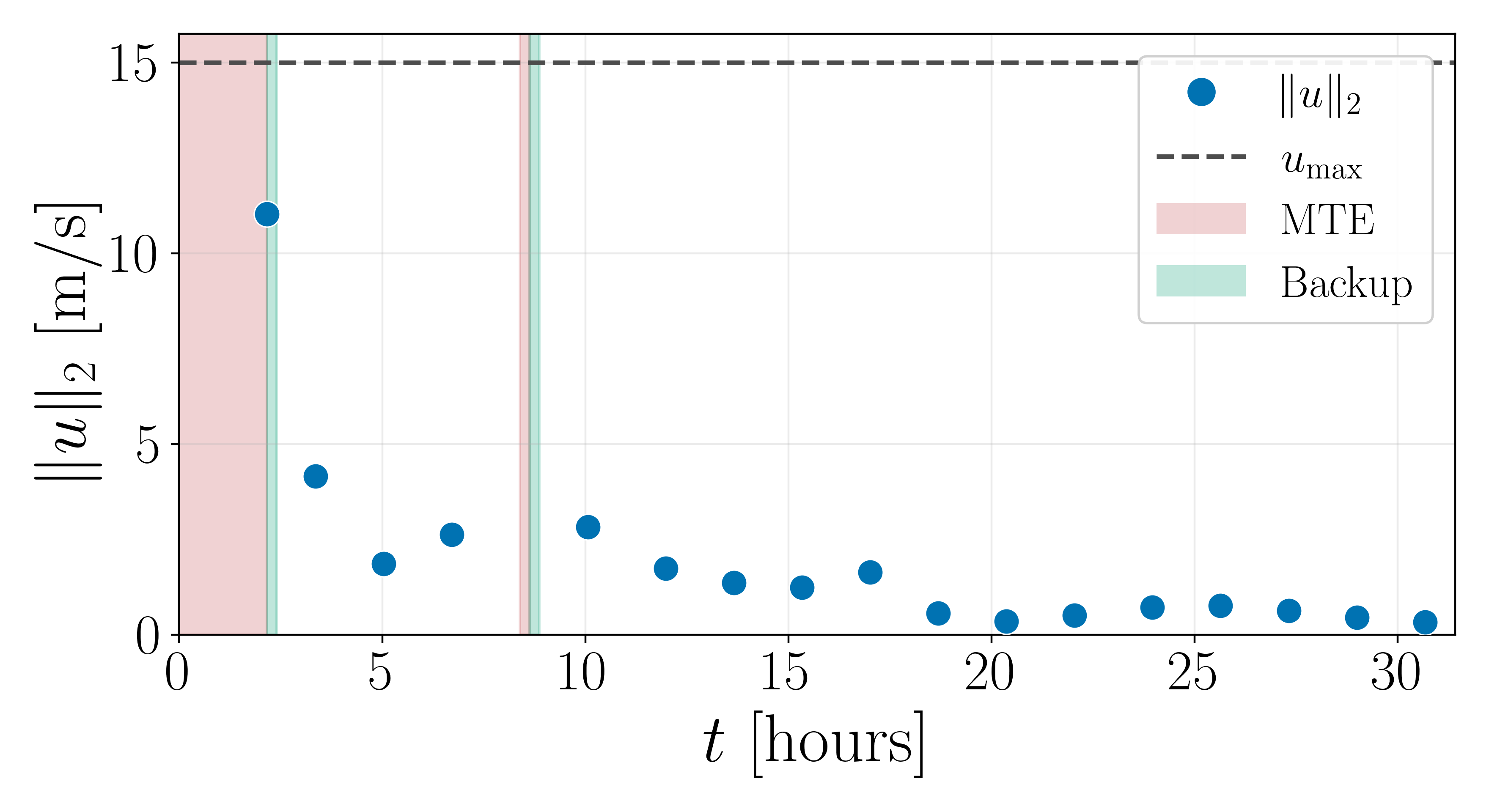}
\caption{Sample burn magnitudes.}
\label{fig:se_imp_usample}
\end{subfigure}
\caption{Saturn-Enceladus impulsive: (a) a single sample colored by mode: nominal (blue), missed thrust (orange), and backup (green), with deviations magnified $2\times$; (b) its burn magnitudes, intervals shaded.}
\label{fig:se_imp_sample}
\end{figure}

\subsubsection{Sensitivity Analysis to MTE Probability}

The results above use a single pair of failure statistics per case. To assess how strongly the closed-loop behavior depends on those parameters, each controller is re-synthesized and re-evaluated over a grid of $TTF_{90} \in \{2,\,4,\,8\}$ orbits and $TTR_{95} \in \{0.1,\,0.2,\,0.5\}$ orbits, with every other design parameter held fixed. \Cref{tab:ttf_ttr_sweep} reports, for each pair, the fraction of the 500 samples that both converge and satisfy the state bound $x_{b,k}$ at every node, together with the $99$th-percentile stationkeeping cost $\Delta v_{99}$ of those samples.

\begin{table}[hbt!]
\caption{\label{tab:ttf_ttr_sweep} Sensitivity to the missed-thrust statistics over 500 Monte Carlo samples: $\Delta v_{99}$ is the 99th-percentile cost for the successful cases, success rate the fraction that converge and satisfy $x_{b,k}$; a dash marks no solution.}
\centering
\small
\begin{tabular}{llccc}
\hline
 & & \multicolumn{3}{c}{$\Delta v_{99}$ in m/s (success rate \%)} \\\cline{3-5}
Case & \diagbox[width=2.4cm,height=1.0cm]{$TTF_{90}$}{$TTR_{95}$} & 0.1 orbits & 0.2 orbits & 0.5 orbits \\\hline
Earth-Moon impulsive       & 2 orbits & 2.24 (100.0)  & 3.61 (100.0) & 5.53 (99.2) \\
                           & 4 orbits & 2.22 (100.0)  & 3.07 (100.0) & 5.96 (99.2) \\
                           & 8 orbits & 1.74 (100.0)  & 3.02 (99.8)  & ---         \\\hline
Earth-Moon low-thrust      & 2 orbits & 8.89 (100.0)  & ---          & ---         \\
                           & 4 orbits & 8.07 (100.0)  & ---          & ---         \\
                           & 8 orbits & 6.37 (100.0)  & ---          & ---         \\\hline
Saturn-Enceladus impulsive & 2 orbits & 30.50 (99.4)  & ---          & ---         \\
                           & 4 orbits & 29.43 (99.8)  & 54.55 (98.2) & ---         \\
                           & 8 orbits & 27.50 (99.0)  & 36.83 (98.4) & ---         \\\hline
\end{tabular}
\end{table}

The results show a clear monotonic trend across both $TTF_{90}$ and $TTR_{95}$ for $\Delta v_{99}$. For a given $TTR_{95}$ say $0.1$ orbits, the total stationkeeping costs decrease as we increase the $TTF_{90}$. This is expected, as increasing it implies that the missed thrust events are less likely, or more spread out in time, which results in closer to nominal operations. As we decrease the $TTF_{90}$ parameter, the $\Delta v_{99}$ rises, which is natural to ensure robust performance against more frequent MTEs. This trend is common to all three cases we consider in this study. For the second parameter $TTR_{95}$, we note a clear trend of increase in $\Delta v_{99}$ costs as we increase the recovery time. Since increasing $TTR_{95}$ leads to longer MTEs, it is natural to see increased fuel costs which would be needed post recovery to counter the uncontrolled deviations caused by longer MTEs. We note that while this monotonic relationship between $TTR_{95}$ and $\Delta v_{99}$ is seen in both the impulsive cases, the problem does not have a feasible solution for the low-thrust case due to the smaller available thrust magnitude. This indicates that for the given parameters, the Earth-Moon low-thrust stationkeeping scenario sits on the edge of feasibility with respect to the MTE parameters, and greater control authority or relaxed robustness parameters for other uncertainties might be desired to guarantee MTE resilience with longer recovery times. The Saturn-Enceladus case yields an infeasible SDP in some cases, especially for higher $TTR_
{95}$ values. Similarly, this indicates the need for greater control authority or relaxed robustness needs to guarantee MTE resilience for higher $TTR_{95}$ cases. Overall, the results also show that MTE recovery times are more critical than MTE occurrence probability, and guaranteeing robustness against longer MTEs is more difficult than guaranteeing robustness against more frequent (but shorter) MTEs. 

On the other hand, the success rate for Monte Carlo simulations are largely unaffected by the change in these parameters. In the Earth-Moon impulsive case, increasing the recovery times does decrease the success rate slightly, but it remains close to $100\%$ for against different MTE occurrence probabilities. A similar trends is observed against $TTF_{90}$ for the low-thrust case too, where the success rate is still $100\%$, but the controller synthesis fails to find a solution for higher $TTR_{95}$ values in this case. For Saturn-Enceladus, the success rate does not show a clear trend against $TTF_{90}$, but it is still $\geq99\%$, and there are not enough converged results to make a conclusive claim on success rate dependence against $TTR_{95}$. Overall, success rate in all cases seems to be close to $100\%$, which implies that if the framework finds a solution, it also performs well in the Monte Carlo simulations.
 
\begin{table}[htbp!]
\caption{\label{tab:backup_nodes} Sensitivity of the Earth-Moon impulsive case to the backup TCM node layout, over 500 Monte Carlo samples. Placement is the offset $+j$ from the preceding nominal TCM node within a seven-node coasting arc.}
\centering
\small
\begin{tabular}{cccc}
\hline
Backup nodes & Placement & Success rate & $\Delta v_{99}$ \\
per coasting arc & within arc & (\%) & (m/s) \\\hline
0 & --- & 100.0 & 2.79 \\
2 & $+1,+2$ & 100.0 & 2.69 \\
5 & $+1,\ldots,+5$ & 100.0 & 2.22 \\
5 & $+3,\ldots,+7$ & 100.0 & 3.07 \\
7 & $+1,\ldots,+7$ (full arc) & 100.0 & 1.87 \\\hline
\end{tabular}
\end{table}

We also note that the placement of backup nodes also affect the controller performance. To analyze this, we study the Earth-Moon impulsive case against different backup TCM node placement strategies, and \Cref{tab:backup_nodes} shows the success rate and $\Delta v_{99}$ for different placement strategies. We consider 5 different placement strategies: (i) no backup nodes (a two mode architecture,) (ii) Two backup TCM nodes immediately after each nominal TCM node, (iii) Five backup TCM nodes immediately after each nominal TCM node (baseline case,) (iv) Five backup TCM nodes, but placed immediately before the next nominal TCM node, and (v) Seven backup TCM nodes filling up the whole coasting arc. While the controller has $100\%$ success rate across all backup node architectures, the $\Delta v_{99}$ trends show a strong dependence on the choice of backup node architecture. We note that adding the number of backup nodes decreases fuel costs monotonically, with the lowest $\Delta v_{99}$ case being the 7 backup node case, which fills up the entire coasting arc. However, this last case is not practical, as it can lead to two consecutive TCM burns under some MTE sequences, leading to almost no time for the spacecraft operators to perform orbit determination. Therefore, the baseline case, with 5 backup TCM nodes offers the best performance trading off practical requirements with reduced fuel costs. The placement of the backup nodes also has a very strong effect on the fuel costs, as shown by case (iv) where the five backup TCM nodes are placed at the end of the coasting arc instead of the beginning. While this is also not practical as explained before because it can lead to consecutive TCM burns under some MTE sequences, it emphasizes that the ideal backup node framework is similar to our baseline case: provide a sufficient number of backup TCM opportunities immediately after nominal TCM node while leaving enough time before next nominal TCM burn. Our framework allows mission designers to conduct such trade studies to optimize stationkeeping strategies without the need for introducing conservative heuristics or relying on extensive Monte Carlo simulations.

\section{Conclusion}\label{sec:conclusion}

This work presents a Markov Jump system framework for robust NRHO stationkeeping under statistical missed thrust events, maneuver execution errors, and linearization errors. We model missed thrust events as a Markov chain fitted to exponential distribution-based failure and recovery statistics, and utilize bounded multiplicative noise to model maneuver execution error and linearization errors. We then study the corresponding 
linear time-varying periodic Markov jump system with multiplicative noise and propose an LMI-based controller synthesis framework that respects path constraints with probability one. To enforce path constraints, we introduce the notion of a $T$-periodic robust positively invariant sequence of sets, and combine it with Lyapunov-like stability constraints in the synthesis framework. We demonstrate the robustness of our framework through numerical results on the Earth-Moon 9:2 NRHO and a Saturn-Enceladus $L_1$ NRHO stationkeeping example. For the impulsive propulsion case, we propose a three-mode MJLS structure, which provides contingency maneuver locations that are activated only when an MTE causes a planned burn to be missed. Further parametric study also offers insight into the characteristics of the MTE-resilient solution and its dependence on various problem parameters. Our framework enables robust maneuver planning for stationkeeping that is resilient against an arbitrary number of missed thrust events, maneuver execution errors, and linearization errors through the multiplicative noise-based Markov Jump System framework.

\section*{Acknowledgments}
AI models were used in the paper to polish grammar, and for making schematics and in coding assistance for plot generation. The authors reviewed and edited the content carefully and take full responsibility for the content in this article. 

\bibliography{ref}

@article{shrivastavaRobustControllerSynthesis2024a,
	title = {Robust {Controller} {Synthesis} {Under} {Markovian} {Mode} {Switching} {With} {Periodic} {LTV} {Dynamics}},
	issn = {2475-1456},
	doi = {10.1109/LCSYS.2024.3522212},
	urldate = {2025-01-03},
	journal = {IEEE Control Syst. Lett.},
	author = {Shrivastava, Shaurya and Oguri, Kenshiro},
	year = {2024},
	pages = {1--1},
}

@article{houSpectralTestsObservability2016,
	title = {Spectral tests for observability and detectability of periodic {Markov} jump systems with nonhomogeneous {Markov} chain},
	volume = {63},
	issn = {0005-1098},
	doi = {10.1016/j.automatica.2015.10.004},
	urldate = {2024-09-11},
	journal = {Automatica},
	author = {Hou, Ting and Ma, Hongji and Zhang, Weihai},
	month = jan,
	year = {2016},
	pages = {175--181},
}

@article{bohmStabilityPeriodicallyTimevarying2012,
	title = {Stability of periodically time-varying systems: {Periodic} {Lyapunov} functions},
	volume = {48},
	issn = {0005-1098},
	shorttitle = {Stability of periodically time-varying systems},
	doi = {10.1016/j.automatica.2012.06.070},
	number = {10},
	urldate = {2024-06-04},
	journal = {Automatica},
	author = {Böhm, Christoph and Lazar, Mircea and Allgöwer, Frank},
	month = oct,
	year = {2012},
	pages = {2663--2669},
}

@inproceedings{sekineRobustTrajectoryDesign2026,
	title = {Robust {Trajectory} {Design} for {Earth}-to-lunar {Orbit} {Transfers} via {Sequential} {Convex} {Programming}},
	booktitle = {{AAS} {Guidance}, {Navigation}, and {Control} {Conference}},
	author = {Sekine, Hirotaka and Oguri, Kenshiro and Kawabata, Yosuke and Funase, Ryu},
	month = feb,
	year = {2026},
}

@article{rubinsztejnNeuralNetworkOptimal2020,
	title = {Neural network optimal control in astrodynamics: {Application} to the missed thrust problem},
	volume = {176},
	issn = {0094-5765},
	shorttitle = {Neural network optimal control in astrodynamics},
	doi = {10.1016/j.actaastro.2020.05.027},
	language = {en},
	urldate = {2023-04-07},
	journal = {Acta Astronautica},
	author = {Rubinsztejn, Ari and Sood, Rohan and Laipert, Frank E.},
	month = nov,
	year = {2020},
	pages = {192--203},
}

@incollection{grebowDawnSafingApproach2016,
	series = {{AIAA} {SPACE} {Forum}},
	title = {Dawn {Safing} {Approach} to {Ceres} {Re}-{Design}},
	doi = {10.2514/6.2016-5426},
	urldate = {2025-01-08},
	booktitle = {{AIAA}/{AAS} {Astrodynamics} {Specialist} {Conference}},
	author = {Grebow, Daniel and Whiffen, Gregory J. and Han, Dongsuk and Kennedy, Brian},
	month = sep,
	year = {2016},
}

@incollection{muralidharanControlStrategyLongTerm,
	title = {Control {Strategy} for {Long}-{Term} {Station}-{Keeping} on {Near}-{Rectilinear} {Halo} {Orbits}},
	doi = {10.2514/6.2020-1459},
	urldate = {2026-08-12},
	booktitle = {{AIAA} {Scitech} 2020 {Forum}},
	author = {Muralidharan, Vivek and Weiss, Avishai and Kalabic, Uros V.},
	year = {2020},
	doi = {10.2514/6.2020-1459},
}

@inproceedings{aletiChanceConstrainedOutputFeedbackControl2023,
	title = {Chance-{Constrained} {Output}-{Feedback} {Control} without {History} {Feedback}: {Application} to {NRHO} {Stationkeeping}},
	shorttitle = {Chance-{Constrained} {Output}-{Feedback} {Control} without {History} {Feedback}},
	booktitle = {{AAS}/{AIAA} {Astrodynamics} {Specialist} {Conference}},
	author = {Aleti, Divija and Oguri, Kenshiro and Kumagai, Naoya},
	month = aug,
	year = {2023},
}

@inproceedings{kumagaiRobustNRHOStationkeeping2023,
	title = {Robust {NRHO} {Station}-keeping {Planning} with {Maneuver} {Location} {Optimization} under {Operational} {Uncertainties}},
	booktitle = {{AAS}/{AIAA} {Astrodynamics} {Specialist} {Conference}},
	author = {Kumagai, Naoya and Oguri, Kenshiro},
	month = aug,
	year = {2023},
}

@inproceedings{shrivastavaMASTOMissedthrustAware2026,
	title = {{MASTO}: {Missed}-thrust {Aware} {Stochastic} {Trajectory} {Optimization}},
	shorttitle = {{MASTO}},
	booktitle = {2026 {AAS}/{AIAA} {Astrodynamics} {Specialist} {Conference}},
	author = {Shrivastava, Shaurya and Oguri, Kenshiro},
	month = jul,
	year = {2026},
}

@article{maH2ControlDiscretetime2013,
	title = {H2 control of discrete-time periodic systems with {Markovian} jumps and multiplicative noise},
	volume = {86},
	issn = {0020-7179},
	doi = {10.1080/00207179.2013.797108},
	number = {10},
	urldate = {2024-09-09},
	journal = {Int. J. of Control},
	author = {Ma, Hongji and Jia, Yingmin},
	month = oct,
	year = {2013},
	pages = {1837--1849},
}

@article{sinhaInitialGuessGeneration2025a,
	title = {Initial {Guess} {Generation} for {Low}-{Thrust} {Trajectory} {Design} with {Robustness} to {Missed}-{Thrust} {Events}},
	volume = {48},
	issn = {0731-5090},
	doi = {10.2514/1.G009050},
	number = {12},
	urldate = {2026-03-28},
	journal = {Journal of Guidance, Control, and Dynamics},
	publisher = {American Institute of Aeronautics and Astronautics},
	author = {Sinha, Amlan and Beeson, Ryne},
	year = {2025},
	pages = {2831--2848},
}

@article{rubinsztejnDesigningTrajectoriesResilient2021,
	title = {Designing trajectories resilient to missed thrust events using expected thrust fraction},
	volume = {115},
	issn = {1270-9638},
	doi = {10.1016/j.ast.2021.106780},
	language = {en},
	urldate = {2023-04-07},
	journal = {Aerospace Science and Technology},
	author = {Rubinsztejn, Ari and Sandel, Carrie Grace and Sood, Rohan and Laipert, Frank E.},
	month = aug,
	year = {2021},
	pages = {106780},
}

@article{lafargeAdaptiveClosedloopManeuver2023,
	title = {Adaptive closed-loop maneuver planning for low-thrust spacecraft using reinforcement learning},
	volume = {211},
	issn = {0094-5765},
	doi = {10.1016/j.actaastro.2023.06.004},
	urldate = {2024-09-18},
	journal = {Acta Astronautica},
	author = {LaFarge, Nicholas B. and Howell, Kathleen C. and Folta, David C.},
	month = oct,
	year = {2023},
	pages = {142--154},
}

@article{mackenzieEnceladusOrbilanderMission2021,
	title = {The {Enceladus} {Orbilander} {Mission} {Concept}: {Balancing} {Return} and {Resources} in the {Search} for {Life}},
	volume = {2},
	issn = {2632-3338},
	shorttitle = {The {Enceladus} {Orbilander} {Mission} {Concept}},
	doi = {10.3847/PSJ/abe4da},
	language = {en},
	number = {2},
	urldate = {2024-12-19},
	journal = {The Planetary Science Journal},
	publisher = {IOP Publishing},
	author = {MacKenzie, Shannon M. and Neveu, Marc and Davila, Alfonso F. and Lunine, Jonathan I. and Craft, Kathleen L. and Cable, Morgan L. and Phillips-Lander, Charity M. and Hofgartner, Jason D. and Eigenbrode, Jennifer L. and Waite, J. Hunter and Glein, Christopher R. and Gold, Robert and Greenauer, Peter J. and Kirby, Karen and Bradburne, Christopher and Kounaves, Samuel P. and Malaska, Michael J. and Postberg, Frank and Patterson, G. Wesley and Porco, Carolyn and Núñez, Jorge I. and German, Chris and Huber, Julie A. and McKay, Christopher P. and Vera, Jean-Pierre de and Brucato, John Robert and Spilker, Linda J.},
	month = apr,
	year = {2021},
	pages = {77},
}

@article{sidhoumLowthrustTrajectoryDesign2024a,
	title = {Low-thrust trajectory design for icy moons orbiters using multi-body techniques},
	volume = {136},
	issn = {1572-9478},
	doi = {10.1007/s10569-024-10228-w},
	language = {en},
	number = {6},
	urldate = {2026-08-13},
	journal = {Celestial Mechanics and Dynamical Astronomy},
	author = {Sidhoum, Yanis and Oguri, Kenshiro},
	month = nov,
	year = {2024},
	pages = {55},
}

@inproceedings{oguriChanceConstrainedControlSafe2024,
	title = {Chance-{Constrained} {Control} for {Safe} {Spacecraft} {Autonomy}: {Convex} {Programming} {Approach}},
	issn = {2378-5861},
	shorttitle = {Chance-{Constrained} {Control} for {Safe} {Spacecraft} {Autonomy}},
	doi = {10.23919/ACC60939.2024.10645008},
	urldate = {2025-01-08},
	booktitle = {2024 {Amer}. {Control} {Conf}. ({ACC})},
	author = {Oguri, Kenshiro},
	month = jul,
	year = {2024},
	pages = {2318--2324},
}

@article{laipertAutomatedMissedThrustPropellant2015,
	title = {Automated {Missed}-{Thrust} {Propellant} {Margin} {Analysis} for {Low}-{Thrust} {Trajectories}},
	volume = {52},
	issn = {0022-4650},
	doi = {10.2514/1.A33264},
	number = {4},
	urldate = {2024-09-23},
	journal = {Journal of Spacecraft and Rockets},
	publisher = {American Institute of Aeronautics and Astronautics},
	author = {Laipert, Frank E. and Longuski, James M.},
	year = {2015},
	pages = {1135--1143},
}

@article{nemirovskyTrajectoryDesignControl2026,
	title = {Trajectory {Design} and {Control} for {Missed}-{Thrust} {Rendezvous}},
	volume = {49},
	issn = {0731-5090},
	doi = {10.2514/1.G008831},
	number = {4},
	urldate = {2026-08-31},
	journal = {Journal of Guidance, Control, and Dynamics},
	publisher = {American Institute of Aeronautics and Astronautics},
	author = {Nemirovsky, Meir and Gurfil, Pini},
	month = apr,
	year = {2026},
	pages = {990--1004},
}

@article{gatesSimplifiedModelMidcourse,
	address = {Jet Propulsion Lab, Pasadena, CA},
	title = {A {Simplified} {Model} of {Midcourse} {Maneuver} {Execution} {Errors}},
	url = {https://ntrs.nasa.gov/citations/19640003365},
	language = {en},
	journal = {Contractor Report JPL-TR-32-504},
	author = {Gates, C R},
	year = {1963},
}

@article{kumagaiRobustCislunarLowThrust,
	title = {Robust {Cislunar} {Low}-{Thrust} {Trajectory} {Optimization} {Under} {Uncertainties} via {Sequential} {Covariance} {Steering}},
	volume = {48},
	issn = {0731-5090},
	doi = {10.2514/1.G009092},
	number = {12},
	urldate = {2025-09-24},
	journal = {Journal of Guidance, Control, and Dynamics},
	author = {Kumagai, Naoya and Oguri, Kenshiro},
	month = dec,
	year = {2025},
	pages = {1--19},
}

@article{qiConvexBoundNonlinear,
	title = {Convex {Bound} of {Nonlinear} {Dynamical} {Errors} for {Covariance} {Steering}},
	volume = {49},
	issn = {0731-5090},
	doi = {10.2514/1.G009692},
	number = {10},
	urldate = {2026-10-07},
	journal = {Journal of Guidance, Control, and Dynamics},
	publisher = {American Institute of Aeronautics and Astronautics},
	author = {Qi, Daniel C. and Oguri, Kenshiro},
	month = oct,
	year = {2026},
	pages = {3014--3024},
}

@inproceedings{imkenModelingSpacecraftSafe2018,
	title = {Modeling spacecraft safe mode events},
	doi = {10.1109/AERO.2018.8396383},
	urldate = {2024-04-26},
	booktitle = {2018 {IEEE} {Aerospace} {Conference}},
	author = {Imken, Travis and Randolph, Thomas and DiNicola, Michael and Nicholas, Austin},
	month = mar,
	year = {2018},
	pages = {1--13},
}

@article{desouzaRobustStabilityStabilization2006,
	title = {Robust stability and stabilization of uncertain discrete-time {Markovian} jump linear systems},
	volume = {51},
	issn = {1558-2523},
	doi = {10.1109/TAC.2006.875012},
	number = {5},
	urldate = {2024-06-05},
	journal = {IEEE Transactions on Automatic Control},
	author = {de Souza, C.E.},
	month = may,
	year = {2006},
	pages = {836--841},
}

@article{barbieriOptimalControlConstrained2018,
	title = {Optimal control with constrained total variance for {Markov} jump linear systems with multiplicative noises},
	volume = {49},
	issn = {0020-7721},
	doi = {10.1080/00207721.2018.1441469},
	number = {6},
	urldate = {2024-06-12},
	journal = {International Journal of Systems Science},
	publisher = {Taylor \& Francis},
	author = {Barbieri, Fabio and Costa, Oswaldo L. V.},
	month = apr,
	year = {2018},
	pages = {1178--1187},
}

@misc{wangCovarianceSteeringDiscreteTime2026,
	title = {Covariance {Steering} of {Discrete}-{Time} {Markov} {Jump} {Linear} {Systems} with {Multiplicative} {Noise}},
	doi = {10.48550/arXiv.2604.19994},
	urldate = {2026-05-11},
	publisher = {arXiv},
	author = {Wang, Fangji and Ganguly, Siddhartha and Tsiotras, Panagiotis},
	month = apr,
	year = {2026},
	note = {arXiv:2604.19994 [math.OC]},
}

@book{bittantiPeriodicSystemsFiltering2008,
	edition = {1st},
	title = {Periodic {Systems}: {Filtering} and {Control}},
	isbn = {978-1-84800-910-3},
	shorttitle = {Periodic {Systems}},
	publisher = {Springer Publishing Company Inc.},
	author = {Bittanti, Sergio and Colaneri, Patrizio},
	month = oct,
	year = {2008},
}

@inproceedings{maStochasticH2Optimal2012,
	title = {Stochastic {H2} optimal control of discrete-time {Markov} jump systems with periodic coefficients},
	issn = {2378-5861},
	doi = {10.1109/ACC.2012.6315462},
	urldate = {2024-06-06},
	booktitle = {2012 {American} {Control} {Conference} ({ACC})},
	author = {Ma, Hongji and Jia, Yingmin and Du, Junping and Yu, Fashan},
	month = jun,
	year = {2012},
	pages = {1640--1645},
}

@book{boydLinearMatrixInequalities1994,
	series = {Soc. for {Ind}. and {Appl}. {Math}.},
	title = {Linear {Matrix} {Inequalities} in {System} and {Control} {Theory}},
	isbn = {978-0-89871-485-2},
	doi = {10.1137/1.9781611970777},
	urldate = {2024-09-16},
	publisher = {Soc. for Ind. and Appl. Math.},
	author = {Boyd, Stephen and El Ghaoui, Laurent and Feron, Eric and Balakrishnan, Venkataramanan},
	month = jan,
	year = {1994},
}

@inproceedings{shrivastavaMarkovJumpApproachRobust2025,
	title = {Markov-{Jump} {Approach} to {Robust} {TCM} {Policy} {Design} under {Statistical} {Missed}-{Thrust} {Events}: {Application} to {NRHO} {Stationkeeping}},
	shorttitle = {{MARKOV}-{JUMP} {APPROACH} {TO} {ROBUST} {TCM} {POLICY} {DESIGN} {UNDER} {STATISTICAL} {MISSED}-{THRUST} {EVENTS}},
	urldate = {2025-03-13},
	booktitle = {{AAS} {Space} {Flight} {Mechanics} {Meeting}},
	author = {Shrivastava, Shaurya and Oguri, Kenshiro},
	year = {2025},
}

@article{liRobustSimultaneousFault2019,
	title = {Robust {Simultaneous} {Fault} {Estimation} and {Nonfragile} {Output} {Feedback} {Fault}-{Tolerant} {Control} for {Markovian} {Jump} {Systems}},
	volume = {49},
	issn = {2168-2232},
	doi = {10.1109/TSMC.2018.2828123},
	number = {9},
	urldate = {2024-08-24},
	journal = {IEEE Trans. on Syst., Man, and Cybern.: Syst.},
	author = {Li, Xiaohang and Ahn, Choon Ki and Lu, Dunke and Guo, Shenghui},
	month = sep,
	year = {2019},
	pages = {1769--1776},
}

@article{costaStabilityResultsDiscreteTime1993,
	title = {Stability {Results} for {Discrete}-{Time} {Linear} {Systems} with {Markovian} {Jumping} {Parameters}},
	volume = {179},
	issn = {0022-247X},
	language = {en},
	number = {1},
	urldate = {2023-04-22},
	journal = {J. of Math. Anal. and Appl.},
	author = {Costa, O. L. V. and Fragoso, M. D.},
	month = oct,
	year = {1993},
	pages = {154--178},
}

@article{dovalRecedingHorizonControl1999,
	title = {Receding horizon control of jump linear systems and a macroeconomic policy problem},
	volume = {23},
	issn = {0165-1889},
	number = {8},
	urldate = {2024-08-24},
	journal = {J. of Econ. Dyn. and Control},
	author = {do Val, João B. R. and Başar, Tamer},
	month = aug,
	year = {1999},
	pages = {1099--1131},
}

@article{jiControllabilityObservabilityDiscretetime1988,
	title = {Controllability, observability and discrete-time markovian jump linear quadratic control},
	volume = {48},
	issn = {0020-7179},
	number = {2},
	urldate = {2024-01-25},
	journal = {International Journal of Control},
	publisher = {Taylor \& Francis},
	author = {JI, YUANDONG and CHIZECK, HOWARD J.},
	month = aug,
	year = {1988},
	pages = {481--498},
}

@article{costaConstrainedQuadraticState1999,
	title = {Constrained quadratic state feedback control of discrete-time {Markovian} jump linear systems},
	volume = {35},
	issn = {0005-1098},
	url = {https://www.sciencedirect.com/science/article/pii/S0005109898002027},
	language = {en},
	number = {4},
	urldate = {2023-01-24},
	journal = {Automatica},
	author = {Costa, O. L. V. and Assumpção Filho, E. O. and Boukas, E. K. and Marques, R. P.},
	month = apr,
	year = {1999},
	pages = {617--626},
}

@article{elghaouiRobustStatefeedbackStabilization1996,
	title = {Robust state-feedback stabilization of jump linear systems via {LMIs}},
	volume = {6},
	copyright = {Copyright © 1996 John Wiley \& Sons, Ltd.},
	issn = {1099-1239},
	doi = {10.1002/(SICI)1099-1239(199611)6:9/10<1015::AID-RNC266>3.0.CO;2-0},
	language = {en},
	number = {9-10},
	urldate = {2024-02-24},
	journal = {Int. J. of Robust and Nonlinear Control},
	author = {El Ghaoui, Laurent and Rami, Mustapha Ait},
	year = {1996},
	pages = {1015--1022},
}

@inproceedings{ridderhofChanceconstrainedCovarianceControl2020,
	title = {Chance-constrained covariance control for low-thrust minimum-fuel trajectory optimization},
	urldate = {2025-03-13},
	booktitle = {{AAS}/{AIAA} {Astrodynamics} {Specialist} {Conf}.},
	author = {Ridderhof, Jack and Pilipovsky, Joshua and Tsiotras, Panagiotis},
	year = {2020},
	pages = {9--13},
}

@article{oguriRobustSpacecraftGuidance2021,
	title = {Robust {Spacecraft} {Guidance} {Around} {Small} {Bodies} {Under} {Uncertainty}: {Stochastic} {Optimal} {Control} {Approach}},
	volume = {44},
	issn = {0731-5090},
	shorttitle = {Robust {Spacecraft} {Guidance} {Around} {Small} {Bodies} {Under} {Uncertainty}},
	url = {https://arc.aiaa.org/doi/10.2514/1.G005426},
	doi = {10.2514/1.G005426},
	number = {7},
	urldate = {2026-08-31},
	journal = {Journal of Guidance, Control, and Dynamics},
	publisher = {American Institute of Aeronautics and Astronautics},
	author = {Oguri, Kenshiro and McMahon, Jay W.},
	month = jul,
	year = {2021},
	pages = {1295--1313},
}

@inproceedings{sekineRobustTrajectoryOptimization2026,
	title = {Robust {Trajectory} {Optimization} {Against} {Missed} {Thrust} {Events} {Using} {Sequential} {Convex} {Programming}},
	issn = {2996-2358},
	url = {https://ieeexplore.ieee.org/document/11520058/},
	doi = {10.1109/AERO66936.2026.11520058},
	urldate = {2026-08-31},
	booktitle = {2026 {IEEE} {Aerospace} {Conference}},
	author = {Sekine, Hirotaka and Oguri, Kenshiro and Kawabata, Yosuke and Funase, Ryu},
	month = mar,
	year = {2026},
	note = {ISSN: 2996-2358},
	pages = {1--15},
}

@article{howellStationkeepingMethodLibration1993,
	title = {Station-keeping method for libration point trajectories},
	volume = {16},
	issn = {0731-5090},
	url = {https://arc.aiaa.org/doi/10.2514/3.11440},
	doi = {10.2514/3.11440},
	number = {1},
	urldate = {2026-08-31},
	journal = {Journal of Guidance, Control, and Dynamics},
	publisher = {American Institute of Aeronautics and Astronautics},
	author = {Howell, K. C. and Pernicka, H. J.},
	month = jan,
	year = {1993},
	pages = {151--159},
}

@article{oguriEQUULEUSTrajectoryDesign2020,
	title = {{EQUULEUS} {Trajectory} {Design}},
	volume = {67},
	issn = {2195-0571},
	url = {https://doi.org/10.1007/s40295-019-00206-y},
	doi = {10.1007/s40295-019-00206-y},
	language = {en},
	number = {3},
	urldate = {2026-08-31},
	journal = {The Journal of the Astronautical Sciences},
	author = {Oguri, Kenshiro and Oshima, Kenta and Campagnola, Stefano and Kakihara, Kota and Ozaki, Naoya and Baresi, Nicola and Kawakatsu, Yasuhiro and Funase, Ryu},
	month = sep,
	year = {2020},
	pages = {950--976},
}

@article{gardnerCAPSTONESummaryHighly2023,
	title = {{CAPSTONE}: {A} {Summary} of a {Highly} {Successful} {Mission} in the {Cislunar} {Environment}},
	shorttitle = {{CAPSTONE}},
	url = {https://digitalcommons.usu.edu/smallsat/2023/all2023/69},
	journal = {Small Satellite Conference},
	author = {Gardner, Tom and Cheetham, Brad and Parker, Jeff and Forsman, Alec and Kayser, Ethan and Thompson, Michael and Ott, Conner and DeMoudt, Lauren and Caudill, Michael and Bolliger, Matt and Kam, Arlen and Thompson, Keith and Rogers, Rebecca and Umansky, Hannah and Bryant, Brennan and Svitek, Tomas},
	month = aug,
	year = {2023},
}

@article{booneHaloOrbitConfigurations2026,
	title = {Halo orbit configurations in the {Saturn}-{Enceladus} system for an {Enceladus} orbiter},
	volume = {77},
	issn = {0273-1177},
	url = {https://www.sciencedirect.com/science/article/pii/S0273117726002140},
	doi = {10.1016/j.asr.2026.02.033},
	number = {9},
	urldate = {2026-08-13},
	journal = {Advances in Space Research},
	author = {Boone, Spencer and Sánchez, Joan Pau and Lizy-Destrez, Stéphanie},
	month = may,
	year = {2026},
	pages = {9230--9246},
}

@article{shimaneRevolutionSpacedOutputFeedbackModel2026,
	title = {Revolution-{Spaced} {Output}-{Feedback} {Model} {Predictive} {Control} for {Station} {Keeping} on {Near}-{Rectilinear} {Halo} {Orbits}},
	volume = {34},
	issn = {1558-0865},
	url = {https://ieeexplore.ieee.org/document/11204374/},
	doi = {10.1109/TCST.2025.3614324},
	number = {1},
	urldate = {2026-08-12},
	journal = {IEEE Transactions on Control Systems Technology},
	author = {Shimane, Yuri and Di Cairano, Stefano and Ho, Koki and Weiss, Avishai},
	month = jan,
	year = {2026},
	pages = {538--544},
}

@article{foltaEarthMoonLibration2014,
	title = {Earth–{Moon} libration point orbit stationkeeping: {Theory}, modeling, and operations},
	volume = {94},
	issn = {0094-5765},
	shorttitle = {Earth–{Moon} libration point orbit stationkeeping},
	url = {https://www.sciencedirect.com/science/article/pii/S0094576513000349},
	doi = {10.1016/j.actaastro.2013.01.022},
	number = {1},
	urldate = {2026-08-12},
	journal = {Acta Astronautica},
	author = {Folta, David C. and Pavlak, Thomas A. and Haapala, Amanda F. and Howell, Kathleen C. and Woodard, Mark A.},
	month = jan,
	year = {2014},
	pages = {421--433},
}

@article{gomezStationKeepingStrategiesTranslunar1998,
	title = {Station-{Keeping} {Strategies} {For} {Translunar} {Libration} {Point} {Orbits}},
	volume = {99},
	journal = {Advances in the Astronautical Sciences},
	author = {Gomez, Gerard and Howell, Kathie and Masdemont, Josep and Simo, Carles},
	month = jan,
	year = {1998},
}

@article{nathersonReachabililtybasedMissedThrust2026,
	title = {Reachabililty-based missed thrust analysis},
	volume = {247},
	issn = {00945765},
	url = {https://linkinghub.elsevier.com/retrieve/pii/S0094576526003413},
	doi = {10.1016/j.actaastro.2026.05.030},
	language = {en},
	urldate = {2026-07-03},
	journal = {Acta Astronautica},
	author = {Natherson, Robyn A. and Scheeres, Daniel J.},
	month = oct,
	year = {2026},
	pages = {565--578},
}

@inproceedings{oguriConvexBlockCholeskyApproach2026,
	title = {Convex {Block}-{Cholesky} {Approach} to {Risk}-{Constrained} {Low}-thrust {Trajectory} {Design} under {Operational} {Uncertainty}},
	url = {https://arxiv.org/abs/2602.18416v1},
	language = {en},
	urldate = {2026-03-29},
	booktitle = {{arXiv}.org},
	author = {Oguri, Kenshiro and Lantoine, Gregory},
	month = feb,
	year = {2026},
}

@book{durrettProbabilityTheoryExamples2019,
	title = {Probability: {Theory} and {Examples}},
	isbn = {978-1-108-59263-5},
	shorttitle = {Probability},
	language = {en},
	publisher = {Cambridge University Press},
	author = {Durrett, Rick},
	month = apr,
	year = {2019},
	note = {Google-Books-ID: vESPDwAAQBAJ},
}

@book{costaDiscreteTimeMarkovJump2005,
	title = {Discrete-{Time} {Markov} {Jump} {Linear} {Systems}},
	isbn = {978-1-85233-761-2},
	language = {en},
	publisher = {Springer Sci. \& Bus. Media},
	author = {Costa, O. L. V. and Fragoso, M. D. and Marques, R. P.},
	month = feb,
	year = {2005},
}

@incollection{starekSpacecraftAutonomyChallenges2016,
	address = {Berlin, Heidelberg},
	title = {Spacecraft {Autonomy} {Challenges} for {Next}-{Generation} {Space} {Missions}},
	volume = {460},
	isbn = {978-3-662-47693-2 978-3-662-47694-9},
	language = {en},
	urldate = {2025-01-08},
	booktitle = {Advances in {Control} {System} {Technology} for {Aerospace} {Applications}},
	publisher = {Springer Berlin Heidelberg},
	author = {Starek, Joseph A. and Açıkmeşe, Behçet and Nesnas, Issa A. and Pavone, Marco},
	editor = {Feron, Eric},
	year = {2016},
	pages = {1--48},
}

@article{spreenTrajectoryDesignTargeting2021,
	title = {Trajectory {Design} and {Targeting} for {Applications} to the {Exploration} {Program} in {Cislunar} {Space}},
	url = {https://docs.lib.purdue.edu/dissertations/AAI30504506},
	journal = {Theses and Dissertations Available from ProQuest},
	author = {Spreen, Emily M. Zimovan},
	month = jan,
	year = {2021},
	pages = {1--249},
}

@article{connorhowellThreedimensionalPeriodicHalo1984,
	title = {Three-dimensional, periodic, ‘halo’ orbits},
	volume = {32},
	issn = {1572-9478},
	url = {https://doi.org/10.1007/BF01358403},
	doi = {10.1007/BF01358403},
	language = {en},
	number = {1},
	urldate = {2024-12-23},
	journal = {Celestial mechanics},
	author = {Connor Howell, Kathleen},
	month = jan,
	year = {1984},
	pages = {53--71},
}

@inproceedings{davisORBITMAINTENANCEBURN2022,
	title = {{ORBIT} {MAINTENANCE} {BURN} {DETAILS} {FOR} {SPACECRAFT} {IN} {A} {NEAR} {RECTILIENAR} {HALO} {ORBIT}},
	language = {en},
	booktitle = {{AAS}/{AIAA} {Astrodynamics} {Specialists} {Conference}},
	author = {Davis, Diane C and Scheuerle, Stephen T and Williams, Dale A and Miguel, Frederick S and Zimovan-Spreen, Emily M and Howell, Kathleen C},
	year = {2022},
}

@inproceedings{shrivastavaFaultTolerantFeedbackControl2023,
	title = {Fault-{Tolerant} {Feedback} {Control} for {Spacecraft} {Rendezvous} using {Markov} {Jump} {Model}},
	booktitle = {{AAS}/{AIAA} {Astrodynamics} {Specialist} {Conference}},
	author = {Shrivastava, Shaurya and Oguri, Kenshiro},
	month = aug,
	year = {2023},
}

@article{fuModelReferenceTracking2017,
	title = {Model reference tracking control of continuous-time periodic linear systems with actuator jumping fault and its applications in orbit maneuvering},
	volume = {15},
	issn = {2005-4092},
	url = {https://doi.org/10.1007/s12555-016-0104-9},
	doi = {10.1007/s12555-016-0104-9},
	language = {en},
	number = {5},
	urldate = {2024-09-18},
	journal = {International Journal of Control, Automation and Systems},
	author = {Fu, Yan-Ming and Lu, Yang and Zhang, Mao-Rui},
	month = oct,
	year = {2017},
	pages = {2182--2192},
}

@inproceedings{zimovan-spreenRecoveryTraejctoriesInadvertent2021,
	address = {Virtual},
	title = {Recovery {Traejctories} for {Inadvertent} {Departures} from an {NRHO}},
	url = {https://ntrs.nasa.gov/citations/20210000237},
	urldate = {2024-09-18},
	booktitle = {{AAS}/{AIAA} {Spaceflight} {Mechanics} {Meeting}},
	author = {Zimovan-Spreen, Emily M. and Davis, Diane C. and Howell, Kathleen C.},
	year = {2021},
}

@inproceedings{mccartyMissedThrustAnalysis2020,
	title = {Missed {Thrust} {Analysis} and {Design} {For} {Low} {Thrust} {Cislunar} {Transfers}},
	booktitle = {{AAS}/{AIAA} {Astrodynamics} {Specialist} {Conference}},
	author = {McCarty, Steven and Grebow, Daniel},
	month = aug,
	year = {2020},
}

@article{desouzaLMIApproachStabilization2000,
	title = {An {LMI} approach to stabilization of linear discrete-time periodic systems},
	volume = {73},
	issn = {0020-7179},
	url = {https://doi.org/10.1080/002071700403466},
	doi = {10.1080/002071700403466},
	number = {8},
	urldate = {2024-04-26},
	journal = {Int. J. of Control},
	author = {De Souza, Carlos E. and Trofino, Alexandre},
	month = jan,
	year = {2000},
	pages = {696--703},
}

@article{suEventtriggeredSlidingMode2021,
	title = {Event-triggered sliding mode control of networked control systems with {Markovian} jump parameters},
	volume = {125},
	issn = {0005-1098},
	url = {https://www.sciencedirect.com/science/article/pii/S0005109820306075},
	doi = {10.1016/j.automatica.2020.109405},
	urldate = {2024-08-24},
	journal = {Automatica},
	author = {Su, Xiaojie and Wang, Chunlian and Chang, Hongbin and Yang, Yue and Assawinchaichote, Wudhichai},
	month = mar,
	year = {2021},
	pages = {109405},
}

@article{bittantiExtendedPeriodicLyapunov1985,
	title = {The extended periodic lyapunov lemma},
	volume = {21},
	issn = {0005-1098},
	url = {https://www.sciencedirect.com/science/article/pii/0005109885900093},
	doi = {10.1016/0005-1098(85)90009-3},
	number = {5},
	urldate = {2024-07-11},
	journal = {Automatica},
	author = {Bittanti, Sergio and Bolzern, Paolo and Colaneri, Patrizio},
	month = sep,
	year = {1985},
	pages = {603--605},
}

@misc{vegaContingencyAwareStationKeepingControl2024,
	title = {Contingency-{Aware} {Station}-{Keeping} {Control} of {Halo} {Orbits}},
	url = {http://arxiv.org/abs/2405.20458},
	language = {en},
	urldate = {2024-06-28},
	publisher = {arXiv},
	author = {Vega, Fausto and Manchester, Zachary and Lo, Martin and Restrepo, Ricardo},
	month = may,
	year = {2024},
	note = {arXiv:2405.20458 [cs, eess]},
}

@article{athanasopoulosStabilityStabilizationPeriodic2014,
	title = {On stability and stabilization of periodic discrete-time systems with an application to satellite attitude control},
	volume = {50},
	issn = {0005-1098},
	url = {https://www.sciencedirect.com/science/article/pii/S0005109814004099},
	doi = {10.1016/j.automatica.2014.10.028},
	number = {12},
	urldate = {2024-06-12},
	journal = {Automatica},
	author = {Athanasopoulos, Nikolaos and Lazar, Mircea and Böhm, Christoph and Allgöwer, Frank},
	month = dec,
	year = {2014},
	pages = {3190--3196},
}

@article{venigallaMultiObjectiveLowThrustTrajectory2022,
	title = {Multi-{Objective} {Low}-{Thrust} {Trajectory} {Optimization} with {Robustness} to {Missed} {Thrust} {Events}},
	volume = {45},
	issn = {0731-5090},
	url = {https://arc.aiaa.org/doi/10.2514/1.G006056},
	doi = {10.2514/1.G006056},
	number = {7},
	urldate = {2023-04-07},
	journal = {Journal of Guidance, Control, and Dynamics},
	publisher = {American Institute of Aeronautics and Astronautics},
	author = {Venigalla, Chandrakanth and Englander, Jacob A. and Scheeres, Daniel J.},
	month = jul,
	year = {2022},
	pages = {1255--1268},
}

@inproceedings{marsillachFailSafeSpacecraftRendezvous2021,
	title = {Fail-{Safe} {Spacecraft} {Rendezvous} on {Near}-{Rectilinear} {Halo} {Orbits}},
	issn = {2378-5861},
	doi = {10.23919/ACC50511.2021.9483328},
	booktitle = {2021 {American} {Control} {Conference} ({ACC})},
	author = {Marsillach, Daniel Aguilar and Di Cairano, Stefano and Kalabić, Uroš and Weiss, Avishai},
	month = may,
	year = {2021},
	pages = {2980--2985},
}

\end{document}